\documentclass[10pt]{amsart}
\usepackage[top=1.5in, bottom=1.0in, left=1in, right=1in]{geometry}
\usepackage[english]{babel}
\usepackage{amsmath}
\usepackage{amssymb, amscd, amsthm}
\usepackage{mathrsfs}
\usepackage{stmaryrd}
\usepackage{array}
\usepackage{cases}
\usepackage{color}
\usepackage{multicol}
\usepackage{enumerate}
\usepackage[all]{xy}
\usepackage{hyperref}
\hypersetup{linktocpage}
\usepackage{todonotes}
\usepackage{url}
\title{Filtrations of global equivariant $K$-theory}
\author{
  Markus Hausmann and Dominik Ostermayr
}
\address{University of Bonn, Germany}
\email{hausmann@math.uni-bonn.de}

\address{University of Cologne, Germany}
\email{dosterma@math.uni-koeln.de}

\newtheorem{theorem}{Theorem}[section]
\newtheorem{Theorem}[theorem]{Theorem}
\newtheorem{Cor}[theorem]{Corollary}
\newtheorem{Lemma}[theorem]{Lemma}
\newtheorem{Prop}[theorem]{Proposition}

\newtheorem*{thmx}{Theorem}
\theoremstyle{definition}
\newtheorem{Def}[theorem]{Definition}

\newtheorem{Remark}[theorem]{Remark}
\newtheorem*{remx*}{Remark}
\newtheorem{Example}[theorem]{Example}
\newtheorem*{Def*}{Definition}
\newtheorem*{Example*}{Example}

\makeatletter
\newcommand{\xRightarrow}[2][]{\ext@arrow 0359\Rightarrowfill@{#1}{#2}}
\makeatother

\newcommand{\xr}{\xrightarrow}
\newcommand{\N}{\mathbb{N}}
\newcommand{\Z}{\mathbb{Z}}
\newcommand{\Q}{\mathbb{Q}}
\newcommand{\R}{\mathbb{R}}
\newcommand{\C}{\mathbb{C}}
\newcommand{\Sph}{\mathbb{S}}
\newcommand{\F}{\mathcal{F}}
\newcommand{\upi}{\underline{\pi}}
\newcommand{\dia}{\diamond}
\newcommand{\mL}{\mathcal{L}}
\newcommand{\oL}{\overline{\mL}}
\newcommand{\U}{\mathcal{U}}
\newcommand{\oP}{\overline{\mathcal{P}^R_n}}
\DeclareMathOperator{\map}{map}
\DeclareMathOperator{\Emb}{Emb}
\DeclareMathOperator*{\colim}{colim}
\DeclareMathOperator*{\coker}{coker}
\DeclareMathOperator*{\hocolim}{hocolim}

\DeclareMathOperator*{\res}{res}
\DeclareMathOperator*{\tr}{tr}
\DeclareMathOperator*{\pr}{pr}
\DeclareMathOperator*{\Sym}{Sym}
\DeclareMathOperator*{\Rep}{Rep}
\DeclareMathOperator*{\rk}{rk}
\DeclareMathOperator*{\Ind}{Ind}
\DeclareMathOperator*{\im}{im}
\DeclareMathOperator*{\Bij}{Bij}

\DeclareMathOperator*{\Iso}{Iso}
\DeclareMathOperator*{\sgn}{sgn}

\begin{document}
\begin{abstract}
In \cite{AL07} and \cite{AL10}, Arone and Lesh constructed and studied spectrum level filtrations that interpolate between connective (topological or algebraic) $K$-theory and the Eilenberg-MacLane spectrum for the integers. In this paper we consider (global) equivariant generalizations of these filtrations and of another closely related class of filtrations, the modified rank filtrations of the $K$-theory spectra themselves. We lift Arone and Lesh's description of the filtration subquotients to the equivariant context and apply it to compute algebraic filtrations on representation rings that arise on equivariant homotopy groups. It turns out that these representation ring filtrations are considerably easier to express globally than over a fixed compact Lie group.
Furthermore, they have formal similarities to the filtration on Burnside rings induced by the symmetric products of spheres, which Schwede computed in \cite{Sch14}.
\end{abstract}
\maketitle

\tableofcontents

\setcounter{section}{-1}
\section{Introduction}
The symmetric products of spheres are a much-studied sequence of spectra interpolating between the sphere spectrum and the Eilenberg-MacLane spectrum $H\Z$. In \cite{AL07}, Arone and Lesh showed that this sequence is an example of a general construction of filtrations of the form
\[ k\mathcal{C}=A^{\mathcal{C}}_0\to A^{\mathcal{C}}_1\to \hdots \to A^{\mathcal{C}}_{\infty}\simeq H\Z \]
with $k\mathcal{C}$ denoting the spectrum associated to an ``augmented'' permutative category $\mathcal{C}$.
The symmetric products are the special case for $\mathcal{C}$ the category of finite sets. This construction in particular applies to connective topological $K$-theory and free algebraic $K$-theory of rings, yielding interesting new filtrations of their respective $0$-th Postnikov sections. Arone and Lesh further argued that these filtrations have substantial formal similarities to the symmetric products of spheres, especially in the case of topological $K$-theory. For example, the $n$-th subquotients are all suspension spectra which vanish if $n$ is not a prime power, and $K(n)$-locally the filtration converges after finitely many steps.
They further proved that while the subquotients of the symmetric product filtration are related to the layers of the Goodwillie tower of the identity, the subquotients of this new filtration are related to the layers of the Weiss tower of the functor $V\mapsto BU(V)$.

In the later paper \cite{AL10} it is shown that the filtrations of \cite{AL07} are linked to filtrations
\[ *=k\mathcal{C}^0\to k\mathcal{C}^1\to \hdots \to k\mathcal{C} \]
of the $K$-theory spectra $k\mathcal{C}$ themselves, so-called ``modified rank filtrations''. These are similar in spirit to the stable rank filtrations of algebraic $K$-theory considered by Rognes in \cite{Rog92}, but not equivalent to them in general. The modified rank filtrations come with maps to the symmetric products, and a suitable homotopy pushout gives the filtrations of \cite{AL07}, which we from now on call \emph{complexity filtrations} (based on the usage of that term in  \cite{Lesh00}). The paper \cite{AL10} also contains a study of the subquotients in the modified rank filtration, which once more turn out to be suspension spectra.

In this paper we set up and investigate equivariant versions of both the modified rank and the complexity filtration, and demonstrate further similarities to the symmetric product filtration that arise through their effect on equivariant homotopy groups. We work with the following equivariant generalizations of the spectra involved:
\begin{itemize}
	\item Topological $K$-theory is replaced by (the connective cover of) equivariant $K$-theory in the sense of Segal (\cite{Seg68b}), the $K$-theory of equivariant vector bundles. This makes sense for all compact Lie groups, though the model we use slightly differs from the actual $K$-theory spectrum for non-finite groups, as we explain in Section \ref{sec:pi0ku}.
	\item Algebraic $K$-theory of a discrete ring $R$ is replaced by a $G$-spectrum whose $H$-fixed points (for $H$ a subgroup of $G$) represent the direct sum $K$-theory of $R[H]$-modules that are finitely generated free as $R$-modules, so-called \emph{$R[H]$-lattices}. These spectra are only defined for finite groups $G$.
	\item The Eilenberg-MacLane spectrum $H\Z$ is replaced by the Eilenberg-MacLane spectrum for the constant Mackey functor $\Z$. This makes sense for all compact Lie groups, though there is again a caveat in the non-discrete case (cf. Example \ref{exa:em}).
\end{itemize}
Instead of treating each compact Lie or finite group separately, we work in a global equivariant context. The global framework packages all equivariant $K$-theory spectra for varying groups $G$ into one ``global'' object, trying to capture the full functoriality in $G$. In particular, a consequence of working in the global category is that the equivariant homotopy groups have a richer structure than one might expect. In addition to transfer maps along inclusions, they allow restriction maps along \emph{arbitrary} group homomorphisms, not only to subgroups. Concretely, the collection $\upi_0(X)=\{\pi_0^G(X)\ |\ G \text{ compact Lie} \}$ for a global spectrum $X$ forms a \emph{global functor} (or global Mackey-functor) in the sense of \cite{Sch15}. This extra functoriality is essential for the computation of the effect of the modified rank and complexity filtrations on $\upi_0$, as we explain below.
Concretely, we use orthogonal spaces (for the unstable theory) and orthogonal spectra (for the stable theory) as our model for global homotopy, as developed by Schwede in his book project \cite{Sch15}. In the case of global algebraic $K$-theory, which only forms a symmetric spectrum, we use \cite{Hau15} for the general framework and \cite{Sch13alg} for properties of this specific example.

We now go through the contents of this paper: After recalling some basics about global homotopy theory, we explain global $\Gamma$-space models for connective topological complex $K$-theory $ku$ (or its real version $ko$, over which everything works analogously) and algebraic $K$-theory $kR$ as they were described in \cite{Sch15} and \cite{Sch13alg}, respectively. These models come with a natural global generalization of the modified rank filtration introduced in \cite{AL10}, which we denote by \[*\to ku^1\to ku^2\to \hdots \to ku \hspace{1cm} \text{ resp.} \hspace{1cm} *\to kR^1\to kR^2\to \hdots \to kR.\]
We then describe the subquotients of these filtrations. For this we let $\mL_n$ denote the topological poset of proper decompositions of $\C^n$ as an orthogonal sum of subspaces, ordered by refinement. Here, ``proper'' means that the trivial decomposition into one summand is excluded. The poset carries a $U(n)$-action by applying the isometry to each summand in the decomposition. Similarly, $\mathcal{P}_n^R$ denotes the $GL_n(R)$-poset of proper decompositions of $R^n$ as a direct sum of free submodules. We show:
\begin{thmx}[Subquotients in the modified rank filtration, Theorems \ref{theo:zig-zag} and \ref{theo:alglattice}] There are global equivalences
\begin{eqnarray*} & ku^n/ku^{n-1} \simeq \Sigma^{\infty} (E_{gl}U(n)_+\wedge_{U(n)} |\mL_n|^{\dia}) \\
									& kR^n/kR^{n-1} \simeq \Sigma^{\infty} (E_{gl}GL_n(R)_+\wedge_{GL_n(R)} |\mathcal{P}_n^R|^{\dia}). \end{eqnarray*}
\end{thmx}
The underlying non-equivariant statement of this theorem is due to Arone and Lesh (cf. \cite[Section 2.2]{AL10} for the case of topological $K$-theory). The expression $\Sigma^{\infty}$ denotes the suspension spectrum of a global space (in the framework we use: an orthogonal space) and $(-)^{\dia}$ stands for the unreduced suspension. The global spaces that appear here are part of a general construction that takes a based $K$-space $X$ for some topological group $K$ and produces a global space $E_{gl}K_+\wedge_K X$, its \emph{global homotopy orbits}. Given a compact Lie group $G$, the underlying $G$-homotopy type of this construction is $E_GK_+\wedge_K X$, where $E_GK$ is a universal space for principal $K$-bundles in $G$-spaces. In particular, the underlying non-equivariant homotopy type agrees with the usual homotopy orbits. 
However, while a $K$-equivariant map that is a non-equivariant weak equivalence already induces a weak equivalence on homotopy orbits, the global version is sensitive to the isotropy of $X$ at all compact Lie subgroups of $K$. Global homotopy orbits also behave differently to their non-equivariant versions in another regard: If $K$ is a compact Lie group and $X$ a finite $K$-CW complex, the global homotopy orbits $E_{gl}K_+\wedge_K X$ are finite global spaces, though their underlying $G$-spaces often are not compact. Via the theorem above this implies that the filtration terms $ku^n$ and $kR^n$ for $R$ finite are compact in the global stable homotopy category, while they are not compact as non-equivariant spectra.

\begin{remx*}[Global Barratt-Priddy-Quillen Theorem, Theorem \ref{theo:gbpq}]Our methods can also be applied to a rank filtration $k\mathcal{F}in^1\to k\mathcal{F}in^2\to \hdots \to k\mathcal{F}in$ of the global $K$-theory of finite sets $k\mathcal{F}in$, yielding a  proof of the global Barratt-Priddy-Quillen Theorem: The subquotient $k\F in^n/k\F in^{n-1}$ is globally equivalent to $\Sigma^{\infty} ((E_{gl} \Sigma_n)_+\wedge_{\Sigma_n} |\mathcal{P}_n^{\Sigma}|^{\dia})$, where $\mathcal{P}_n^{\Sigma}$ denotes the lattice of partitions of the set $\{1,\hdots,n\}$, excluding the trivial decomposition into just one subset.
If $n$ is at least two, this lattice has a $\Sigma_n$-invariant smallest element, namely the decomposition $\{1,\hdots,n\}=\{1\}\sqcup \{2\} \sqcup \hdots \sqcup \{n\}$. Hence its nerve is $\Sigma_n$-equivariantly contractible. The first stage $k\mathcal{F}in^1$ is easily identified with the global sphere spectrum, so one obtains that the unit $\mathbb{S}\to k\mathcal{F}in$ is a global equivalence. In \cite{Rog92}, Rognes also used his version of the stable rank filtration to give a proof of the non-equivariant Barratt-Priddy-Quillen theorem.
\end{remx*}

We then proceed by considering complexity filtrations
\[ ku\simeq A_0^u\to A_1^u\to \hdots \to A_{\infty}^u\simeq H\Z \hspace{1cm} \text{ and} \hspace{1cm} kR\simeq A_0^R\to A_1^R\to \hdots \to A_{\infty}^R\simeq H\Z.\]
We define these as suitable homotopy pushouts that involve the modified rank and symmetric product filtration, generalizing the non-equivariant description of \cite{AL10}. In \cite[Corollary 8.3]{AL07} it is shown that the $n$-th subquotient of the complexity filtration can be non-equivariantly described as the suspension spectrum of a classifying space for the collection of so-called standard subgroups of $U(n)$ (resp. $GL_n(R)$). We denote this collection by $\overline{\mathcal{C}}_n$. There are natural global equivariant generalizations of classifying spaces for collections (similarly to the global homotopy orbits discussed above), which we explain in Section \ref{sec:uni} and denote by~$B_{gl}\overline{\mathcal{C}}_n$. Generalizing the non-equivariant statement in \cite{AL07}, we then show:
\begin{thmx}[Subquotients in the complexity filtration, Theorems \ref{theo:quotcomp} and  \ref{theo:algquotcomp}] There is a global equivalence
\[ A_n^u/A_{n-1}^u\simeq \Sigma^{\infty} (B_{gl}\overline{\mathcal{C}}_n^u)^{\dia} \]
and if $R$ is an integral domain with $2\neq 0$ also
\[ A_n^R/A_{n-1}^R\simeq \Sigma^{\infty} (B_{gl}\overline{\mathcal{C}}_n^R)^{\dia}. \]
\end{thmx}
The conditions on $R$ were also required in \cite{AL07}, and in fact it can be shown that the statement is false in full generality.

\begin{remx*}It might be interesting to the reader familiar with \cite{AL07} and \cite{AL10} that the methods we use to obtain these descriptions of the subquotients are quite different. While Arone and Lesh perform categorical constructions, we work with an explicit $\Gamma$-space model for connective $K$-theory and decompose it geometrically. It turns out that with this model the quotients are in fact \emph{isomorphic} to suspension spectra of global spaces. Hence the main work lies in examining the global equivariant homotopy type of those and identifying them as geometric models of classifying spaces or lattices.
\end{remx*}

Afterwards, we apply our global description of the filtration subquotients to show another formal similarity between complexity filtrations and the symmetric product filtration. For this we recall a result of Schwede \cite{Sch14}, where he considers the global version of the symmetric product filtration
\[ \mathbb{S}=Sp^1\to Sp^2\to \hdots \to Sp^{\infty}\simeq H\Z. \]
On $0$-th homotopy, the map $\mathbb{S}\to H\Z$ induces the augmentation from the Burnside ring global functor $A(-)\cong \upi_0(\mathbb{S})$ to the constant functor $\Z$, sending a finite $G$-set to its number of elements. Applying $\upi_0$ to the symmetric product filtration gives a filtration
\[ \upi_0(\mathbb{S})\to \upi_0(Sp^2)\to \hdots \to \upi_0 (H\Z)\cong \underline{\Z} \]
of this augmentation. Schwede showed that this algebraic filtration allows a compact description when considered in the global context. For this we let $\tau_n^{\Sigma}$ denote the tautological $n$-element $\Sigma_n$-set, thought of as an element in $\pi_0^{\Sigma_n}(\mathbb{S})\cong A(\Sigma_n)$.
\begin{thmx}[Schwede, {\cite[Theorem 3.13]{Sch14}}] The map $\upi_0 (\mathbb{S})\to \upi_0 (Sp^n)$ is surjective for all $n\in \N$, with kernel generated as a global functor by the single element $(\tau_n^{\Sigma}-n\cdot 1)\in \pi_0^{\Sigma_n} (\mathbb{S})$. In particular, \[ \upi_0(Sp^n)\cong A(-)/(\tau_n^{\Sigma}-n\cdot 1)\] as global functors.
\end{thmx}
The $\Sigma_n$-set $\tau_n^{\Sigma}$ is the universal $n$-element $G$-set, in the sense that for every $n$-element $G$-set $X$ there is a unique up to conjugacy group homomorphism $\alpha:G\to \Sigma_n$ such that $\alpha^*(\tau_n^{\Sigma})\cong X$. So quotiening out by $(\tau_n^{\Sigma}-n\cdot 1)$ can be loosely interpreted as forgetting all $G$-actions on sets with size at most $n$, though it is in fact more complicated, due to the presence of transfers.

In our case the complexity filtration induces a filtration of the augmentation from the representation ring global functor $\Rep(-)$ (over $\C$ in the case of $\upi_0(ku)$ resp. over $R$ in the case of $\upi_0(kR)$) to the constant functor with value $\Z$, sending a $G$-representation to its dimension or rank. There is a natural replacement for the universal $\Sigma_n$-set $\tau_n^{\Sigma}$ in this context: The $n$-th unitary group $U(n)$ acts tautologically on $\C^n$ (respectively $GL_n(R)$ on $R^n$), and every $n$-dimensional $G$-representation can be obtained by pulling this representation back along a homomorphism which is unique up to conjugacy. We let these universal representations be denoted by $\tau_n^{\C}\in \pi_0^{U(n)}(ku)$ respectively $\tau_n^{R}\in \pi_0^{GL_n(R)}(kR)$ if $R$ is finite. Then we have:
\begin{thmx}[Complexity filtration on $\upi_0$, Theorems \ref{theo:picomp} and \ref{theo:algpicomp}]\ 
\begin{enumerate}
     \item The map $\upi_0 (ku)\to \upi_0 (A_n^u)$ is surjective for all $n\in \N$, with kernel generated as a global functor by the element $(\tau_n^{\C}-n\cdot 1)\in \pi_0^{U(n)} (ku)$. In particular, \[ \upi_0(A_n^u)\cong \upi_0(ku)/(\tau_n^{\C}-n\cdot 1)\] as global functors.
    \item Let $R$ be a finite ring. Then the map $\upi_0(kR)\to \upi_0(A_n^u)$ is surjective for all $n\in \N$, with kernel generated as a global functor by the element $(\tau_n^{R}-n\cdot 1)\in \pi_0^{GL_n(R)}(kR)$. In particular,
    \[ \upi_0(A_n^R)\cong {\Rep}_R(-)/(\tau_n^{R}-n\cdot 1)\]
    as global functors.
\end{enumerate}
\end{thmx}
\begin{remx*} We emphasize that this theorem (as well as Schwede's) relies heavily on working in the global context and having restrictions along non-injective homomorphisms available. When working over a fixed finite group $G$, the kernel of $\pi_0^G(kR)\to \pi_0^G(A_n^R)$ is usually not generated by a single element (neither as an abelian group, nor as a $G$-Mackey functor) and is more complicated and less conceptual to describe.
\end{remx*}
Hence, $\upi_0(A_n^u)$ and $\upi_0(A_n^R)$ can be interpreted as the representation ring global functor modulo forgetting all group actions on vector spaces of dimension $n$ (respectively free $R$-modules of rank $n$).
This theorem reduces an explicit calculation of $\pi_0^G(A_n^u)$ or $\pi_0^G(A_n^R)$ to an algebraic exercise in representation theory, for which we give examples in Section \ref{sec:exa}. The reason why $R$ has to be finite is that otherwise the general linear groups are not finite and so are not part of the global theory. There is also a description of $\upi_0(A_n^R)$ when $R$ is not finite (Proposition \ref{prop:algpicomp2}) but it is no longer simplified by the global framework.

Moreover, we compute an algebraic description of the filtration on the representation ring itself that is induced from the modified rank filtration:
\newpage
\begin{thmx}[Modified rank filtration on $\upi_0$] The global functor $\upi_0(ku^n)$ (and similarly $\upi_0 (kR^n)$ for finite $R$) is the free global functor on the classes $\tau_1^{\C},\tau_2^{\C},\hdots,\tau_n^{\C}$ modulo finitely many universal relations that identify
\begin{itemize}
	\item homotopy-theoretic sums with direct sums of representations
	\item transfers with induction of representations
\end{itemize}
as long as the total dimension is at most $n$.
\end{thmx}
A precise formulation is given in Theorems \ref{theo:pirank} and \ref{theo:algpirank}. The proof uses an elementary examination of the fixed points of decomposition lattices of $G$-representations and the construction of an explicit geometric representative of a certain stable map, which allows us to identify its effect on $\upi_0$. We note that this is only a filtration in the sense that the colimit gives the representation ring, as the connecting maps are in general neither injective nor surjective. Again there is also a description for non-finite $R$ (Proposition \ref{prop:algpirank2}), but it is in general not finitely generated as a global functor.

In the paper we do not treat the cases of topological and algebraic $K$-theory in parallel. Because the setup is technically somewhat different for the algebraic case (symmetric versus orthogonal spectra, finite groups versus compact Lie groups), we first focus on topological $K$-theory where we give all proofs in full detail (Section \ref{sec:top}). Later in Section \ref{sec:alg} we deal with algebraic $K$-theory, leaving out the arguments that are similar to their topological version.

The appendix contains three rather technical sections. The first provides a proof that the maps relating the filtration terms are cofibrations, the second shows that the (orthogonal) spaces that appear in the filtration quotients are sufficiently cofibrant, and the third is needed to ensure that a certain square constructed in Section \ref{sec:pi0ku} is a homotopy-pushout.

\textbf{Acknowledgements}: We would like to thank Stefan Schwede for suggesting this project and for many helpful discussions and comments. Moreover, we would like to thank Greg Arone for various conversations about the content of this paper. The research was supported by the Deutsche Forschungsgemeinschaft Graduiertenkolleg 1150 ``Homotopy and Cohomology''.

\section{Global homotopy theory of orthogonal spaces and spectra} \label{sec:orththeory}
We begin by recalling the basic definitions of global equivariant homotopy theory as introduced by Schwede in \cite{Sch14} and \cite{Sch15}, starting with the unstable theory of orthogonal spaces (Section \ref{sec:orthspaces}) together with some important examples (Sections \ref{sec:uni} and \ref{sec:glhomorb}). In Section \ref{sec:orthspec} we deal with stable global homotopy theory via orthogonal spectra.

\subsection{Orthogonal spaces} \label{sec:orthspaces}

Let $L_{\R}$ be the topological category of finite dimensional real inner product spaces with linear isometric embeddings. 
\begin{Def} An \emph{orthogonal space} is a continuous functor from $L_{\R}$ to the category of spaces. 
\end{Def}
All orthogonal spaces that occur in this paper have the following additional property:
\begin{Def} An orthogonal space $X$ is \emph{closed} if for every inner product map $i:V\to W$ the structure map $X(i):X(V)\to X(W)$ is a closed embedding.
\end{Def}

Equivariance comes into play as follows: Let $G$ be a compact Lie group and $V$ an inner product space with a $G$-action through linear isometries. Then the evaluation $X(V)$ of any orthogonal space $X$ on $V$ inherits a $G$-action via the homomorphism $G\to O(V)$. Moreover, if $W$ is another $G$-representation and $i:V\to W$ is a $G$-equivariant inner product map, then the structure map $X(i):X(V)\to X(W)$ is $G$-equivariant.

The evaluation can be extended further to countably infinite dimensional $G$-representations via the formula
\[ X(W)=\colim_{f.d.V\subseteq W} X(V) \]
where the colimit is taken over all finite dimensional $G$-subrepresentations $V$ of $W$. In particular this is used to evaluate orthogonal spaces on a complete $G$-universe $\U_G$, a countably infinite dimensional representation in which every finite dimensional representation embeds.

An important invariant for a (global) orthogonal space $X$ is the $0$-th homotopy set functor, i.e., the collection of the sets  \[ \pi_0^G (X)=\colim_{f.d. V\subseteq \U_G} \pi_0( X(V)^G)\] for all compact Lie groups $G$. Given a continuous group homomorphism $\alpha:K\to G$ there is an induced restriction map $\alpha^*:\pi_0^G(X)\to \pi_0^K(X)$ constructed as follows: Every $G$-fixed point $x\in X(V)$ for some $V$ also represents a $K$-fixed point in $X(\alpha^*(V))$, where $\alpha^*(V)$ denotes the restricted representation. While $\alpha^*(V)$ might not be contained in the chosen $K$-universe $\U_K$, there exists a $K$-embedding $\alpha^*(V)\to \U_K$ which we can use to obtain an element $\alpha^*([x])$ in $\pi_0^K(X)$. In \cite[Prop. I.6.5]{Sch15} it is shown that this element does not depend on the chosen embedding and that, furthermore, inner automorphisms act as the identity.
In other words the assignment $\upi_0 (X):G\mapsto \pi_0^G (X)$ defines a functor from the opposite of the category $\Rep$ of compact Lie groups and conjugacy classes of continuous group homomorphisms to the category of sets.
\begin{Remark} If $X$ is closed, $\pi_0^G(X)$ can be naturally identified with $\pi_0(X(\U_G)^G)$, as in this case $\pi_0$ commutes with the colimit.\end{Remark}
Taking into account the equivariant evaluations of an orthogonal space leads to a notion of weak equivalence, called \emph{global equivalence}. It is easiest to state if the involved orthogonal spaces are closed: 
\begin{Def}[cf. {\cite[Proposition I.1.14]{Sch15}}] A map $f:X\to Y$ of closed orthogonal spaces is called a \emph{global equivalence} if for every compact Lie group $G$ the induced map \[ f(\U_G)^G:X(\U_G)^G\to Y(\U_G)^G \]
on $G$-fixed points is a weak equivalence of spaces.
\end{Def}
The evaluation $X(\U_G)$ should be thought of as \emph{the $G$-space underlying $X$}. In this sense a map of orthogonal spaces is a global equivalence if and only if it is an equivariant equivalence on all underlying $G$-spaces.

For general orthogonal spaces the colimit defining $X(\U_G)$ needs to be replaced by a homotopy colimit, but we do not need the general notion in this paper. Details are given in \cite{Sch15}. There it is also shown that the class of global equivalences takes part in several model structures and hence the localized homotopy category can be dealt with by methods of homotopical algebra.

\subsection{Global universal and classifying spaces of collections}
\label{sec:uni}
In this section we explain a class of examples of orthogonal spaces that is central to this paper, so-called \emph{global universal} and \emph{global classifying} spaces associated to a collection of subgroups of a Lie group $K$.

\begin{Def} Let $K$ be a Lie group. A set of closed subgroups of $K$ is called a \emph{collection} if it is closed under conjugation. 
\end{Def}
A \emph{universal space} for a collection $\mathcal{C}$ is a cofibrant $K$-space $E\mathcal{C}$ with the property that all isotropy groups lie in $\mathcal{C}$ and for every (closed) subgroup $H$ in $\mathcal{C}$ the $H$-fixed points $E\mathcal{C}^ H$ are contractible. Here and throughout the paper we say that a $K$-space is \emph{cofibrant} if it is the retract of a $K$-cell complex. Every collection possesses a universal space unique up to $K$-homotopy equivalence. The quotient of such a universal space by the $K$-action is called a \emph{classifying space of $\mathcal{C}$} and denoted $B\mathcal{C}$.

Given a collection $\mathcal{C}$ of subgroups of a Lie group $K$ together with an additional Lie group $G$, the set of closed subgroups of $K\times G$ whose intersection with $K\times 1$ lies in $\mathcal{C}$ also forms a collection, which we denote by $\mathcal{C}\langle G\rangle$.
\begin{Example} An important example is the collection $1_K$ which only contains the trivial subgroup of $K$. A subgroup of $K\times G$ lies in $1_K\langle G\rangle$ if and only if it is of the form $\{(\varphi(h),h)\ |\ h\in H\}$ for a closed subgroup $H$ of $G$ and a continuous group homomorphism $\varphi:H\to K$. 
\end{Example}

This give rise to the following global notion:

\begin{Def} Let $\mathcal{C}$ be a collection of subgroups of a Lie group $K$. A closed $K$-orthogonal space $X$ is called a \emph{global universal space} for $\mathcal{C}$ if for every compact Lie group $G$ the $(K\times G)$-space $X(\U_G)$ is a universal space for the collection $\mathcal{C}\langle G\rangle$. The quotient of a global universal space by the $K$-action is called a \emph{global classifying space of $\mathcal{C}$}.
\end{Def}
A global universal space for $\mathcal{C}$ will be denoted $E_{gl}\mathcal{C}$, a global classifying space $B_{gl}\mathcal{C}$. The following example of a global classifying space is fundamental to global equivariant homotopy theory:
\begin{Example} \label{exa:bglg}A global universal space $E_{gl}1_K$ (resp. global classifying space $B_{gl}1_K$) associated to the collection $1_K$ is called a \emph{global universal space} (resp. \emph{global classifying} \emph{space) of $K$}. We also use the notation $E_{gl}K$ (resp. $B_{gl}K$) for this global homotopy type. If $K$ is compact, a model for $E_{gl}K$ is given by the $K$-orthogonal space
\[ W \mapsto L_{\R}(V,W)
\]
for a fixed \emph{faithful} $K$-representation $V$ (cf. \cite[Section I.2]{Sch15}).

The $\Rep$-functor $\upi_0 (B_{gl}K)$ is naturally isomorphic to the one which sends a compact Lie group $G$ to the set of conjugacy classes of continuous group homomorphisms from $G$ to $K$, with functoriality through precomposition. This is proved in \cite[Thm. 1.7]{Sch14} if $K$ is compact and follows from Corollary \ref{cor:homorbquot} below in the general case. Hence, it is representable if $K$ is compact.

Global classifying spaces of compact Lie groups are the fundamental building blocks of global homotopy theory. Their suspension spectra form a class of compact generators of the triangulated stable global homotopy category (\cite[Cor. IV.3.4]{Sch15}).
\end{Example}

We will see many other examples of global universal and classifying spaces in this paper.

\subsection{Global homotopy orbits} \label{sec:glhomorb}
Let $K$ be a Lie group, $X$ a $K$-space. There is an associated orthogonal space $E_{gl}K\times_K X$, the \emph{global homotopy orbits} of $X$, defined via $(E_{gl}K\times_K X)(V)=(E_{gl}K)(V)\times_K X$. This gives rise to a large class of examples. To understand the underlying $G$-spaces of global homotopy orbits, we need the following lemma:
\begin{Lemma} \label{lem:fixedpointsquot} Let $K$ and $G$ be Lie groups of which $G$ is compact, and $Y$ be a $(K\times G)$-cell complex such that the $K$-action is free. Then there is a natural homeomorphism
\[(Y/K)^G\cong \bigsqcup_{\langle \alpha:G\to K\rangle} Y^{\Gamma(\alpha)}/C(\alpha) \]
where $\alpha$ ranges through a set of representatives of conjugacy classes of continuous group homomorphisms from $G$ to $K$, $C(\alpha)$ denotes the centralizer of the image of $\alpha$ and $\Gamma(\alpha)\subseteq K\times G$ is the graph of $\alpha$.
\end{Lemma}
\begin{proof} The statement is \cite[Prop. A.1.28]{Sch15}. There the Lie group $K$ is also required to be compact. This is not necessary, as one only needs the space of continuous group homomorphisms from $G$ to $K$ modulo conjugation to be discrete to see that the topology on the union of the $Y^{\Gamma(\alpha})/C(\alpha)$ is indeed that of a disjoint union. For this it suffices that the source $G$ is compact (\cite[Lemma 38.1]{CF64}).
\end{proof}
Applying this to $Y=X(\U_G)$ we see:
\begin{Cor} \label{cor:homorbquot} For $K$ a Lie group, $X$ a cofibrant $K$-space and $G$ a compact Lie group there is a natural homeomorphism
\[ (E_{gl} K\times_K X)(\U_G)^G \cong \bigsqcup_{\langle \alpha:G\to K \rangle} EC(\alpha)\times_{C(\alpha)} X^{\im(\alpha)}. \]
\end{Cor}
This shows that the global homotopy orbits depend on the fixed points $X^H$ for all compact subgroups $H$ of $K$, or more precisely on the functor on the orbit category of compact subgroups of $K$ that is associated to $X$. This stands in contrast to the underlying space of $E_{gl}K\times_K X$, the usual homotopy orbits, which only depend on $X$ up to non-equivariant equivalences that commute with the $K$-action.

\begin{Remark} The unstable global homotopy category is equivalent to the homotopy category of stacks, in the sense introduced in \cite{GH08} (the equivalence follows from the results in \cite[Section 4.4]{GH08} together with \cite[Theorem 8.37]{Sch15}). In this language, $E_{gl}G\times_G X$ corresponds to the quotient stack $X\sslash G$.
\end{Remark}
There is also a pointed version of global homotopy orbits, defined as $E_{gl}K_+\wedge_K X$.

\subsection{Orthogonal spectra}
\label{sec:orthspec}
We quickly give the relevant definitions of orthogonal spectra from the perspective of global equivariance, for details we refer to \cite[Chapter III]{Sch15}.

\begin{Def} An orthogonal spectrum consists of
\begin{itemize}
 \item a based $O(n)$-space $X_n$ and
 \item based structure maps $\sigma_n:X_n\wedge S^1\to X_{n+1}$
\end{itemize}
for every natural number $n$, such that all iterates $\sigma_n^m:X_n\wedge S^m\to X_{n+m}$ are $(O(n)\times O(m))$-equivariant. Here, $O(m)$ acts on $S^m$ by one-point compactification of its natural action on $\R^m$ and the $(O(n)\times O(m))$-action on $X_{n+m}$ is via pullback of the $O(n+m)$-action along the block sum embedding.
\end{Def}
A morphism $f:X\to Y$ of orthogonal spectra is a sequence of based $O(n)$-maps $f_n:X_n\to Y_n$ such that $f_{n+1}\circ \sigma^{(X)}_n=\sigma_{n+1}^{(Y)}\circ (f_n\wedge S^1)$ as maps from $X_n\wedge S^1$ to $Y_{n+1}$.

Even though not featured in the definition, an orthogonal spectrum can be evaluated on any finite dimensional real inner product space $V$ via the formula
\[ X(V)=L_{\R}(\R^n,V)_+\wedge_{O(n)}X_n, \]
where $n$ is the dimension of $V$. In addition, one can define generalized associative structure maps
\[ X(V)\wedge S^W\to X(V\oplus W) \]
for any pair of inner product spaces $V,W$ and these are $(O(V)\times O(W))$-equivariant (cf. \cite{Sch15}, in particular Remark III.1.5).

\begin{Example}[Suspension spectra] Any based orthogonal space $X$ gives rise to an orthogonal spectrum $\Sigma^{\infty} X$, its \emph{suspension spectrum}, via
\[ (\Sigma^{\infty} X)_n=X(\R^n)\wedge S^n \]
with diagonal $O(n)$-action and structure map the smash product of the map $X(i:\R^n\to \R^{n+1})$ with the homeomorphism $S^n\wedge S^1\cong S^{n+1}$. More generally, the evaluation $(\Sigma^{\infty}X)(V)$ is always naturally homeomorphic to $X(V)\wedge S^V$.

For an unbased orthogonal space $X$ we denote by $\Sigma^{\infty}_+ X$ the suspension spectrum of the based orthogonal space $X_+$. In particular, $\Sigma^{\infty}_+*$ gives the sphere spectrum $\mathbb{S}$.
\end{Example}

If $V$ comes equipped with an action of a compact Lie group $G$, the evaluation $X(V)$ also inherits a $G$-action through the isometries $L_{\R}(\R^n,V)$, just like for orthogonal spaces. Furthermore, for any other $G$-representation $W$ the structure map $X(V)\wedge S^W\to X(V\oplus W)$ becomes $G$-equivariant. Fixing a complete $G$-universe $\U_G$ for every compact Lie group $G$, one defines the equivariant homotopy groups of an orthogonal spectrum $X$ via
\[ \pi_k^G (X) = \begin{cases}  \colim_{f.d. V\subseteq U_G} [S^{k+V},X(V)]_*^G & \text{for }k\geq 0 \\
 \colim_{f.d. V\subseteq U_G}[S^V,X(\R^{-k}\oplus V)]^G_* & \text{for } k\leq 0\end{cases} \]
where the connecting maps in the colimit system are induced by the generalized equivariant structure maps.

It is an important feature of the global equivariant homotopy theory of orthogonal spectra that for each fixed $k\in \N$ the collection of homotopy groups $\underline{\pi}_k (X)=\{\pi_k^G (X)\}_{G\text{ compact Lie}}$ has a rich natural functoriality, it forms a so-called \emph{global functor}. Concretely this means that there are
\begin{itemize} 
 \item contravariantly functorial \emph{restriction maps} $\varphi^*:\pi_k^G (X)\to \pi_k^K (X)$ for every continuous group homomorphism $\varphi:K\to G$, and
 \item covariantly functorial \emph{transfer maps} $\tr_H^G:\pi_k^H (X)\to \pi_k^G (X)$ for every closed subgroup inclusion $H\subseteq G$.
\end{itemize}
The restrictions are defined in the same manner as for orthogonal spaces in Section \ref{sec:orthspaces}. If $\varphi$ is the inclusion of a closed subgroup $H$ of $G$, we also use the notation ${\res}_H^G$ instead of $\varphi^*$.

We quickly recall the construction of the transfer in the case where $H$ is of finite index in $G$, as we need it explicitly in Section \ref{sec:pi0ku}:
\begin{Example}[Finite index transfers] Let $X$ be an orthogonal spectrum and $x\in \pi_0^H (X)$ (the construction for other degrees is similar). The element $x$ is represented by an $H$-map $f:S^V\to X(V)$ for some $H$-representation $V$, which we can without loss of generality assume to be the restriction of a $G$-representation which also allows an embedding of the $G$-set $G/H$. This embedding can be extended to a $G$-equivariant embedding $G/H\times D(V)\to V$ (where $D(V)$ denotes the unit disc in $V$). Collapsing everything outside the interiors of the discs to a point (the ``Thom-Pontryagin construction'') gives a $G$-map $S^V\to G/H_+ \wedge S^V$, from which one obtains a representative for the transfer $\tr_H^G(x)$ of $x$ by postcomposing with the map $G/H_+\wedge S^V\to S^V$ sending a tuple $([g]\wedge v)$ to $gf(g^{-1}v)$.
\end{Example}

Restrictions and transfers satisfy the double coset formula (cf. \cite[III.4.15]{Sch15}). Furthermore, like for orthogonal spaces, the restriction along an inner automorphism of a compact Lie group is always the identity, and transfers along subgroup inclusions whose Weyl group is infinite are always zero.

For every orthogonal space $X$, the group $\pi_0^G(\Sigma^{\infty}_+X)$ can be expressed in terms of the set $\pi_0^G(X)$. Every element $[x]$ of $\pi_0^G (X)$ is represented by a point $x\in X(V)^G$ for some $G$-representation $V$ contained in the chosen $G$-universe $\U_G$ and gives rise to an element (also denoted $[x]$) in $\pi_0^G (\Sigma^{\infty}_+ X)$ represented by the $G$-map $S^V\to X(V)_+\wedge S^V,v\mapsto x\wedge v$. This construction commutes with restrictions along group homomorphisms. Based on the tom Dieck-splitting, Schwede shows:

\begin{Prop}[{\cite[Prop. III.4.8]{Sch15}}] \label{prop:pi0susp} Let $G$ be a compact Lie group and $X$ an orthogonal space. Then the $0$-th homotopy group $\pi_0^G(\Sigma^{\infty}_+X)$  is free with basis $\{\tr_H^G([x])\}$, where $H$ ranges through conjugacy classes of subgroups of $G$ with finite Weyl group $W_GH=N_GH/H$ and $[x]$ ranges through a set of representatives of $W_GH$-orbits of $\pi_0^H(X)$.
\end{Prop}

\begin{Example} \label{exa:stablebglg} Applying this to $X=B_{gl}K$ as in Example \ref{exa:bglg} we see that $\pi_0^G(\Sigma^{\infty}_+ (B_{gl}K))$ has a basis $\{\tr_H^G([\alpha])\}$, where $(H,\alpha:H\to K)$ ranges through $G$-conjugacy classes of pairs of a subgroup of $G$ together with a (continuous) group homomorphism to $K$. If $K$ is compact, we call $\upi_0(\Sigma^{\infty}_+ (B_{gl} K))$ the \emph{free global functor in degree $K$}. It is representable if one interprets global functors as functors on the global Burnside category, cf. \cite[Section III.4]{Sch15}).
\end{Example}
We give one more example, as it soon becomes relevant:

\begin{Example}[Global Eilenberg-MacLane spectrum] \label{exa:em} The \emph{global Eilenberg-MacLane spectrum} $H\Z$ of the integers is given by $(H\Z)_n=\widetilde{\Z}[S^n]$, the reduced linearization of $S^n$. The structure map $\widetilde{\Z}[S^n]\wedge S^1\to \Z[S^{n+1}]$ sends a pair $(\sum a_i x_i)\wedge y$ to $\sum a_i (x_i\wedge y)$. For $G$ a finite group there is an isomorphism \[ \pi_k^G (H\Z) \cong \begin{cases} \Z & k=0 \\ 0 & \text{else} \end{cases}. \]
Under this isomorphism all restrictions are the identity and transfers are given by multiplication with the index (cf. \cite{dS03}). In other words, on finite groups  $H\Z$ is a model for an Eilenberg-MacLane spectrum for the constant global functor with value $\Z$. This is not true for compact Lie groups. Neither is $\pi_0^G (H\Z)$ isomorphic to $\Z$ in general, nor do higher homotopy groups vanish (cf. \cite[Proposition V.3.21]{Sch15}).
\end{Example}
Equivariant homotopy groups also give rise to the notion of global equivalence:

\begin{Def} A morphism $f:X\to Y$ of orthogonal spectra is a \emph{global equivalence} if it induces an isomorphism on $\pi_k^G$ for all integers $k$ and every compact Lie group $G$.
\end{Def}
The localization of the category of orthogonal spectra at the class of global equivalences gives rise to the \emph{global stable homotopy category} $\mathcal{GH}$. Again there are several model structures available for examining this homotopy category (cf. \cite[Section IV.1]{Sch15}).

\section{The symmetric product filtration/rank filtration of $H\Z$}
We now begin studying modified rank and complexity filtrations. The first example deals with the symmetric product filtration which interpolates between the (global) sphere spectrum and the (global) Eilenberg-MacLane spectrum for the integers $\Z$. It has been much studied non-equivariantly, in particular the homotopy type of the filtration quotients has been determined by Lesh in \cite{Lesh00}. As explained in the introduction, its (global) equivariant properties were examined by Schwede in \cite{Sch14}. Everything we write in this section can be found there. Nevertheless we repeat the argument here, because it is the basic example of a rank filtration and because we need it later to introduce and study complexity filtrations.

Let $n$ be a natural number. The $n$-th symmetric product $Sp^n$ of the sphere spectrum is given by
\[ Sp^ n_k=(S^ k)^ {\times n}/{\Sigma_n}. \]
There are inclusions $Sp^ {n-1}\to Sp^ n$ inserting a basepoint in the last component, giving rise to a filtration that starts with the global sphere spectrum $\Sph=Sp^1$ and converges to $Sp^ {\infty}$, which is globally equivalent to the global Eilenberg-MacLane spectrum $H\Z$ of Example \ref{exa:em} by \cite[Proposition V.3.29]{Sch15}.

The quotients in this filtration are given by 
\[(Sp^ n/Sp^ {n-1})_k=(S^ k)^ {\wedge n}/{\Sigma_n}, \]
or in other words the $\Sigma_n$-orbit space of the one-point compactification of the $(\Sigma_n\times O(k))$-representation $\R^ {{n}}\otimes \R^ k$. The natural $\Sigma_n$-representation $\R^ {{n}}$ decomposes as
\[ \R^ {{n}}=\overline{\R^ {{n}}}\oplus \R, \]
where $\overline{\R^ {{n}}}$ is the reduced natural representation of vectors whose entries sum to zero and $\mathbb{R}$ the trivial diagonal copy. Using this decomposition, we see that
\begin{align*} (Sp^ n/Sp^ {n-1})_k 	&\cong S^ {\overline{\R^ {{n}}}\otimes \R^ k}/{\Sigma_n}\wedge S^ k
											\cong S(\overline{\R^ {{n}}}\otimes \R^ k)^\dia/{\Sigma_n}\wedge S^ k.
\end{align*}
Here, $S(-)$ denotes the sphere of vectors of length one in an inner product space and the notation $X^\dia$ stands for the unreduced suspension of a space $X$ equipped with the basepoint $X\times 1$. Moreover, under this homeomorphism the structure map $(Sp^ n/Sp^ {n-1})_k\wedge S^ 1\to (Sp^ n/Sp^ {n-1})_{k+1}$ corresponds to the smash product of the inclusion $S(\overline{\R^ {{n}}}\otimes \R^ k)^\dia/\Sigma_n\to S(\overline{\R^ {{n}}}\otimes \R^ {k+1})^\dia/\Sigma_n$ and the homeomorphism $S^ k\wedge S^ 1\cong S^ {k+1}$. In other words, $Sp^ n/Sp^ {n-1}$ is isomorphic (!) to the suspension spectrum of the based orthogonal space
\[ V\mapsto S(\overline{\R^ {{n}}}\otimes V)^\dia/\Sigma_n. \]

Let $\mathcal{T}_n$ denote the collection of \emph{non-transitive subgroups} of $\Sigma_n$, i.e., those subgroups whose tautological action on the set $\underline{n}$ is not transitive. Further we denote by $\mathcal{C}^{\Sigma}_n$ the collection of \emph{complete subgroups} of $\Sigma_n$, i.e., those conjugate to one of the form $\Sigma_{n_1}\times \Sigma_{n_2}\times \hdots \times \Sigma_{n_k}$ with $n_1+n_2+\hdots + n_k=n$, all $n_i\geq 1$ and $k>1$. We note that a subgroup of $\Sigma_n$ is non-transitive if and only if it is contained in a complete subgroup.

\begin{Prop} \label{prop:sympow} The $\Sigma_n$-orthogonal space $S(\overline{\R^ {{n}}}\otimes - )$ is a global universal space for both $\mathcal{T}_n$ and $\mathcal{C}^{\Sigma}_n$.
\end{Prop}
The notion of a global universal space for a collection is explained in Section \ref{sec:uni}.
\begin{proof} Clearly all structure maps $S(\overline{\R^ {{n}}}\otimes V)\to S(\overline{\R^ {{n}}}\otimes W)$ are closed. Now let $G$ be a compact Lie group and $\U_G$ a complete $G$-universe. As a consequence of a theorem of Illman (cf. \cite{Ill83}), all $(\Sigma_n\times G)$-spheres $S(\overline{\R^ {{n}}}\otimes V)$ are $(\Sigma_n\times G)$-cofibrant.

We are now going to show that all $\Sigma_n$-isotropy of $S(\overline{\R^ {{n}}}\otimes \U_G)$ lies in complete subgroups and that the fixed points for a subgroup $H$ of $\Sigma_n\times G$ are contractible whenever the intersection $H\cap (\Sigma_n\times 1)$ is non-transitive, implying universality for both collections. An element of $\overline{\R^ {n}}\otimes \U_G$ can be represented by an $n$-tuple $(v_1,\hdots, v_n)$ of vectors of $\U_G$ who sum up to zero, with $\Sigma_n$ acting by permuting the coordinates. Every such element defines a partition of the set $\underline{n}$ by the equivalence relation that $i\sim j$ if $v_i=v_j$.
Let the equivalence classes be denoted by $A_1,\hdots , A_k$. Then a permutation in $\Sigma_n$ fixes the element $(v_1,\hdots, v_n)$ if and only if it maps each $A_i$ into itself, i.e., if and only if it lies in the subgroup $\Sigma(A_1)\times \hdots \times \Sigma(A_k)$. Since the $v_i$'s add up to zero and are of total length one, they cannot all be the same and hence $k$ is greater than $1$ and the isotropy subgroup is complete.

Now let $H$ be an element of $\mathcal{C}^\Sigma_n\langle G\rangle$, i.e., a subgroup of $\Sigma_n\times G$ whose intersection $K:=H\cap (\Sigma_n\times 1)$ acts non-transitively. There is a short exact sequence of groups
\[ 1\to K\to H\to {\pr} _G(H) \to 1 \]
and hence the $H$-fixed points of $S(\overline{\R^ n}\otimes \U_G)$ equal the $\pr_G(H)$-fixed points of
\[ S(\overline{\R^ n}\otimes \U_G)^ K=S((\overline{\R^ n})^K\otimes \U_G)=S(\overline{\R^ {n/{K}}}\otimes \U_G).\]
The $\pr_G(H)$-action on the latter representation is the tensor product of the action on $\overline{\R^ {n/K}}$ induced from the short exact sequence above and the restricted action on $\U_G$. Since $\U_G$ is a complete $\pr_G(H)$-universe, it in particular contains an infinite direct sum of copies of the dual of $\overline{\R^ {n/{K}}}$ (which is again isomorphic to $\overline{\R^ {n/{K}}}$).
The tensor product of any finite dimensional representation with its dual always contains a trivial representation, so it follows that the $\pr_H(G)$-fixed points of $S(\overline{\R^ {n/K}}\otimes \U_G)$ are a unit sphere in an infinite dimensional vector space and hence contractible, and so we are done.
\end{proof}
So one obtains:
\begin{Theorem}[{\cite[Prop. 1.11]{Sch14}}] The quotient $Sp^n/Sp^{n-1}$ is globally equivalent (in fact, isomorphic) to the suspension spectrum of the unreduced suspension of a global classifying space for both the collection of non-transitive and complete subgroups of $\Sigma_n$, or in short:
\[ Sp^n/Sp^{n-1} \simeq \Sigma^{\infty}(B_{gl}\mathcal{T}_n)^\dia\simeq \Sigma^{\infty}(B_{gl}\mathcal{C}_n^ {\Sigma})^\dia \]
\end{Theorem}

\section{Filtrations associated to global topological $K$-theory} \label{sec:top}
In this section we introduce and study global equivariant versions of the modified rank filtration and complexity filtration associated to connective complex $K$-theory $ku$. Connective global $K$-theory is a global equivariant version of connective $K$-theory in the sense that it (at least for finite groups $G$) is the connective cover of a spectrum assembling all $G$-equivariant periodic $K$-theories into one global object.
For non-finite compact Lie groups this is not quite true, as is explained in \cite[Remark V.6.14 and Example V.6.15]{Sch15} and will also be dealt with further in Section \ref{sec:pi0ku}. The reason is that the model we use is based on equivariant $\Gamma$-spaces, the theory of which is not as well-behaved over non-finite groups. We note also that this is \emph{not} equivariant connective $K$-theory in the sense of Greenlees \cite{Gre04} (even though there does exist a global version of this, too, see \cite[V.6.43]{Sch15}), which is in fact not an equivariantly connective spectrum.

\subsection{Quotients in the modified rank filtration} \label{sec:kurank}
We quickly recall the construction of connective global $K$-theory as it is given in \cite[V.6.1]{Sch15}, together with the modified rank filtration. For a complex inner product space $W$ and a finite based set $A_+$ we define $k(W,A_+)$ to be the space of tuples $(W_a)_{a\in A}$ of finite dimensional pairwise orthogonal subspaces of $W$ indexed on $A$, or in other words the space
\[ \bigsqcup_{(n_a\in \N)_{a\in A}}L_\C(\bigoplus_{a\in A} \C^{n_a},W)/{\prod_{a\in A} U(n_a)}. \]
It comes with a filtration \[*=k^0(W,A_+)\subseteq k^ 1(W,A_+)\subseteq \hdots \subseteq k^ n(W,A_+)\subseteq \hdots \subseteq k(W,A_+)\] by restricting to those tuples whose dimensions add up to at most a fixed number $n$.
   
For a real inner product space $V$ we denote by $\Sym(\C\otimes V)$ the symmetric algebra on the complexification $\C\otimes V$, equipped with an inner product structure as described in \cite[II.7.16]{Sch15}. Then the assignment
\[ (V,A)\mapsto k(\Sym(\C\otimes V),A_+) \]
forms an orthogonal $\Gamma$-space (i.e., a functor from $\Gamma^{op}$ to orthogonal spaces) via direct sum of subspaces. The connective global $K$-theory spectrum $ku$ is defined as the realization of this orthogonal $\Gamma$-space, i.e., \[ ku(V)=k(\Sym(\C\otimes V),S^ V).\]
The $O(V)$-action on $ku(V)$ is the diagonal one through $\Sym(\C\otimes V)$ and $S^V$. Furthermore, the filtration $k^ n(\Sym(\C\otimes V),A_+)$ is compatible with the orthogonal $\Gamma$-space structure and hence gives rise to a filtration of orthogonal $\Gamma$-spaces and finally to the modified rank filtration of orthogonal spectra
\[ *=ku^ 0\xr{i_0} ku^ 1 \xr{i_1} \hdots \xr{i_{n-1}} ku^ n\xr{i_n} \hdots \to ku. \] 

\begin{Remark} The tensor product of complex vector subspaces turns $ku$ into a strictly commutative monoid for the smash product of orthogonal spectra, an \emph{ultracommutative ring spectrum} in the language of \cite[Chapter V]{Sch15}. This uses the natural isometry $\Sym(V)\otimes \Sym(W)\cong \Sym(V\oplus W)$, the main reason for the appearance of the symmetric algebra in the definition for $ku$. The multiplication is compatible with the rank filtration in the sense that it restricts to maps $ku^m\wedge ku^n\to ku^{m\cdot n}$. In particular, each $ku^n$ and also the filtration quotients $ku^n/ku^{n-1}$ become modules over the ultracommutative ring spectrum $ku^1$, which is the suspension spectrum of a global classifying space for $U(1)$. The description of the filtration quotients we give below could be stated in the category of $ku^1$-modules, but we refrain from doing so to ease the exposition.
\end{Remark}

In the following we determine the filtration quotients $ku^n/ku^{n-1}$ of this rank filtration. We argue in Appendix \ref{app:inclusions} that the maps $ku^{n-1}\to ku^n$ are levelwise equivariant cofibrations, so the strict quotient indeed has the homotopy type of the homotopy cofiber. The first step consists of rewriting the quotient $\Gamma$-spaces in a slightly different way. For a complex vector space $W$ and a finite set $A$ the space $k^ n(W,A_+)/k^ {n-1}(W,A_+)$ is given by
\[ \bigvee_{(n_a),\Sigma n_a=n}(L_\C(\bigoplus_{a\in A} \C^{n_a},W)/{\prod _{a\in A}U(n_a)})_+. \]
Using that composition defines a homeomorphism
\[L_\C(\C^ n,W)_+\wedge _{U(n)}L_\C(\bigoplus_{a\in A} \C^{n_a},\C^ n)_+\xr{\cong} L_\C(\bigoplus_{a\in A} \C^{n_a},W)_+ \]
we obtain that this quotient is isomorphic to 
\begin{equation*} \label{eq:split} L_\C(\C^ n,W)_+\wedge _{U(n)}(\bigvee_{(n_a),\Sigma n_a=n}(L_\C(\bigoplus_{a\in A} \C^{n_a},\C^ n)/{\prod_{a\in A} U(n_a)})_+). \end{equation*}
Applying this to $W=\Sym(\C\otimes V)$ we see that we have rewritten the orthogonal $\Gamma$-space $k^n(\Sym(\C\otimes -),-)/k^{n-1}(\Sym(\C\otimes -),-)$ as a balanced smash product of two parts. The first -- $L_\C(\C^ n,\Sym(\C\otimes -))_+$-- is constant in the $\Gamma$-space direction, while the second is constant in the orthogonal space direction. This second factor is the $U(n)$-$\Gamma$-space of decompositions of $\C^ n$ into a direct sum of orthogonal sub-vector spaces. We give it the shorter notation
\[\mathcal{L}(n,A_+) = (\bigvee_{(n_a),\Sigma n_a=n}(L_\C(\bigoplus_{a\in A} \C^{n_a},\C^ n)/{\prod_{a\in A} U(n_a)})_+). \]

Hence we can treat the two parts separately. The first is easy:
\begin{Lemma} The $U(n)$-orthogonal space $L_\C(\C^n,\Sym(\C\otimes -))$ is a global universal space for $U(n)$.
\end{Lemma}
\begin{proof} This is merely the complex version of Example \ref{exa:bglg}, since $\Sym(\C\otimes \U_G)$ is a complete complex $G$-universe if $\U_G$ is a complete real $G$-universe.
\end{proof}

Now we concentrate on the second part and consider the evaluation of the $U(n)$-$\Gamma$-space $\mathcal{L}(n,-)$ on a representation sphere $S^ V$ for some compact Lie group $G$ and $G$-representation $V$. Every element is represented by a tuple $(W_i,x_i)_{i\in I}$ for some finite indexing set $I$, where the $W_i$ form an orthogonal decomposition of $\C^n$ into complex subspaces and the $x_i$ are elements of $S^ V$. (This representative becomes unique up to a change of labels if we require all the $x_i$ to be distinct elements of $V$ and all the $W_i$ to be non-zero.) One can imagine the vector space $W_i$ sitting on the point $x_i$ and the topology is such that if two points approach one another the vector spaces sitting on them are added. The action of $U(n)$ is through the partition, that of $G$ through the points $x_i$.

Sitting inside $\mathcal{L}(n,S^V)$ we have a copy of $S^V$ (with trivial $U(n)$-action) as elements of the form $(\C^ n,x)$. In fact we can mimic the construction for the symmetric products to see that $S^ V$ splits off $(U(n)\times G)$-equivariantly as a smash factor. The other factor is given by the subspace of $\mathcal{L}(n,S^ V)$ consisting of elements represented by tuples $(W_i,x_i)_{i\in I}$ (with $x_i\in V$) that satisfy the relation \[\sum_{i\in I} (\dim(W_i)\cdot x_i)=0\]
as elements of $V$ (this property is independent of the representing tuple). We denote this ``reduced'' subspace by $\overline{\mathcal{L}(n,S^ V)}$. Then the $(U(n)\times G)$-homeomorphism
\[ \mathcal{L}(n,S^ V)\xr{\cong} \overline{\mathcal{L}(n,S^ V)}\wedge S^ V \]
is given by $[(W_i,x_i)_{i\in I}]\mapsto [(W_i,x_i-x)]\wedge x$ where $x=\frac{1}{n}\sum_{i\in I} (\dim(W_i)\cdot x_i)$. 

Under this identification, for another $G$-representation $V'$ the structure map
\[\overline{\mathcal{L}(n,S^ V)}\wedge S^ V\wedge S^ {V'}\to \overline{\mathcal{L}(n,S^ {V\oplus V'})}\wedge S^ {V\oplus V'} \]becomes the smash product of the canonical homeomorphism $S^ V\wedge S^ {V'}\cong S^ {V\oplus V'}$ with the inclusion $\overline{\mathcal{L}(n,S^ V)}\to \overline{\mathcal{L}(n,S^ {V\oplus V'})}$ induced from the inclusion of $V$ into $V\oplus V'$. Thus we see that the orthogonal spectrum realization of $\mathcal{L}(n,-)$ (with induced $U(n)$-action) is isomorphic to the suspension spectrum of the based $U(n)$-orthogonal space sending $V$ to $\overline{\mathcal{L}(n,S^ V)}$.

Finally, like in the symmetric product filtration, this based orthogonal space is itself the unreduced suspension of the subspace of ``norm 1 elements": Let $\overline{\mathcal{L}_{|.|=1}(n,S^V)}$ be the subspace of $\overline{\mathcal{L}(n,S^ V)}$ consisting of those elements that are represented by a tuple $(W_i,x_i)_{i\in I}$ (with $x_i\in V$) satisfying the relation $\sum_{i=1}^ m (\dim(W_i) |x_i|^2)=1$. Then there is a $(U(n)\times G)$-homeomorphism
\begin{align*}  \overline{\mathcal{L}(n,S^ V)}&\to\overline{\mathcal{L}_{|.|=1}(n,S^V)}^\dia \\
		  [(W_i,x_i)_{i\in I}]& \mapsto ([(W_i,x_i/|x|)_{i\in I}],|x|/(1+|x|)),
\end{align*}
where $|x|=\sqrt{\sum_{i=1}^ m \dim(W_i)|x_i|^2}$ and the image of the basepoint is understood to be the endpoint at $1$. Hence we obtain:

\begin{Cor} \label{cor:kuquot} The quotient $ku^ n/ku^ {n-1}$ is isomorphic to the suspension spectrum of the based orthogonal space \[ L_\C(\C^n,\Sym(\C\otimes -))_+\wedge_{U(n)}(\overline{\mathcal{L}_{|.|=1}(n,S^-)}^\dia)\] 
\end{Cor}
It remains to determine the global homotopy type of the $U(n)$-orthogonal space $\overline{\mathcal{L}_{|.|=1}(n,S^-)}$, which we from now on abbreviate by $\oL_n$. We introduce two collections of subgroups of $U(n)$:
\begin{Def} A subgroup of $U(n)$ is called
\begin{itemize}
 \item \emph{complete} if it is conjugate to one of the form $U(n_1)\times \hdots \times U(n_k)$ with each $n_i$ positive, $n_1+\hdots + n_k=n$ and $k>1$.
  \item \emph{non-isotypical} if its tautological action on $\C^ n$ is not isotypical, i.e., if $\C^n$ is not the direct sum of isomorphic copies of one irreducible representation.
\end{itemize}
The collection of complete subgroups is denoted by $\mathcal{C}_n^u$, that of non-isotypical subgroups by $\mathcal{I}^u_n$.
\end{Def}
We note that to every (unordered) decomposition $\C^ n=\bigoplus_{i\in I} W_i$ into at least two pairwise orthogonal non-trivial subspaces we can associate a complete subgroup $\prod_{i\in I} U(W_i)$ of $U(n)$. This assignment is bijective, the inverse maps a complete subgroup to the decomposition of $\C^ n$ into the isotypical components of its action. Note also that every complete subgroup is non-isotypical. Then we have:

\begin{Prop} \label{prop:clacompl} The $U(n)$-orthogonal space $\oL_n$ is a global universal space for both $\mathcal{C}_n^u$ and $\mathcal{I}^u_n$.
\end{Prop}
\begin{proof} Let $G$ be a compact Lie group and $\U_G$ a complete $G$-universe. In Appendix \ref{app:cw} it is proved that $\oL_n(\U_G)$ is $(U(n)\times G)$-cofibrant. We now show that all $U(n)$-isotropy of $\oL_n(\U_G)$ lies in complete subgroups and that the $H$-fixed points are contractible whenever $H$ lies in $\mathcal{I}^u_n\langle G\rangle$. Since complete subgroups are non-isotypical, this implies universality for both collections. Any point $x$ in $\oL_n(\U_G)$ is represented by a tuple $(W_i,x_i)_{i\in I}$ satisfying the relations $\sum_{i\in I} \dim(W_i)\cdot x_i=0$ and $\sum_{i\in I} \dim(W_i)|x_i|^ 2=1$. Without loss of generality we can assume that all the $x_i$ are distinct.
Since an element $\varphi$ of $U(n)$ only acts through the $W_i$ and the presentation of $x$ as such a tuple is unique up to a permutation, $\varphi$ fixes $x$ if and only if it fixes each of the $W_i$. In other words, the isotropy group of $x$ is the product $\prod_{i\in I} U(W_i)$. The two relations force $|I|$ to be larger than $1$ (the only element would have to be zero by the ``reduced'' condition, contradicting that the tuple has norm $1$) and hence this product is complete.

We move on to show that the relevant fixed point spaces are contractible. First let $K\subseteq U(n)$ be any subgroup and denote by $W_1,W_2,\hdots ,W_k$ the isotypical components of its action on $\C^n$. Since every $K$-representation decomposes canonically into isotypical subrepresentations, we see that
\[ \mathcal{L}(\C^n,S^ {\U_G})^ K \cong \bigwedge_{i=1}^k\mathcal{L}(W_i,S^ {\U_G})^K. \]
Here, the notation $\mathcal{L}(W_i,S^{\U_G})$ is used to denote the evaluation of the $\Gamma$-space of decompositions of the complex vector space $W_i$ on $\U_G$, i.e., $\mathcal{L}(n,S^{\U_G})$ with $\C^n$ replaced by $W_i$.

We can perform the manipulations of this section to each smash factor separately and obtain $k$ smash copies of $S^{\U_G}$, of which the diagonal corresponds to the one of $\mathcal{L}(\C^n,S^{\U_G})$ used as the suspension spectrum coordinate. Hence we have:
\[  \overline{\mathcal{L}(\C^n,S^{\U_G})}^K\cong (S^{\overline{\R^k}\otimes {\U_G}})\wedge \bigwedge_{i=1}^k \overline{\mathcal{L}(W_i,S^{\U_G})}^K \]
Finally we make use of the fact that a smash product of unreduced suspensions is (based) homeomorphic to the unreduced suspension of the join (denoted by $-*-$) and obtain
\[ \oL_n(\U_G)^K\cong S(\overline{\R^k}\otimes {\U_G})^K*\overline{\mathcal{L}_{|.|=1}(W_1,S^{\U_G})}^K*\hdots * \overline{\mathcal{L}_{|.|=1}(W_K,S^{\U_G})}^K. \]

Now let $H$ be a subgroup of $U(n)\times G$ such that $K:= H\cap (U(n)\times 1)$ acts non-isotypically. We have to show that the $H$-fixed points of $\oL_n(\U_G)$ are contractible. Again we make use of the short exact sequence
\[ 1\to K\to H\to {\pr} _G(H)\to 1 \]
to write these $H$-fixed points as the $\pr_G(H)$-fixed points of $\oL_n(\U_G)^K$. But by the homeomorphism above, these are given ($\pr_G(H)$-equivariantly) by the join of $S(\overline{\R^k}\otimes {\U_G})$ with another space. Here, the $\pr_G(H)$-action on $S(\overline{\R^k}\otimes {\U_G})$ comes from the fact that any element of $U(n)$ which normalizes $H$ permutes its isotypical components and hence acts on the set $\underline{k}$. But we have seen in the proof of Proposition \ref{prop:sympow} that the $\pr_G(H)$-fixed points of $S(\overline{\R^k}\otimes \U_G)$ under such an action are contractible if $k>1$ and hence so is the join and we are done.
\end{proof}

Putting everything together:
\begin{Theorem}[Subquotients in the modified rank filtration] There are global equivalences
\begin{align*}  ku^n/ku^ {n-1}& \simeq \Sigma^ {\infty} (E_{gl}U(n)_+\wedge_{U(n)} (E_{gl}\mathcal{C}^u_n)^\dia)  \simeq \Sigma^ {\infty} (E_{gl}U(n)_+\wedge_{U(n)} (E_{gl}\mathcal{I}^u_n)^\dia). \end{align*}
\end{Theorem}
As explained in the introduction, the underlying non-equivariant statement of this theorem is due to Arone and Lesh \cite[Section 2.2]{AL10}.

\subsection{Equivalence to decomposition lattice} \label{sec:lattice}
In this section we show that the description of the filtration quotients in the modified rank filtration can be further simplified. The $U(n)$-orthogonal space $\oL_n$ can be replaced by an actual $U(n)$-space, the nerve of the lattice $\mL_n$ of non-trivial orthogonal sum decompositions of $\C^n$. This implies (see Proposition \ref{theo:zig-zag} below) that $ku^n/ku^{n-1}$ is globally equivalent to the suspension spectrum of the global homotopy orbits of this lattice, in the sense of Section \ref{sec:glhomorb}.
\begin{Def}[Decomposition lattice] Let $\mL_n$ be the topological lattice of orthogonal decompositions $\C^n=\bigoplus _{i\in I} W_i$ with $|I|>1$ and all $W_i\neq 0$ (modulo bijections of indexing sets $I$), ordered by refinement. Concretely, a decomposition $\bigoplus _{i\in I} W_i$ is smaller than or equal to $\bigoplus _{j\in J} W'_j$ if for every $i\in I$ there exists a $j\in J$ such that $W_i\subseteq W'_j$.
\end{Def}
We think of $\mL_n$ as a topological category with a $U(n)$-action given by \[ \varphi\cdot (\bigoplus _{i\in I} W_i)=\bigoplus _{i\in I} \varphi(W_i).\] The topology on both the objects and the morphisms of $\mL_n$ is the weakest topology such that the $U(n)$-action becomes continuous, i.e., as the disjoint union of its $U(n)$-orbits. We denote the geometric realization of the (topological) nerve of $\mL_n$ by $|\mL_n|$.

As written above, our aim is to show the following:
\begin{Theorem} \label{theo:zig-zag} There is a zig-zag of $U(n)$-maps between $\oL_n$ and the constant orthogonal space $|\mL_n|$ inducing a global equivalence \[ S^1\wedge (E_{gl}U(n)\times_{U(n)}\oL_n)_+\simeq S^1\wedge(E_{gl}U(n)\times_{U(n)} |\mL_n|)_+. \] In particular, there is a global equivalence
\[ ku^n/ku^{n-1}\simeq \Sigma^{\infty} (E_{gl}U(n)_+\wedge_{U(n)}|\mL_n|^{\dia}). \] 
\end{Theorem}

To compare $|\mL_n|$ with $\oL_n$ we define an intermediate $U(n)$-orthogonal space $Z_n$, restricting in some sense to the ``regular'' elements of $\oL_n$. Let $F_n:\mathcal{L}_n\to \{\text{orthogonal spaces}\}$ be the continuous functor which assigns to a decomposition $\C^n=\bigoplus_{i\in I} W_i$ and an inner product space $V$ the subspace of $\oL_n(V)$ represented by elements of the form $(W_i,x_i)_{i\in I}$ for which the $x_i$ span a subspace of dimension $|I|-1$. This is the maximal dimension possible, since by definition the $x_i$ share a linear relation. This relation also shows that the regularity condition for an element $(W_i,x_i)_{i\in I}\in F_n(\bigoplus_{i\in I} W_i)$ is equivalent to all $(|I|-1)$-element subtuples of $(x_i)_{i\in I}$ being linearly independent.
Given a refinement $\C^n=\bigoplus_{i\in I}W_i$ of $\C^n=\bigoplus_{j\in J} W_j'$, the associated map $F_n(\bigoplus_{i\in I}W_i)\to F_n(\bigoplus_{j\in J} W_j')$ sends an element $(W_i,x_i)_{i\in I}$ to $(W_j',\sum_{i\in I_j} ((\dim W_i/\dim W_j')\cdot x_i))_{j\in J}$, where $I_j\subseteq I$ denotes the subset of those $i$ for which $W_i\subseteq W_j'$.
\begin{Remark} \label{rem:regular} It is not hard to check that this element does indeed lie in $F_n(\bigoplus_{j\in J} W_j')$, i.e., that the span of the second coordinates is $(|J|-1)$-dimensional. In fact, any linear relation of $|J|-1$ many of the $\sum_{i\in I_j} ((\dim W_i/\dim W_j')\cdot x_i)$ directly gives a linear relation of at most $|I|-1$ of the $x_i$. In particular this keeps the span from being zero and hence allows a rescaling to norm $1$ in the sense of the previous section. The rescaling is omitted from the notation above in favor of readability but necessary for the image to land in $\oL_n$. For arbitrary elements $(W_j,x_j)_{j\in j}$ of $\oL_n(V)$ the averaging sum can be zero, making the above map ill-defined. This is the reason for restricting to the regular elements in the definition of $F_n$.
\end{Remark}
\begin{Def} We define the orthogonal space $Z_n$ as the homotopy colimit of $F_n$.
\end{Def}
By the homotopy colimit we mean the (topologically enriched) bar construction applied levelwise. The $U(n)$-action on $\oL_n$ restricts to compatible maps \[ \varphi_*:F_n(\bigoplus _{i\in I} W_i)\to F_n(\varphi\cdot (\bigoplus _{i\in I} W_i))\] and hence turns $F_n$ into a $U(n)$-diagram in the sense of \cite{JS01} or \cite{DM14}. Then the bar construction inherits a $U(n)$-action by combining the one on the nerve of $\mathcal{L}_n$ and the ones on the values of $F_n$ (cf. \cite[Proposition 2.4]{JS01}), turning $Z_n$ into a $U(n)$-orthogonal space.

Mapping $F_n$ to the terminal constant functor $*$ induces a $U(n)$-equivariant map \[ p:Z_n\to \hocolim_{\mL_n}*=|\mL_n|,\]
where we think of $|\mL_n|$ as a constant orthogonal space.
\begin{Prop} \label{prop:zglobal} The morphism \[ E_{gl}U(n)\times_{U(n)} p:E_{gl}U(n)\times_{U(n)} Z_n\to E_{gl}U(n)\times_{U(n)} |\mL_n|\] is a global equivalence.
\end{Prop}
\begin{proof} In view of Lemma \ref{lem:fixedpointsquot} we have to show that for every compact Lie group $G$ and every continuous group homomorphism $\varphi:G\to U(n)$ the map $p(\mathcal{U}_G)$ induces a weak equivalence on $\Gamma(\varphi)$-fixed points. The $\Gamma(\varphi)$-fixed points of $Z_n$ are given by $\hocolim_{\mL_n^{\im(\varphi)}} F_n^{\Gamma(\varphi)}$. Since the homotopy colimit is homotopical, it thus suffices to show that $F_n(\bigoplus_{i\in I}W_i)^{\Gamma(\varphi)}$ is weakly contractible for every decomposition $\C^n=\bigoplus_{i\in I}W_i$ that is fixed by $\im(\varphi)$. Given such a decomposition, the indexing set $I$ is acted on by $G$.
Then the space $F_n(\bigoplus_{i\in I}W_i)^{\Gamma(\varphi)}$ can be identified with the space of linear embeddings from the reduced permutation representation $\widetilde{\R}[I]$ into $\mathcal{U}_G$ (via $f\mapsto (W_i,1/(\dim(W_i))\cdot f(e_i))_{i\in I}$). As $\U_G$ is a complete $G$-universe this space is weakly contractible (which for example follows by the same argument as for \cite[Proposition 2.4]{Sch15}) and so we are done.
\end{proof}
We now construct a map $\alpha:Z_n\to \oL_n$ by applying the universal property of the bar construction, i.e., by giving a homotopy coherent natural transformation from $F_n$ to the constant functor with value $\oL_n$. On objects it is defined to be the inclusion of $F_n(\bigoplus_{i\in I} W_i)$ into $\oL_n$. To define $\alpha$ on higher simplices we introduce the following notation: Given an element $x=(W_i,x_i)_{i\in I}$ of $F_n(\bigoplus_{i\in I} W_i)$ and a subset $J\subseteq I$, we denote by $W_J$ the direct sum of all $W_j$ with $j\in J$ and by $x_J$ the vector $\sum_{j\in J} (\dim(W_j)/\dim(W_J)\cdot x_j)$.

Now we assume given an ascending chain of decompositions starting with $\bigoplus_{i\in I} W_i$. We interpret it as a chain of equivalence relations on the indexing set $I$. In particular, for every $i\in I$ we get an ascending chain $\{i\}=J_i^0\subseteq J_i^1\subseteq \hdots \subseteq J_i^k$ of subsets of $I$ given by those elements which are equivalent to $i$ at the respective stage of the chain. Then the map $F_n(\bigoplus_{i\in I} W_i)\times \Delta^k\to \oL_n$ is defined via
\[ ((W_i,x_i)_{i\in I},(t_0,\hdots,t_k)) \mapsto (W_i,\sum_{l=0}^kt_l\cdot x_{J_i^l})_{i\in I}\]
plus rescaling to norm $1$. It is not hard to check that this is well-defined (again using the regularity of the $x_i$ as in Remark \ref{rem:regular}) and that it gives a homotopy-coherent cone over $F$. It remains to show:

\begin{Prop} After one suspension, the map $\alpha:Z_n\to \oL_n$ becomes a $U(n)$-global equivalence, i.e., for all compact Lie groups $G$ it induces a weak $(U(n)\times G)$-equivalence when evaluated on $\mathcal{U}_G$. 
\end{Prop}
\begin{proof} Let $G$ be a compact Lie group. We consider the following filtration on $Z_n(\U_G)$: Given a decomposition $\C^n=\bigoplus_{i\in I} W_i$ we denote by $\mL_n(\geq \bigoplus_{i\in I} W_i)$ the sub-poset of $\mL_n$ given by the decompositions that are refined by $\bigoplus_{i\in I} W_i$ (and similarly by $\mL_n(> \bigoplus_{i\in I} W_i)$ those that are properly refined). Furthermore, by $\mL_n^{(k)}$ we denote the subposet of all decompositions into at most $k$ summands. This filtration on the level of posets induces a filtration $\{Z_n^{(k)}\}$ of $Z_n$ after taking homotopy colimits. There are similar filtrations for $\oL_n$, where we say that an element $(W_j',x_j)_{j\in J}$ (with all $x_j$ distinct) lies in $\oL_n(\geq \bigoplus_{i\in I} W_i)$, $\oL_n(> \bigoplus_{i\in I} W_i)$ or $\oL_n^{(k)}$ if the decomposition $\C^n=\bigoplus_{j\in J}W_j'$ is contained in the respective sub-poset of $\mL_n$.
 
These filtrations are preserved by $\alpha$ and so we can consider the induced map on quotients. The quotient $(Z_n^{(k)}/Z_n^{(k-1)})(\U_G)$ can be $(U(n)\times G)$-equivariantly identified with the wedge
\[ \bigvee_{n=l_1+\hdots+l_k, l_1\geq \hdots \geq l_k} U(n)\times_{N_{U(n)}(\prod_{i=1}^k U(l_i))}((\hocolim_{(\mL_n)(\geq \bigoplus \C^{l_i})}F_n(\U_G)/\hocolim_{(\mL_n)(>\bigoplus \C^{l_i})}F_n(\U_G)). \]
Likewise, $\oL_n(\U_G)^{(k)}/\oL_n(\U_G)^{(k-1)}$ also decomposes as
\[ \bigvee_{n=l_1+\hdots+l_k, l_1\geq \hdots \geq l_k} U(n)\times_{N_{U(n)}(\prod_{i=1}^k U(l_i))} (\oL_n(\U_G)(\geq \bigoplus_{i=1,\hdots,k} \C^{l_i}))/(\oL_n(\U_G)(> \bigoplus_{i=1,\hdots,k} \C^{l_i})).
\]
Moreover, the map $\alpha$ preserves these decompositions. The space $\oL_n(\U_G)(\geq \bigoplus \C^{l_i})$ is $(N_{U(n)}(\prod_{i=1}^k U(l_i))\times G)$-homeomorphic to the unit sphere in the $G$-representation $\overline{\R^k}\otimes \U_G$ (where the normalizer acts through its projection to $\Sigma_k$). In this description, $\oL_n(\U_G)(< \bigoplus \C^{l_i})$ corresponds to the subspace of tuples $(v_1,\hdots,v_k)$ for which at least two of the $v_i$ are equal. We first note:
\begin{Lemma} The restriction of $\alpha$ induces $(W_{U(n)}(\prod_{i=1}^k U(l_i))\times G))$-equivalences
\[ \hocolim_{(\mL_n)(\geq \bigoplus \C^{l_i})}F_n(\U_G) \simeq \oL_n(\U_G)(\geq \bigoplus_{i=1,\hdots,k} \C^{l_i})\cong S(\overline{\R^k}\otimes \U_G) \]
and
\[ \hocolim_{(\mL_n)(>\bigoplus \C^{l_i})}F_n(\U_G)\simeq \oL_n(\U_G)(> \bigoplus_{i=1,\hdots,k} \C^{l_i}), \]
after one suspension.
\end{Lemma}
\begin{proof} The product of the $U(l_i)$'s acts trivially, so it suffices to show that both maps are $(W_{U(n)}(\prod_{i=1}^k U(l_i))\times G)$-equivalences. We start with the first one. The poset $\mL_n(\geq \bigoplus \C^{l_i})$ has a minimal element, so it suffices to show that the embedding $F_n(\bigoplus \C^{l_i})(\U_G)\hookrightarrow S(\overline{\R^k})$ is an equivariant weak equivalence. Up to rescaling, this inclusion corresponds to the inclusion of the space of linear embeddings $\overline{\R^k}\hookrightarrow \U_G$ into the space of all non-zero linear maps $\overline{\R^k}\to \U_G$. We claim that this map is even a $(\Sigma_k\times G)$-weak equivalence.
We saw in the proof of Proposition \ref{prop:sympow} that the space of non-zero linear maps is a universal space for the family $\mathcal{C}_k^{\Sigma}\langle G\rangle$. But $F_n(\bigoplus \C^{l_i})(\U_G)$ is also a universal space for this family, since for every $G$-representation $W$ the space of equivariant linear embeddings $W\hookrightarrow \U_G$ is weakly contractible, as we already used in the proof of Proposition \ref{prop:zglobal}. Any equivariant map between universal spaces for the same collection of subgroups is automatically an equivariant weak equivalence, so this finishes the proof of the first statement.

For the second statement we can again filter both sides by the number of summands. Then, similarly to above, the subquotients are wedges of inductions of spaces of the form $(\hocolim_{(\mL_n)(>\bigoplus_{j=1}^{l} \C^{m_j})}F_n(\U_G))$ with $l<k$, and similarly for $\oL_n$. So, by induction we can reduce to the first statement and are done.
\end{proof}
Hence we see that $\alpha$ induces a $(U(n)\times G)$-weak equivalence on the subquotients. The inclusion maps in both filtrations are equivariant cofibrations (which in the case of $\oL_n$ follows from the pushouts described in Appendix \ref{app:cw}), so this shows that $\alpha$ induces a weak equivalence on homotopy cofibers. Hence, it induces a $(U(n)\times G)$-weak equivalence after one suspension.
\end{proof}

\subsection{Quotients in the complexity filtration} \label{sec:kucomp}
In this section we describe a global version of another filtration induced by Arone and Lesh, which we call the \emph{complexity filtration}. It interpolates between $ku$ and $Sp^{\infty}\simeq H\mathbb{Z}$ and is constructed as follows: We first note that there is a morphism from $ku$ to $Sp^{\infty}$ that sends a complex vector space to its dimension, non-equivariantly realizing the $0$-th Postnikov section. By definition, it maps $ku^n$ into $Sp^n$. Then the $n$-th term $A^u_n$ of the complexity filtration is defined as the homotopy pushout
\[ \xymatrix{ ku^n\ar[r] \ar[d] & ku \ar[d]\\
	      Sp^n\ar[r] & A^u_n. }
\]
Concretely, we let $A^u_n$ be the spectrum $([0,1]_+\wedge Sp^n)\vee_{ku^n} ku$, where the embedding $Sp^n\to [0,1]_+\wedge Sp^n$ is via the endpoint $1$. Since the map $ku^n\to ku$ is always a level-cofibration, $A^u_n$ could also be defined as the strict pushout, but we use the mapping cylinder construction to ensure that the induced maps $A_n^u\to A_{n+1}^u$ are level-cofibrations. Hence we obtain a sequence of morphisms of orthogonal spectra
\[ ku\cong A_0^u\xr{p_0} A^u_1\xr{p_1} \hdots \to A^u_{\infty}\simeq H\Z, \]
since both $ku^0$ and $Sp^0$ are a point and the $Sp^n$'s converge to $H\mathbb{Z}$.

\begin{Remark}The presentation of the complexity filtration given here is different to the way it was originally constructed by Arone and Lesh. In \cite{AL07} they associated to any augmented permutative category $\mathcal{C}$ a sequence of permutative categories interpolating between $\mathcal{C}$ and the category $\N$ and obtained the complexity filtration as the spectrum realization of this categorical filtration. Later in \cite{AL10} they constructed the modified rank filtration of the spectrum of an augmented permutative category via $\Gamma$-spaces and showed that the complexity filtration has this different description that we use here. 
\end{Remark}
 
Since we know the filtration quotients of the rank filtration and the symmetric product filtration, it is not difficult to obtain a description for the filtration quotients of the complexity filtration. By forming termwise quotients in the pushout diagram defining $A_n^u$, we see that the sequence \begin{equation} \label{eq:quot} ku^n/ku^{n-1}\to Sp^n/Sp^{n-1}\to A_n^u/A_{n-1}^u \end{equation}
is a mapping cone sequence. Using the results and notation of the previous sections we can identify the first two terms with suspension spectra of the orthogonal spaces $L_{\C}(\C^n,\Sym(\C\otimes -))_+\wedge_{U(n)} \oL_n^\dia$ respectively $S(\overline{\R^n}\otimes -)^\dia/\Sigma_n$. Moreover, the map is induced from the map of orthogonal spaces which collapses $L_{\C}(\C^n,\Sym(\C\otimes -))$ to a point and sends an element in $\oL_n$ represented by a tuple $(W_i,x_i)_{i\in I}$ to the element of $S(\overline{\R^n}\otimes -)/\Sigma_n$ represented by 
\[ (\underbrace{x_{i_1},\hdots,x_{i_1}}_{\dim W_{i_1}},\underbrace{x_{i_2},\hdots, x_{i_2}}_{\dim W_{i_2}},\hdots, \underbrace{x_{i_j},\hdots, x_{i_j}}_{\dim W_{i_j}}) \]
for some enumeration $i_1,i_2,\hdots,i_j$ of $I$, on which it does not depend since the $\Sigma_n$-action is quotiened out. In fact, the map $\oL_n\to S(\overline{\R^n}\otimes -)/\Sigma_n$ induces an isomorphism $\oL_n/U(n)\xr{\cong} S(\overline{\R^n}\otimes -)/\Sigma_n$ (and in particular, the global classifying space of complete subgroups of $U(n)$ is globally equivalent to the global classifying space of complete subgroups of $\Sigma_n$). In other words, the map $ku^n/ku^{n-1}\to Sp^n/Sp^{n-1}$ is induced -- by forming $U(n)$-orbits and applying the suspension spectrum functor -- from the map of $U(n)$-orthogonal spaces
\[ L_{\C}(\C^n,\Sym(\C\otimes -))_+\wedge \oL_n^\dia\to  \oL_n^\dia\]
that collapses $L_{\C}(\C^n,\Sym(\C\otimes -))$ to a point. So we have:

\begin{Cor} The quotient $A^u_n/A^u_{n-1}$ is isomorphic to the suspension spectrum of the based orthogonal space $L_{\C}(\C^n,\Sym(\C\otimes -))^\dia\wedge_{U(n)} \oL_n^\dia$.
\end{Cor}
Since the smash product of two unreduced suspensions is isomorphic to the unreduced suspension of the join $*$, this based orthogonal space can be rewritten as $(L_{\C}(\C^n,\Sym(\C\otimes -))* \oL_n)^\dia/U(n).$ From Section \ref{sec:kurank} we know that the first join factor is a global universal space for $U(n)$ and that the second is a global universal space for the collection of complete (or non-isotypical) subgroups of $U(n)$.

The global homotopy type of this join is then implied by the following easy lemma:
\begin{Lemma} \label{lem:join} Let $\mathcal{F}$ be any collection of subgroups of a Lie group $K$, $E_{gl}\mathcal{F}$ a global universal space for $\mathcal{F}$ and $E_{gl}K$ be a global universal space for $K$. Then the join $E_{gl}\mathcal{F}*E_{gl}K$ is a global universal space for the collection $\overline{\mathcal{F}}$, i.e., $\mathcal{F}$ with the trivial subgroup added. 
\end{Lemma}
\begin{proof}
This follows directly from the fact that the join commutes with taking fixed points.
\end{proof}
Hence, denoting the collection of complete and trivial subgroups of $U(n)$ by $\overline{\mathcal{C}}^u_n$ and the collection of non-isotypical and trivial subgroups by $\overline{\mathcal{I}}^u_n$, we obtain:
\begin{Theorem}[Subquotients in the complexity filtration] \label{theo:quotcomp} There are global equivalences \[ A^u_n/A^u_{n-1} \simeq \Sigma^{\infty} (B_{gl}(\overline{\mathcal{C}}^u_n)^\dia) \simeq \Sigma^{\infty} (B_{gl}(\overline{\mathcal{I}}^u_n)^\dia). \]
\end{Theorem}

\begin{Remark} One can show that there is also a global equivalence\[A_n^u/A_{n-1}^u\simeq \Sigma^{\infty} (S^1\wedge (E_{gl}U(n)_+\wedge _{U(n)} (\mL_n^\dia\wedge S^{\C^n}))),\]
globally generalizing another description of the quotients due to Arone and Lesh. The proof uses straightforward equivariant adaptions of the arguments in \cite[Sec. 9]{AL07} together with Theorem \ref{theo:zig-zag} of this paper.
\end{Remark}

\subsection{The modified rank filtration on $0$-th homotopy}
\label{sec:pi0ku}
In this section we describe the effect of the modified rank filtration on the global functor $\underline{\pi}_0$, the zero-th homotopy group.

For a compact Lie group $G$ we denote by $\Rep_{\C}(G)$ the complex representation ring of $G$. Every group homomorphism $f:G\to K$ gives rise to a restriction map $f^*:\Rep_{\C}(G)\to \Rep_{\C}(K)$ by pulling back the action on representations. Furthermore, if $H$ is a finite index subgroup of $G$ there is an induction map $\Ind_H^G:\Rep_{\C}(H)\to \Rep_{\C}(G)$ which sends an $H$-representation $W$ to the $G$-representation $\Ind_H^G W=\map_H(G,W)$.

This is connected to $\pi_0^G(ku)$ as follows: Let $W$ be a finite dimensional complex $G$-representation together with an isometric embedding $\psi:W\to \Sym(\C\otimes V)$ for some finite dimensional real representation $V$. This data gives rise to an element $[W]$ in $\pi_0^G (ku)$ represented by the map $S^V\to ku(V),v\mapsto (\psi(W),v)$. A different way to view this assignment $W\mapsto [W]$ is given by the following: There are maps \[ \alpha_n:\Sigma^{\infty}_+ (L_{\C}(\C^n,\Sym(\C\otimes -))/U(n))\to ku^n \] which in level $V$ send a pair $(\varphi,v)$ to the configuration $(\varphi(\C^n),v)$.
The orthogonal space $L_{\C}(\C^n,\Sym(\C\otimes -))/U(n)$ is a model for a global classifying space $B_{gl}U(n)$ of $U(n)$ and hence \[ \pi_0^G(L_{\C}(\C^n,\Sym(\C\otimes -))/U(n))\cong \Rep(G,U(n))\cong \{\text{isom. classes of } n\text{-dim. } G\text{-representations}\}\]
The assignment $W\mapsto [W]$ is then the composition
\[ \pi_0^G(B_{gl}U(n))\to \pi_0^G (\Sigma^{\infty}_+ B_{gl}U(n))\xr{(\alpha_n)_*} \pi_0^G(ku). \]
In particular, this shows that $[W]$ only depends on the isomorphism type of $W$ and not on the choice of $\psi$.
\begin{Remark} \label{rem:smoothind} The map $W\mapsto [W]$ is additive and induces a homomorphism $\Rep_{\C}(G)\to \pi_0^G(ku)$ which sends restrictions to homotopy-theoretic restrictions and finite index inductions to homotopy-theoretic transfers. It is an isomorphism if $G$ is finite. The functor mapping a compact Lie group $G$ to its complex representation ring does extend to a full global functor, i.e., there also exist inductions along infinite index subgroup inclusions satisfying the double coset formula. These are given by the \emph{smooth inductions} introduced by Segal in \cite{Seg68}.
However, they are not mapped to homotopy theoretic transfers under the map $\Rep_{\C}(G)\to \pi_0^G(ku)$ above, which can be seen as a reason for why this map is in general not an isomorphism for non-discrete compact Lie groups. In fact, the representation ring always maps injectively into $\pi_0^G(ku)$ and the cokernel is generated by transfers of elements in $\pi_0^H(ku)$, where $H$ ranges through all subgroups of $G$ that have infinite index but finite Weyl group. These results are due to Schwede, and also follow from Theorem \ref{theo:pirank} below.
\end{Remark}
In this section we are going to explain the intermediate groups $\pi_0^G (ku^n)$. If $m\leq n$, the map $\alpha_m:\Sigma^{\infty}_+ (B_{gl}U(m))\to ku$ described above has image in $ku^n$ and hence every $m$-dimensional $G$-representation $W$ already defines an element $[W]$ in $\pi_0^G(ku^n)$ which only depends on its isomorphism type. It turns out that $\pi_0^G (ku^n)$ is additively generated by transfers of these elements. To understand the relations, we observe: If $W$ is $n$-dimensional, the class $[W]$ does not make sense in $\pi_0^G (ku^{n-1})$ yet, but it might already secretly live there in the following sense:
\begin{itemize}
\item If $W=W_1\oplus W_2$ is (non-trivially) decomposable, then the classes $[W_1],[W_2]$ already live in $\pi_0^G (ku^{n-1})$ and hence so does their homotopy theoretic sum $[W_1]+[W_2]$.
\item If $W=\Ind_H^GW'$ is induced up from a proper finite index subgroup $H$, then $[W']$ is an element in $\pi_0^H (ku^{n-1})$ and hence one can form the homotopy theoretic transfer $\tr_H^G [W']\in \pi_0^G (ku^{n-1})$.
\end{itemize}
We will see that both elements map to $[W]$ under $\pi_0^G (ku^{n-1})\to \pi_0^G (ku^n)$ and that $[W]$ lies in the image if and only if one of those two conditions is satisfied. Furthermore, if $[W]$ does lie in the image, then it might be so for different reasons: It could both be decomposable and induced, or it could be induced up in different ways. We will see that these are reflected nicely in certain fixed points of the decomposition poset $\mathcal{L}_n$ and that it is exactly these different reasons that are identified via the boundary map $\partial:\pi_1^G (ku^n/ku^{n-1})\to \pi_0^G (ku^{n-1})$.

\begin{Example} The easiest case is that of $G=\Z/p$ for a prime $p$. Let $\eta_p$ be a primitive $p$-th root of unity in $\C$. Then $\pi_0^{\Z/p} (ku^1)$ is free with basis $\{[\eta_p^1],[\eta_p^2],\hdots,[\eta_p^p]\}\cup \{\tr_1^{\Z/p}[1]\}$ (cf. Example \ref{exa:stablebglg}). The elements $[\eta_p^i]$ also form a basis of the representation ring of $G$ and hence the only difference between $\pi_0^{\Z/p}(ku^1)$ and $\pi_0^{\Z/p}(ku)$ lies in the element $\tr_1^{\Z/p}[1]$, which is equal to the sum of the $\eta_p^i$ in $\pi_0^{\Z/p}(ku)$. It will be a consequence of Theorem \ref{theo:pirank} that this identification takes place in $\pi_0^{\Z/p}(ku^p)$ for the first time.
\end{Example}
In global equivariant homotopy theory, all this can be phrased via universal examples: If an $n$-dimensional $G$-representation $W$ is the direct sum of two subrepresentations $W_1$ and $W_2$ of dimensions $k$ and $l$, then -- up to conjugation -- the associated homomorphism $\beta:G\to U(n)$ factors through the embedding $U(k)\times U(l)\to U(k+l=n)$. For $t\geq 1$ let  $\tau_t^{\C}$ denote the tautological complex $t$-dimensional representation of $U(t)$. Then the fact that $[W]=[W_1]+[W_2]$ in $\Rep_{\C}(G)$ is the restriction along $\widetilde{\beta}:G\to U(k)\times U(l)$ of the relation
\begin{equation} \label{eq:alpha} ({\res}_{U(k)\times U(l)}^{U(k+l)})^*(\tau_{k+l}^{\C})=(p_1)^*(\tau_k^{\C}) + (p_2)^*(\tau_l^{\C}), \end{equation}
where $p_1$ and $p_2$ denote the projections from $U(k)\times U(l)$ to $U(k)$ respectively $U(l)$. We denote this relation by $a(k,l)$.

Likewise, if $W$ is the induction of a $j$-dimensional representation $W'$ of a subgroup $H$ of index $i$, the associated group homomorphism $\beta:G\to U(n)$ takes image in the wreath product $\Sigma_i\wr U(j)$, i.e., the semidirect product of $U(j)^{\times i}$ and $\Sigma_i$ via the permutation action on the product coordinates. Then the relation $[W]=\tr_H^G[W']$ in the representation ring global functor is the restriction along $\widetilde{\beta}:G\to \Sigma_i\wr U(j)$ of
\begin{equation}  \label{eq:beta} {\res}_{\Sigma_i\wr U(j)}^{U(i\cdot j)}(\tau_{i\cdot j}^{\C})={\tr}_{U(j)\times \Sigma_{i-1}\wr U(j)}^{\Sigma_i\wr U(j)}(p^*(\tau_j^{\C})). \end{equation}
Here, $p$ stands for the projection from $U(j)\times (\Sigma_{i-1}\wr U(j))$ to $U(j)$. Let $b(i,j)$ denote this relation.

In these terms the intermediate homotopy groups can be described as follows:
\begin{Theorem} \label{theo:pirank} The global functor $\upi_0 (ku^n)$ is isomorphic to the free global functor generated by the elements $\tau_1^{\C},\tau_2^{\C},\hdots, \tau_n^{\C}$ modulo the relations $a(k,l)$ for all $k+l\leq n$ and $b(i,j)$ for all $i\cdot j\leq n$.
\end{Theorem}

We try to make clear what this means at a fixed compact Lie group $G$ in a few examples of these filtrations in Section \ref{sec:exa}. However, the result at a specific group is often a lot more complicated than the global formula.

The remainder of this section is devoted to proving this theorem. It proceeds by comparing the cofiber sequence $ku^{n-1}\to ku^n\to ku^n/ku^{n-1}$ to another one with the same cofiber, namely
\begin{equation} \label{eq:cofiber} \xymatrix{ \Sigma^{\infty}_+ (E_{gl}U(n)\times_{U(n)}\oL_n) \ar[r]^-p\ar@{-->}[d]_{\psi_n} & \Sigma^{\infty}_+ (B_{gl}U(n))\ar[r]\ar[d]^{\alpha_n} & \Sigma^{\infty}(E_{gl}U(n)_+\wedge_{U(n)} \oL_n^\dia) \ar[d]^{\cong}\\
  ku^{n-1}\ar[r]_{i_n} & ku^n\ar[r] & ku^n/ku^{n-1},}
\end{equation}
where the vertical isomorphism on the right is the one explained in Section \ref{sec:kurank}. The map $\psi_n$ could be obtained (at least as a stable map) via the triangulated structure on the homotopy category, but we make it explicit below in order to understand its effect on $\upi_0$.

The map of cofiber sequences exhibits the left square as a homotopy pushout, giving rise to a Mayer-Vietoris sequence on homotopy groups. In particular:
\begin{Cor} \label{cor:exact} The sequence
\[ \upi_0 (\Sigma^{\infty}_+ (E_{gl}U(n)\times_{U(n)}\oL_n))\xr{(p_*,-(\psi_n)_*)} \upi_0(\Sigma^{\infty}_+ (B_{gl}U(n)))\oplus \upi_0(ku^{n-1})\xr{\binom{(\alpha_n)_*}{(i_n)_*}} \upi_0(ku^n)\to 0 \]
is exact. 
\end{Cor}
The rest of the proof is divided into the following two parts:
\begin{enumerate}
 \item A description of $\upi_0 (\Sigma^{\infty}_+ (E_{gl}U(n)\times_{U(n)}\oL_n))$.
 \item Constructing $\psi_n$ and determining its effect on $\upi_0$. 
\end{enumerate}

We start with number $(1)$. Applying Lemma \ref{lem:fixedpointsquot} to $Y=(E_{gl}U(n)\times_{U(n)}\oL_n)(\U_G)$ and $K=U(n)$ for a compact Lie group $G$, we see that the $G$-fixed points of the quotient decompose as
\[ \bigsqcup_{\langle \alpha:G\to U(n)\rangle}EC(\alpha)\times_{C(\alpha)} ((\oL_n)(\U_G))^{\Gamma(\alpha)}. \]

A tuple $(W_i,x_i)_{i\in I}\in \oL_n(\U_G)$ (with pairwise different $x_i$) is $\Gamma(\alpha)$-fixed if and only if each pair $(\alpha(g)(W_i),g\cdot x_i)$ is equal to some $(W_j,x_j)$ in the tuple. Hence, any such fixed point in particular gives rise to a non-trivial decomposition $\C^n=\bigoplus_{i\in I} W_i$ that is weakly fixed by $\alpha(G)$. Here, \emph{weakly fixed} means that not necessarily every $W_i$ is fixed itself, but they may be permuted in a way encoded by a $G$-action on the indexing set $I$.
Making use of the weak equivalence to the decomposition lattice constructed in Section \ref{sec:lattice}, we see that the path-component of $(W_i,x_i)_{i\in I}$ in the $\Gamma(\alpha)$-fixed points only depend on this associated decomposition and that every weakly fixed decomposition is realized. Furthermore, if one decomposition refines another, the associated fixed points lie in the same path-component. Written in a more coordinate-free  way we get:
\begin{Prop} The set $\pi_0^G((E_{gl}U(n)\times_{U(n)}\oL_n))$ stands in natural bijection to the set of pairs \[ \{(W,\oplus_{i\in I}W_i)\ |\ W\text{ n-dim $G$-rep.}, W=\bigoplus_{i\in I}W_i \text{ non-trivial and weakly $G$-fixed}\}\] modulo isomorphisms of representations and refinement of decompositions.

In this description, the induced map to $\pi_0^G(B_{gl}U(n))\cong \{\text{isom. classes of $n$-dim $G$-rep.}\}$ is given by forgetting the decompositions.
\end{Prop}
Let $W=\bigoplus_{i\in I} W_i$ be such a weakly $G$-fixed partition and denote by $A_1,A_2,\hdots,A_k$ the orbits of the induced $G$-action on $I$. Then the decomposition $W=\bigoplus_{j=1,\hdots k}(\bigoplus_{i\in A_j}W_i)$ is strongly fixed. It is non-trivial if $I$ is not transitive, in which case it refines a strongly fixed decomposition with two summands. So we see:
\begin{Cor} \label{cor:type} Every point in $\pi_0^G(E_{gl}U(n)\times_{U(n)}\oL_n)$ is represented by a weakly $G$-fixed decomposition of one of the following two types:
\begin{enumerate}
	\item $W=W_1\oplus W_2$ and $W_1,W_2$ are $G$-subrepresentations.
	\item $W=W_1\oplus  \hdots \oplus W_k$ and the $W_i$ are permuted transitively by the $G$-action.
\end{enumerate}
\end{Cor}

Decompositions of the second type can be interpreted in the following way: Let $H$ be the isotropy subgroup of $W_1$ under the $G$-action on the set $\{W_i\}_{i=1,\hdots,k}$. Then $H$ is a subgroup of index $k$ in $G$, $W_1$ is an $H$-representation and the map $\Ind_H^G(W_1)\to W$ adjoint to the inclusion gives an isomorphism of $G$-representations. Vice versa, every induced representation $\Ind_H^G W'$ (where $H$ has finite index in $G$) possesses the weakly $G$-fixed decomposition $\Ind_H^G W'=\bigoplus _{gH\in G/H} gW'$. Hence, a general weakly $G$-fixed point of the decomposition poset can be interpreted as exhibiting $W$ as a combination of sums and transfers.

\begin{Remark} This also lets us determine $\pi_0^G$ of the cone $E_{gl}U(n)_+\wedge_{U(n)} \oL_n^\dia$ (and hence of the quotient $ku^n/ku^{n-1}$ via Proposition \ref{prop:pi0susp}) explicitly. It is given by isomorphism classes of irreducible $n$-dimensional $G$-representations which are not the induction of a representation from a proper finite index subgroup. The tautological $U(n)$-representation always has this property, so we see that the maps $\upi_0 (ku^{n-1})\to \upi_0 (ku^n)$ are never globally surjective.
\end{Remark}

The decompositions of Corollary \ref{cor:type} have universal representatives: Given $k,l>0$ with $k+l=n$, the $(U(k)\times U(l))$-representation obtained by restricting $\tau_n^{\C}$ along the embedding $U(k)\times U(l)\hookrightarrow U(n)$ decomposes as $\tau_k^{\C}\oplus \tau_l^{\C}$. We denote the element of $\pi_0^{U(k)\times U(l)}(E_{gl}U(n)\times_{U(n)}\oL_n)$ associated to this decomposition by $\widetilde{\alpha}(k,l)$. Likewise, given $i,j\in \N$ with $i\cdot j=n$, the restriction of $\tau_n^{\C}$ along $\Sigma_i\wr U(j)\hookrightarrow U(n)$ is the transfer of $p^*(\tau_j^{\C})$, where $p$ is the projection to $U(j)$. Hence there is an associated weakly $G$-fixed decomposition of type $(2)$ above, which we denote by $\widetilde{\beta}(i,j)$. We obtain:
\begin{Cor} \label{cor:gen} The $\Rep$-functor $\upi_0((E_{gl}U(n)\times_{U(n)}\oL_n))$ is generated by the elements $\{\widetilde{\alpha}(k,l)\}_{k+l=n}$ and $\{\widetilde{\beta}(i,j)\}_{i\cdot j=n}$. Hence, by Proposition \ref{prop:pi0susp}, so is $\upi_0(\Sigma^{\infty}_+(E_{gl}U(n)\times_{U(n)}\oL_n))$ as a global functor.
\end{Cor}

So it remains to show that $\psi_n$ indeed maps the $\widetilde{\alpha}'$ and $\widetilde{\beta}'s$ to the right hand sides of Equations (\ref{eq:alpha}) and (\ref{eq:beta}) respectively. For this we require an explicit construction of $\psi_n$, which we now explain.

We quickly recall the objects involved: An element of $E_{gl}U(n)(V)$ is a linear isometric embedding $\C^n\to \Sym(\C\otimes V)$. Points in $\oL_n$ are represented by tuples $(W_i,x_i)_{i\in I}$ indexed on a finite set $I$, where the $x_i$ are elements of $V$ and the $W_i$ are pairwise orthogonal subspaces of $\C^n$ which add up to all of $\C^n$. Furthermore, these tuples are required to satisfy the two conditions $\sum \dim(W_i)\cdot x_i=0$ and $\sum \dim(W_i)|x_i|^2=1$. Finally, elements of $ku^n(V)$ are also represented by tuples $(W_i,x_i)_{i\in I}$, but this time the $W_i$ are orthogonal subspaces of $\Sym(\C\otimes V)$ and the only requirement is that the sum of the dimensions is at most $n$.

Now we come to the construction of $\psi_n$. We would like to define each level $(E_{gl}U(n)\times_{U(n)}\oL_n)(V)_+\wedge S^V\to ku^{n-1}(V)$ by sending $(\varphi,(W_i,x_i)_{i\in I},v)$ to the tuple $(\varphi(W_i),x_i+v)_{i\in I}$. However, even though all the $W_i$ necessarily have smaller dimension than $n$, their sum is still $n$-dimensional. So for fixed $v$ this tuple does not represent an element in $ku^{n-1}$. The idea is to shrink the domain of each of the coordinate functions $v\mapsto (\varphi(W_i),x_i+v)$, so that they become equal to the basepoint outside a certain neighborhood of $-x_i$.
For this let $s:[0,\infty]\to [0,\infty]$ be a map which induces a homeomorphism $[0,1/(2n^2)]\xr{\cong} [0,\infty]$ and is constant $\infty$ on $[1/(2n^2),\infty]$. Furthermore, given a finite tuple $x=(x_i)_{i\in I}$ of vectors of a (finite dimensional) real inner product space $V$ we let $p_x:V\to \langle \{x_i\}_{i\in I} \rangle \subseteq V$ denote the linear map defined by
\[ p_x(v)=\sum_{i\in I} \langle v,x_i\rangle \cdot x_i. \]
We need the following properties of this map:
\begin{Lemma} \label{lem:px} The value $p_x(v)$ only depends on the orthogonal projection of $v$ onto the span of the $x_i$ and is an automorphism when restricted to this span. Furthermore, it satisfies the inequality
\[ |p_x(x_j)|\geq |x_j|^3 \]
for every $j$ in $I$.
\end{Lemma}
\begin{Remark} The reason for using $p_x$ instead of the orthogonal projection onto the span of the $x_i$ is that the latter is not continuous in the $x_i$. However, the linear homotopy from the identity to $p_x$ restricts to an isotopy on this span and hence for a fixed tuple $x$ there is essentially no difference.
\end{Remark}
\begin{proof} If a vector is orthogonal to each of the $x_i$ it is sent to $0$ under $p_x$ and hence the value only depends on the orthogonal projection onto the span. For the other two statements we note that the scalar product of $p_x(v)$ and $v$ is equal to the sum of the squares $\langle v,x_i\rangle ^ 2$. Hence, if $v$ is a non-zero vector in the span of the $x_i$, this scalar product is non-zero and in particular $p_x(v)$ is non-zero, so the restriction of $p_{x}$ to the span is injective. Finally, the stated inequality follows from
\[ |p_x(x_j)|\cdot |x_j|\geq \langle p_x(x_j),x_j \rangle = \sum _{i\in I} \langle x_j,x_i \rangle ^2 \geq \langle x_j,x_j \rangle^ 2=|x_j|^ 4, \] where the first step is the Cauchy-Schwarz inequality.
\end{proof}

We use this to obtain a selfmap $s_x^V:S^V\to S^V$ via the formula
\[ s_x^V(v)=(s(|p_x(v)|)-|p_xv|)\cdot p_x(v)+v. \]
This map sends every vector $v$ for which $p_x(v)$ has length larger than $1/(2n^2)$ to the basepoint and is the identity on the orthogonal complement of the span of the $x_i$. Finally, for an element $(\varphi,(W_i,x_i)_{i\in I},v)$ of $(E_{gl}U(n)\times_{U(n)} \oL_n)(V)$ we set
\[ \psi_n(\varphi,(W_i,x_i)_{i\in I},v)=(\varphi(W_i),s_x^V(x_i+v))_{i\in I}.\]

This gives a map of orthogonal spectra: It commutes with the action of elements $A$ of $O(V)$ because of the equality $A(\langle v,x_i\rangle \cdot x_i)=\langle Av,Ax_i\rangle \cdot Ax_i$. Furthermore, if $W$ is another vector space and $w$ an element, then $s_x^{V\oplus W}(x_i+v+w)=s_x^V(x_i+v)+w$ since $w$ is orthogonal to all the $x_i$ and hence $\psi_n$ also commutes with the structure map. In fact, assuring this equality was the reason for introducing the projections into the formula.

Finally, we have to show that $\psi_n$ does indeed take image in $ku^{n-1}$. Each component function $(v\mapsto (\varphi(W_i),s_x^V(x_i+v))$ is equal to the basepoint on all points $v$ for which $p_x(v)$ is more than $1/(2n^2)$ away from $-p_x(x_i)$. We claim that for every fixed $v$ there is at least one $i$ such that this is the case. If this was not true, it would imply that all the $p_x(x_i)$ are less than $1/n^2$ away from each other. Now the conditions for the $x_i$ come into play. Since $\sum \dim(W_i)|x_i|^2=1$, there is at least one $j$ such that $|x_j|^2\geq 1/n$ and hence, by Lemma \ref{lem:px}, we have $|p_x(x_j)|\geq 1/n^ {3/2}\geq 1/n^2$. The equality 
\[ \sum \dim(W_i) \cdot p_x(x_j-x_i) + \sum \dim(W_i) \cdot p_x(x_i)=n\cdot p_x(x_j)\]
implies that  \[|\sum \dim(W_i)\cdot p_x(x_i)|\geq n\cdot |p_x(x_j)|-\sum \dim(W_i)|p_x(x_i-x_j)|>1/n-1/n=0,\]
which contradicts the condition $\sum \dim(W_i)\cdot x_i=0$. Hence, for every $(\varphi;(W_i,x_i)_{i\in I})$ the tuple $(\varphi(W_i),s_x^V(x_i+v))_{i\in I}$ contains at least one basepoint and thus represents an element in $ku^{n-1}$, as the total dimension of the remaining $\varphi(W_i)$ is strictly less than $n$.

Some justification is also needed that $\psi_n$ indeed turns Diagram (\ref{eq:cofiber}) into a map of cofiber sequences, but we outsource this to Appendix \ref{app:cofiber}.

Now we let $(\varphi,(W_i,x_i)_{i\in I})$ be a $G$-fixed point of $(E_{gl}U(n)\times_{U(n)} \oL_n)(V)$, assume that the balls of radius $1/(2n^2)$ around the $p_x(x_i)$ are pairwise disjoint and denote the span of the $x_i$ by $V'$. As noted before, the set $\{x_i\}_{i\in I}$ is permuted by the $G$-action. Then the induced $G$-map $S^V\to ku^{n-1}(V)$ given by $v\mapsto \psi_n(\varphi,(W_i,x_i)_{i\in I},v)$ is equal to the composite
\[ S^V\xr{p_{x}\wedge id} S^{V'}\wedge S^{V-V'}\to (\{x_i\}_+\wedge S^{V'})\wedge S^{V-V'}\cong \bigvee_{i \in I} S^V\xr{\bigvee \varphi(W_i)} ku^{n-1}(V), \]
where the second map is the smash product of $S^{V-V'}$ with the pinch map which collapses everything outside the balls of radius $1/(2n^2)$ around the $x_i$ to a point, and $\varphi(W_i)$ maps $v$ to the configuration $(\varphi(W_i),v)$. Up to homotopy, the first map can be replaced by the canonical homeomorphism $S^V\cong S^{V'}\wedge S^{V-V'}$, since $p_x$ is isotopic to the identity on $V'$.

This lets us prove:
\begin{Prop} \label{prop:image} If an element $y\in \pi_0^G(E_{gl}U(n)\times_{U(n)}\oL_n)$ is associated to $W=W_1\oplus W_2$ and $W_1,W_2$ are $G$-fixed, then \[ (\psi_n)_*(y)=[W_1]+[W_2]\in \pi_0^G (ku^{n-1}). \]
If $y$ is associated to $W=W_1\oplus \hdots \oplus W_k$ and the $W_i$ are permuted transitively by $G$, then \[ (\psi_n)_*(y)={\tr}_H^G [W_1]\in \pi_0^G (ku^{n-1}), \]
where $H$ is the subgroup of elements fixing $W_1$. In particular, 
\[ (\psi_n)_*(\widetilde{\alpha}(k,l))=p_1^*(\tau_k^{\C})+p_2^*(\tau_l^{\C})\]
and
\[ (\psi_n)_*(\widetilde{\beta}(i,j))={\tr}_{\Sigma_{i-1}\wr U(j)}^{\Sigma_i\wr U(j)} (p^*(\tau_j^{\C})).\] 
\end{Prop}
\begin{proof}  Without loss of generality we can assume that $W$ is equal to $\C^n$ with some $G$-action. We start with the first case. Let $(\varphi:\C^n\to \Sym(\C\otimes V),(W_1,x_1),(W_2,x_2))\in (E_{gl}U(n)\times_{U(n)} \oL_n)(V)$ be a fixed point giving rise to such a decomposition, for some finite dimensional $G$-representation $V$. Let $V'$ be the (one-dimensional) span of the $x_i$. It carries the trivial $G$-action, since the $x_i$ are fixed by assumption. Furthermore, as there are only two points it is automatic that the intervals of radius $1/(2n^2)$ around them (or their images under $p_x$) are disjoint and thus the description above shows that $(\psi_n)_*(y)$ is the class of the composition
\[ S^{V'} \wedge S^{V-V'}\to (\{x_1,x_2\}_+\wedge S^{V'})\wedge S^{V-V'}\cong S^V\wedge S^V\xr{\varphi(W_1)\vee \varphi (W_2)}ku^{n-1}(V). \]
Since $G$ acts trivially on the set $\{x_1,x_2\}$, the first map is just the usual pinch map and the second one is the wedge of two $G$-equivariant maps. Hence, this composite represents their sum $[\varphi(W_1)]+[\varphi(W_2)]=[W_1]+[W_2]\in \pi_0^G (ku^{n-1})$.

Now we let $y$ correspond to a decomposition of type two, i.e., $\C^n=W_1\oplus \hdots \oplus W_k$ and the $W_i$ are permuted transitively by $G$. Let $(\varphi,(W_i,x_i)_{i=1,\hdots,k})$ be a representative for this fixed point, chosen in a way that the $x_i$ have distance larger than $1/(n^2)$ from each other. We again denote by $V'$ their span. Then, as seen above, $(\psi_n)_*(y)$ is represented by the composite

\[ S^V\wedge S^{V-V'}\to (\{x_1,\hdots,x_k\}_+\wedge S^{V'})\wedge S^{V-V'}\cong \bigvee_{i=1,\hdots,k} S^V\xr{\bigvee \varphi(gW_i)} ku^{n-1}(V). \]
Using that the $G$-set $\{x_1,\hdots,x_k\}$ is isomorphic to $G/H$, we see that this is precisely the definition of the transfer of the class $[\varphi(W_1)]=[W_1]\in \pi_0^H (ku^{n-1})$ recalled in Section \ref{sec:orthspec}: The first map is the ``transfer pinch map'' and each wedge summand of the second equals the composite $S^V\xr{g^{-1}}S^V\xr{\varphi(W_1)}ku^{n-1} \xr{g\cdot} ku^{n-1}$. This finishes the proof.
\end{proof}

Now we are ready for:
\begin{proof}[Proof of Theorem \ref{theo:pirank}] We proceed by induction on $n$, the case $n=0$ being clear. So now let $n$ be a positive natural number and assume the statement to be true for $n-1$. Together with the induction hypothesis and the fact that $\upi_0(\Sigma^{\infty}_+ B_{gl}U(n))$ is generated by the tautological $U(n)$-representation, the exact sequence \[ \upi_0 (\Sigma^{\infty}_+ (E_{gl}U(n)\times_{U(n)}\oL_n))\xr{(p_*,-(\psi_n)_*)} \upi_0(\Sigma^{\infty}_+ B_{gl}U(n))\oplus \upi_0(ku^{n-1})\xr{\binom{(\alpha_n)_*}{(i_n)_*}} \upi_0(ku^n)\to 0 \] of Corollary \ref{cor:exact} shows that the global functor $\upi_0 (ku^n)$ is generated by the elements $\tau_1^{\C},\hdots,\tau_n^{\C}$. It further shows that the relations are generated by the ones of $\upi_0 (ku^{n-1})$, which we know by induction, and the image of $\upi_0 (\Sigma^{\infty}_+ (E_{gl}U(n)\times_{U(n)}\oL_n))$ in $\upi_0(\Sigma^{\infty}_+ B_{gl}U(n))\oplus \upi_0(ku^{n-1})$.
By Corollary \ref{cor:gen}, the global functor $\upi_0 (\Sigma^{\infty}_+ (E_{gl}U(n)\times_{U(n)}\oL_n))$ is generated by the elements $\{\widetilde{\alpha}(k,l)\}_{k+l=n}
$ and $\{\widetilde{\beta}(i,j)\}_{i\cdot j=n}$, which are sent to the relations $\alpha(k,l)$ and $\beta(i,j)$ by Proposition \ref{prop:image}, so we are done. \end{proof}
\begin{Remark} The fact that for infinite compact Lie groups $G$ the group $\pi_0^G(ku)$ fails to be the representation ring is reflected in the rank filtration: Infinite index inductions of representations do not give rise to a fixed point in the associated decomposition poset and hence these are never identified with the homotopy theoretic transfer.
\end{Remark}

\subsection{The complexity filtration on $0$-th homotopy}
\label{sec:pi0comp}
In this section we explain the behavior of the complexity filtration on $\pi_0$, again starting with the case of $ku$.  The method is similar to that of the last section, but shorter as we can make use of the constructions we made there. Also, as mentioned in the introduction, the result can be stated more compactly. Again we denote by $\tau_n^{\C}$ the tautological complex $U(n)$-representation. Then we show:
\begin{Theorem}[Complexity filtration on $\upi_0$] \label{theo:picomp}For all $n\in \N$ the map $q_n:ku\to A^u_n$ induces a surjection on $\upi_0$ and the kernel is generated as a global functor by the single element  \[\tau_n^{\C}-n\cdot [1]\in \pi_0^ {U(n)} (ku).\] 
\end{Theorem}

\begin{Remark} \label{rem:stabcomp} It turns out that if $G$ is finite this filtration already stabilizes at $\pi_0^G(A_1^u)$, for algebraic reasons closely related to Artin's theorem that the complex representation ring is generated by inductions of $1$-dimensional representations. The fastest way to see that the filtration stabilizes is to make use of the fact that the representation rings form a global functor also for compact Lie groups, making use of Segal's smooth induction mentioned in Remark \ref{rem:smoothind}. It takes on the following two values:
\begin{align*} {\Ind}_{U(1)\times U(n-1)}^{U(n)}(p^*(\tau_1^{\C}))& = \tau_n^{\C} \\
								{\Ind}_{U(1)\times U(n-1)}^{U(n)}([1])& = n\cdot [1] \end{align*}
These equalities follow from the character formula (cf. \cite[page 119]{Seg68} and \cite[Prop. 2.3]{Ol98}). Hence, the class $\tau_n^{\C}-n\cdot [1]$ can be obtained by applying restriction and induction to the class $\tau_1^{\C}-[1]$ and thus lies in the global functor generated by it. Since on finite groups $\pi_0^G(ku)$ agrees with the representation ring global functor, this shows that all restrictions of $\tau_n^{\C}-n\cdot [1]$ to finite groups already lie in the global functor generated by $\tau_1^{\C}-[1]$. This also shows that if $\pi_0^G(ku)$ was the representation ring also for infinite compact Lie groups, the filtration would stabilize at stage $1$ globally.
As it stands it does not, $\tau_n^{\C}$ is identified with the trivial $n$-dimensional representation exactly in the $n$-th step, since $U(n)$ has no proper finite index subgroups.

From this it also follows that over $\R$ the filtration stabilizes at $\pi_0^G(A_2^o)$ for finite $G$. For algebraic $K$-theory this is usually not the case and the complexity filtration can take arbitrarily long to stabilize on $\pi_0^G$ (for example over $\Q$ or finite fields), as the examples in Section \ref{sec:exa} show.
\end{Remark}
Similarly to the previous section, the proof makes use of an exact sequence associated to a homotopy-cocartesian square, which this time takes the form \begin{equation} \label{eq:comppushout} \xymatrix{ \Sigma^{\infty}_+ ((E_{gl}U(n)*\oL_n)/U(n)) \ar[r]^-{p} \ar[d]_{\gamma_n} & \Sph \ar[d]^i \\
 A^u_{n-1} \ar[r]^{p_n} & A^u_n, }
\end{equation}
where $p$ is induced from the constant map $((E_{gl}U(n)*\oL_n)/U(n)\to *$. This homotopy-cocartesian square is established by forming the homotopy-pushout of three homotopy-cocartesian squares: 
\[ \xymatrix{  \Sigma^{\infty}_+(E_{gl}U(n)\times_{U(n)} \oL_n)_+ \ar[rr] \ar[d]_{\psi_n} && \Sigma^{\infty}_+(B_{gl}U(n)) \ar[d]_{\alpha_n}^{\xRightarrow{\makebox[2.8cm]{}}} & \Sigma^{\infty}_+(\oL_n/U(n)) \ar[r] \ar[d]^{\psi_n} & \mathbb{S} \ar[d]\\ 
   ku^{n-1}\ar[rr]^{i_n} & \ar@{=>}[d] & ku^n & Sp^{n-1}\ar[r]^{i_n} & Sp^n \\
\Sigma^{\infty}_+(B_{gl}U(n)) \ar[rr]_{id} \ar[d]_{\alpha_n} && \Sigma^{\infty}_+(B_{gl}U(n)) \ar[d]_{\alpha_n} \\
ku \ar[rr]^{id} && ku }
\]
The fact that the first square is a homotopy-pushout (and in particular the construction of a homotopy between the two composites) is treated in Appendix \ref{app:cofiber}. The square on the right hand side can be dealt with by the same formulas, replacing complex subspaces by natural numbers, making the upper double arrow a homotopy-coherent natural transformation between the two squares. Finally, the lower square of course commutes on the nose and the vertical double arrow can be made a homotopy-coherent transformation by using the same homotopy as in the upper square. Hence we see that there exists a cocartesian square of the form (\ref{eq:comppushout}) above.

A comparison of the associated long exact sequences shows:
\begin{Cor} \label{cor:exact2} There is an exact sequence of global functors
\[ \ker(p_*) \xr{(\gamma_n)_*} \upi_0 (A^u_{n-1})\xr{(p_n)_*} \upi_0 (A^u_n) \to \upi_0 (A^u_n/A^u_{n-1})\cong \coker(p_*)\to 0. \] 
\end{Cor}

We now consider the map $k_n:B_{gl}U(n)\to (E_{gl}U(n)*\oL_n)/U(n)$, induced by mapping $E_{gl}U(n)$ into the join. By the definition of $\gamma_n$ above, it fits into the following homotopy-commutative square:
\begin{equation} \label{eq:square} \xymatrix{ \Sigma^{\infty}_+ (B_{gl}U(n)) \ar[r]^-{\Sigma^{\infty}_+ k_n} \ar[d]_{\alpha_n} & \Sigma^{\infty}_+ ((E_{gl}U(n)*\oL_n)/U(n)) \ar[d]^ {\gamma_n} \\
	      ku\ar[r]_{q_{n-1}} & A_{n-1} }
\end{equation}
Furthermore, we have:
\begin{Lemma} For every compact Lie group $G$ the induced map  \[ \pi_0^G (B_{gl}U(n))\xr{(k_n)_*} \pi_0^G ((E_{gl}U(n)*\oL_n)/U(n)) \] is surjective.
\end{Lemma}
\begin{proof} Since the square \[ \xymatrix{ \pi_0^G (E_{gl}U(n)\times_{U(n)} \oL_n) \ar[r] \ar[d] & \pi_0^G (B_{gl} U(n)) \ar[d] \\
  \pi_0^G (\oL_n/U(n)) \ar[r] & \pi_0^G ((E_{gl}U(n)*\oL_n)/U(n))}      \]
is a pushout of sets, it suffices to show that the projection $E_{gl}U(n)\times_{U(n)} \oL_n\to \oL_n/U(n)$ induces a surjection on $\pi_0^G$. An element of $\oL_n/U(n)$ is represented by a tuple $(W_i,x_i)_{i\in I}$, where the $x_i$ are elements of a complete $G$-universe $\U_G$ and the $W_i$ form an orthogonal decomposition of $\C^n$ (such that the equalities $\sum \dim W_i \cdot x_i=0$ and $\sum \dim W_i |x_i|^2=1$ are satisfied). Since the $U(n)$-action is modded out, the represented element only depends on the $x_i$ and the dimensions of the $W_i$. That the tuple $(W_i,x_i)_{i\in I}$ is a $G$-fixed point means that every element of $G$ maps each $x_i$ to an element $x_{g(i)}$ such that $\dim W_i= \dim W_{g(i)}$.
Now let $\alpha:G\to U(n)$ be any homomorphism such that $\alpha(G)\cdot W_i=W_{g(i)}$. For example one can choose orthonormal bases $\{a_{i,k}\}$ of the $W_i$ and define $\alpha(g)(a_{i,k})=a_{g(i),k}$. Furthermore, let $\varphi:\C^n\to \U_G$ be an embedding which is equivariant for the action on $\C^n$ induced by $\alpha$. Then the tuple $(\varphi,(W_i,x_i)_{i\in I})\in (E_{gl}U(n)\times \oL_n)(\U_G)$ is a fixed point for the graph $\Gamma(\alpha)$ and hence a $G$-fixed point of the quotient. Its projection to $\oL_n/U(n)$ gives back the tuple we started with, which hence lies in the image, and so we are done.
\end{proof}
\begin{Remark} Conceptually, the main input in the proof of the previous lemma is that for every complete subgroup $L$ of $U(n)$ the projection $N_{U(n)}L\to W_{U(n)}L$ splits. 
\end{Remark}
\begin{Cor} \label{cor:gencomp} The global functor $\upi_0 (\Sigma^{\infty}_+ ((E_{gl}U(n)*\oL_n)/U(n)))$ is generated by the element $(k_n)_*(\tau_n^{\C})$. Hence, the kernel of $\gamma_n$ is generated as a global functor by the element $(k_n)_*(\tau_n^{\C}-n\cdot [1])$. 
\end{Cor}

Now we are ready for:
\begin{proof}[Proof of Theorem \ref{theo:picomp}] We make use of the exact sequence of Corollary \ref{cor:exact2}. Since $(E_{gl}U(n)*\oL_n)/U(n)$ is not the empty orthogonal space, the projection to a point splits and hence $p_*$ is surjective. It follows that $\upi_0 (A^u_n/A^u_{n-1})$ is zero and we see that all the maps $(p_n)_*:\upi_0 (A^u_{n-1})\to \upi_0 (A^u_n)$ are surjective. Hence so is the composition $(q_n)_*:\upi_0 (ku)\to \upi_0 (A^u_n)$ and we have proved the first statement.

It remains to show that the kernel of $(q_n)_*:\upi_0 (ku)\to \upi_0 (A^u_n)$ is generated by the element $\tau_n^{\C}-n\cdot [1]\in \pi_0^{U(n)}(ku)$. We proceed by induction on $n$. For $n=0$ there is nothing to show. Now let $n$ be a natural number and assume the statement to be proved for $n-1$. By Corollary \ref{cor:gencomp}, the kernel of $\upi_0 (A^u_{n-1})\to \upi_0 (A^u_n)$ is generated by the element $(\gamma_n)_* ((k_n)_*(\tau_n^{\C}-n\cdot [1]))$ which by the commutativity of Square (\ref{eq:square}) above is equal to $(q_{n-1})_*(\tau_n^{\C}-n\cdot [1])$.
Thus, by induction hypothesis we see that the kernel of $(q_n)_*$ is generated as a global functor by the elements $(\tau_n^{\C}-n\cdot [1])$ and $(\tau_{n-1}^{\C}-(n-1)\cdot [1])$. But the latter is obtained from the former by restriction along the inclusion $U(n-1)\hookrightarrow U(n)$ and it follows that $(\tau_n^{\C}-n\cdot [1])$ generates the whole kernel, so we are done. 
\end{proof}

\begin{Remark}[Filtrations associated to $ko$]
Replacing $\C$ by $\R$ in the definition of Section \ref{sec:kurank}, one obtains a model $ko$ for connective global real $K$-theory and an associated rank and complexity filtration. All the proofs go through verbatim to give the analogous statements for $ko$, replacing unitary groups $U(n)$ by their orthogonal counterparts $O(n)$.
\end{Remark}

\section{The rank filtration for the category of finite sets and the global Barratt-Priddy-Quillen theorem}
\label{sec:finsets}
In this section we describe a spectrum which represents the global $K$-theory of finite sets, together with its rank filtration. The global Barratt-Priddy-Quillen Theorem states that this spectrum is globally equivalent to the sphere spectrum and we explain how it can be derived directly from the description of the filtration quotients. The method of proving the non-equivariant Barratt-Priddy-Quillen Theorem via a rank filtration was also used in \cite{Rog92}.

For a finite pointed set $A_+$ and a real vector space $W$ we denote by $k\mathcal{F}in(W,A_+)$ the space of tuples $(M_a)_{a\in A}$ of pairwise orthogonal finite orthonormal systems $M_a$ of vectors of $W$, or in other words the space
\[  \bigsqcup_{(n_a\in \N)_{a\in A}}L_\R(\bigoplus_{a\in A} \R^{n_a},W)/{\prod_{a\in A} \Sigma_{n_a}}. \]
As in the case of $ku$, these spaces come with a filtration where the $n$-th stage $k\mathcal{F}in^n(W,A_+)$ only contains those components where the sum $\sum_{a\in A}n_a$ is smaller than or equal to $n$.

For a fixed $W$ the spaces $k\mathcal{F}in(W,A_+)$ carry a $\Gamma$-space structure by disjoint union of subsets. Again we let $W$ vary in order to implement equivariance and multiplicativity and obtain an orthogonal $\Gamma$-space by defining
\[ (V,A_+) \mapsto k\mathcal{F}in(\Sym(V),A_+). \]
We let $k\mathcal{F}in$ be the realization of this orthogonal $\Gamma$-space. Just like for $ku$, the tensor product of embeddings turns $k\mathcal{F}in$ into an ultracommutative ring spectrum.

Since the filtration described above is compatible with both the $\Gamma$-space and the orthogonal structure, we again obtain a rank filtration
\[ *=k\mathcal{F}in^0\to k\mathcal{F}in^1\to \hdots \to k\mathcal{F}in^n \to \hdots \to k\mathcal{F}in. \]
 
The quotient $k\mathcal{F}in^n/k\mathcal{F}in^{n-1}$ is the realization of the orthogonal $\Gamma$-space which sends $(V,A_+)$ to the space
\[ \bigvee_{(n_a\in \N)_{a\in A},\sum n_a=n}(L_\R(\bigoplus_{a\in A} \R^{n_a},\Sym(V))/{\prod_{a\in A} \Sigma_{n_a}})_+ \]
which by the same trick as before is homeomorphic to
\[ L_\R(\R^n,\Sym(V))_+\wedge_{\Sigma_n} (\bigvee_{(n_a\in \N)_{a\in A},\sum n_a=n} \Bij(\bigsqcup_{a\in A} \underline{n_a},n)_+).\]
The first smash factor $L(\R^n,\Sym(-))$ is untouched by the $\Gamma$-space structure. Since the permutation representation $\R^n$ of $\Sigma_n$ is faithful, this first factor is a global universal space for $\Sigma_n$. The second smash factor is the $\Sigma_n$-$\Gamma$-space of partitions of the set $\underline{n}$. But this $\Gamma$-space can be described in an easier way, it is isomorphic to the one that sends a finite pointed set $A_+$ to its $n$-fold smash product $(A_+)^{\wedge n}$.
Hence its realization is given in level $k$ by $(S^k)^ {\wedge n}$, the $n$th quotient of the symmetric product filtration before taking orbits under the $\Sigma_n$-action. By Proposition \ref{prop:sympow}, this is the suspension spectrum of the unreduced suspension of a global universal space for the collection of complete subgroups $\mathcal{C}^{\Sigma}_n$ of $\Sigma_n$. So we obtain:
\begin{Prop} The quotient $k\mathcal{F}in^n/k\mathcal{F}in^ {n-1}$ is globally equivalent to
\[ \Sigma^{\infty} ({E_{gl}\Sigma_n}_+\wedge_{\Sigma_n} (E_{gl} \mathcal{C}^{\Sigma}_n)^\dia). \]
\end{Prop}
But for $n>1$ the collection of complete subgroups of $\Sigma_n$ contains the trivial subgroup (unlike the collection of complete subgroups of $U(n)$ that appeared in the filtration quotients of $ku$). Hence the map from $E_{gl}\mathcal{C}^{\Sigma}_n$ to a point induces a global equivalence after taking $\Sigma_n$-homotopy orbits. In other words, all filtration quotients $k\mathcal{F}in^n/k\mathcal{F}in^{n-1}$ are globally trivial. Furthermore, the spectrum $k\mathcal{F}in^1$ is isomorphic to the suspension spectrum of the based orthogonal space $V\mapsto S(\Sym(V))_+$. Since the unit sphere in a complete $G$-universe is equivariantly contractible, we hence obtain:

\begin{Cor}[Global Barratt-Priddy-Quillen Theorem] \label{theo:gbpq} The unit \[ \Sph\to k\mathcal{F}in\] is a global equivalence of ultracommutative ring spectra. 
\end{Cor}
This also shows that the complexity filtration for the global $K$-theory of finite sets agrees with the symmetric product filtration (except for level $0$). As explained in the introduction, its effect on $\upi_0$ is computed in \cite{Sch14}.

\section{Filtrations associated to global algebraic $K$-theory} \label{sec:alg}

In this section we explain the modified rank and complexity filtration of the global algebraic $K$-theory spectrum of a discrete ring $R$. While the results on the filtration and also the methods to obtain them are similar to the case of topological $K$-theory, the setup is a little different. The global algebraic $K$-theory spectrum of a discrete ring, as introduced by Schwede in \cite{Sch13alg}, only forms a \emph{symmetric} spectrum and not an orthogonal spectrum, in particular it only represents a global homotopy type on finite groups. Likewise, the filtration quotients turn out to be suspension spectra of \emph{$I$-spaces} (i.e., functors from the category of finite sets and injective maps to the category of topological spaces), the symmetric analog of orthogonal spaces.
See \cite[Section I.7]{Sch15} for the definitions and global homotopy theory of $I$-spaces, and \cite{Hau15} for global homotopy theory of symmetric spectra. A lot of the theory is parallel to the orthogonal case with $G$-representations replaced by $G$-sets, with the only caveat that in general the notion of global equivalence is more complicated. However, as explained in \cite{Hau15}, there is a subclass of symmetric spectra - called globally semistable - which behave very similar to orthogonal spectra: A map between two globally semistable symmetric spectra is a global equivalence if and only if it induces an isomorphism on equivariant homotopy groups, and every globally semistable symmetric spectrum allows a global $\upi_*$-isomorphism to an orthogonal spectrum.
In Corollary \ref{cor:semistable} we argue that the symmetric spectra we encounter are globally semistable, so we can treat them just like orthogonal spectra.

In \cite{Sch13alg}, Schwede sets up a machinery which takes a category with the structure of a certain categorical $E_{\infty}$-operad and produces a symmetric spectrum. We do not repeat the general construction here, but just present the output for the case of free global algebraic $K$-theory of discrete rings. In that case the machinery produces a symmetric spectrum whose $G$-fixed points for a finite group $G$ represent direct sum algebraic $K$-theory of those $R[G]$-modules whose underlying $R$-module is free, so-called $R[G]$-lattices.

\subsection{Quotients in the modified rank filtration}
We describe a slight modification of the construction of \cite{Sch13alg}, as explained in \cite[Sec. 6.3]{Hau15}. From now on let $R$ denote a discrete ring satisfying the dimension invariance property ($R^m\cong R^n$ implies $m=n$). Let $W$ be a free $R$-module and $A_+$ a finite pointed set. We define $kR(W,A_+)$ to be the nerve of the following category: Objects are tuples of the form $(W_a)_{a\in A}$ where the $W_a$ are finite rank free $R$-submodules of $W$ such that their sum $\sum_{a\in A}W_a$ in $W$ is direct and splits off from $W$ as a direct summand (but no such splitting is part of the data). Morphisms are tuples of $R$-module isomorphisms, again indexed by $A$. Another description of this category is as the quotient category
\[ \bigsqcup_{(n_a)_{a\in A}} E(\Emb_R(\bigoplus_{a\in A} R^{n_a},W))/\prod_{a\in A} GL_{n_a}(R) \]
where $E(-)$ of a set is the category with objects the set and exactly one morphism between any two objects and $\Emb(-,-)$ is the set of splittable $R$-module monomorphisms between two $R$-modules. Again, given a natural number $n$, we can restrict the space $kR(W,A_+)$ to those components where the sum of the $n_a$ is at most $n$ and obtain a filtration via spaces $kR^n(W,A_+)$.

For fixed $W$, the assignment $A_+\mapsto kR(W,A_+)$ possesses the structure of a $\Gamma$-space by forming the inner direct sum of objects and morphisms. We obtain an $I$-$\Gamma$-space via
\[ (M,A_+)\mapsto kR(R[M],A_+) \]
where $R[M]$ denotes the polynomial ring with commuting variable set $M$. The global algebraic $K$-theory spectrum $kR$ of $R$ is the realization of this $I$-$\Gamma$-space, i.e.,
\[ kR(M)= kR(R[M],S^M) \]
with diagonal $\Sigma_M$-action. The filtration of the spaces $kR(W,A_+)$ is compatible with the $I$-$\Gamma$-space structure and hence we obtain the modified rank filtration
\[ *\to kR^1\to kR^2\to \hdots \to kR^n\to \hdots \to kR. \] If $R$ is commutative, the tensor product of modules turns $kR$ into a strictly commutative ring spectrum and all the $kR^n$ into modules over $kR^1$ (which is globally equivalent to $\Sigma^{\infty}_+ (B_{gl}R^{\times})$, as we explain below).

We proceed similarly to Section \ref{sec:kurank} to describe the filtration quotients. The $n$-th quotient $kR^n/kR^{n-1}$ is the realization of the $I$-$\Gamma$-space that is given in level $(M,A_+)$ by
\[ \bigvee_{(n_a)_{a\in A},\sum n_a=n} (|E(\Emb_R(\bigoplus_{a\in A} R^{n_a},R[M]))|/\prod_{a\in A} GL_{n_a}R)_+. \]
We can rewrite this as
\[ |E(\Emb_R(R^n,R[M]))|_+\wedge_{GL_n(R)} (\bigvee_{(n_a)_{a\in A},\sum n_a=n} \Iso(\bigoplus_{a\in A} R^{n_a},R^n)_+). \]
The first factor $E(\Emb_R(R^n,R[-]))$ is constant in the $\Gamma$-space direction and we have:
\begin{Lemma} \label{lem:alguniversal} The $GL_n(R)$-$I$-space $|E(\Emb_R(R^n,R[-]))|$ is a universal space for $GL_n(R)$.
\end{Lemma}
We first explain what we mean by this statement in the setting of $I$-spaces. In principle, a $GL_n(R)$-$I$-space $X$ should be a universal space for $GL_n(R)$ if its evaluations $X(\U_G)$ are universal spaces for the family $1_{GL_n(R)}\langle G\rangle$, where $\U_G$ is a countable $G$-set in which every finite $G$-set embeds.
The complication with $I$-spaces is that in general the evaluation $X(\U_G)$ has to be derived, in the sense that $X$ has to be replaced by a $GL_n(R)$-globally equivalent static $GL_n(R)$-$I$-space first. Here, a $GL_n(R)$-$I$-space is called \emph{static}, if for any finite group $G$ and any embedding of faithful $G$-sets $M\hookrightarrow N$ the structure map $X(M)\to X(N)$ is a $(GL_n(R)\times G)$-equivalence. In the proof below we show that $E(\Emb_R(R^n,R[-]))$ is \emph{positive} static, in the sense that this condition is satisfied in all cases except for $G$ the trivial group and $M$ the empty set. This is good enough for our purposes, since the map from a positive static $GL_n(R)$-$I$-space to its static replacement induces a $(GL_n(R)\times \Sigma_n)$-equivalence in all levels $n>0$. So we do not have to worry about deriving~$-(\U_G)$.

\begin{proof} Every structure map is a closed inclusion, since it is the realization of a functor that is injective on objects and morphisms. Now let $G$ be a finite group and $M$ a non-empty faithful finite $G$-set. We have to show that the fixed points  $|E(\Emb_R(R^n,R[M]))|^L$ for a subgroup $L\subseteq GL_n(R)\times G$ are trivial if $L$ does not lie in $1_{GL_n(R)}\langle G\rangle$ and contractible otherwise. Since any map between universal spaces for the same family is an equivariant equivalence, this shows both the claim that $E(\Emb_R(R^n,R[-]))$ is positive static and that it is a global universal space for $GL_n(R)$.
The fixed points are homeomorphic to the nerve of the category $E(\Emb_R(R^n,R[M])^L)$. Hence it suffices to show that these fixed points are empty if and only if $L$ contains a non-trivial element of $GL_n(R)\times 1$. The ``if'' part is clear, since $GL_n(R)$ acts freely on $\Emb_R(R^n,R[M])$. If $L$ does not contain such an element, it is the graph of a homomorphism from a subgroup $H$ of $G$ to $GL_n(R)$. This homomorphism defines an $H$-module structure on $R^n$ and the $L$-fixed points are given by the (non-equivariantly splittable) $H$-equivariant embeddings of $R^n$ into $R[M]$. We have to show that such an embedding always exists. Since the canonical map from induction to coinduction is an isomorphism, we obtain an $H$-equivariant map
\[ R^n\xr{\epsilon} \map(H,R^n)\xleftarrow{\cong} \bigoplus_H R^n \]
which is $R$-linearly (though not $R[H]$-linearly) split by the projection onto the component of the neutral element of $H$. In particular, $R^n$ allows an $H$-equivariant embedding into the permutation representation $\bigoplus_H R^n$. Hence it suffices to show that this permutation representation in turn sits inside $R[M]$ as a direct summand. But this follows from the observation that any monomial $\prod_{m\in M} m^{i_m}$ with all $i_m$ pairwise different spans a free $H$-subset, since $M$ is faithful. This finishes the proof, since we assumed that $M$ is non-empty.
\end{proof}
The second factor is constant in the $I$-space direction and forms the $GL_n(R)$-$\Gamma$-space of partitions of $R^n$, we denote it by $\mathcal{P}^R(n,-)$. Its realization even forms a $GL_n(R)$-orthogonal spectrum, so we see:
\begin{Cor} \label{cor:semistable} The quotients $kR^n/kR^{n-1}$ are globally semistable symmetric spectra. Hence, by induction, so are the $kR^n$.
\end{Cor}
\begin{proof} Since the $GL_n(R)$-$I$-space $E(\Emb_R(R^n,R\langle - \rangle))$ is positive static, it follows that as a $G$-symmetric spectrum $|E(\Emb_R(R^n,R\langle M \rangle))|_+\wedge_{GL_n(R)} |\mathcal{P}^R(n,-)|$ is $\upi_*$-isomorphic to the restriction of the $G$-orthogonal spectrum $E(1_{GL_n(R)}\langle G\rangle)_+\wedge_{GL_n(R)} |\mathcal{P}^R(n,-)|$, and hence $G$-semistable. As this holds for all finite $G$, it follows from \cite[Prop. 4.13 (i)]{Hau15} that the quotient is globally semistable.\end{proof}

We proceed by examining $\mathcal{P}^R(n,-)$. A point in the $M$-th level of the realization of this $\Gamma$-space is represented by a tuple $(W_i, x_i)_{i\in I}$ where the $x_i$ are elements of $\mathbb{R}^M$ and the $W_i$ are free submodules of $R^n$ whose inner sum is direct and all of $R^n$. In other words, it is the direct algebraic analog of $\mathcal{L}(n,S^M)$ of Section \ref{sec:kurank}. Many of the arguments we used there can also be applied here by formally replacing complex subspaces by free $R$-submodules. We define two subspaces (where $\rk(-)$ denotes the rank of a free $R$-module):
\[ \overline{\mathcal{P}^R(n,S^V)}=\{[(W_i, x_i)_{i\in I}]\ |\ \sum \rk(W_i)\cdot x_i =0 \} 
\]
and
\[ \overline{\mathcal{P}_{|.|=1}^R(n,S^V)}=\{[(W_i, x_i)_{i\in I}]\ |\ \sum \rk(W_i)\cdot x_i =0,\ \sum \rk(W_i)|x_i|^2=1 \}.  \]
The same arguments as in Section \ref{sec:kurank} show that the realization of $\mathcal{P}^R(n,-)$ is $GL_n(R)$-isomorphic to the suspension spectrum of the unreduced suspension of the $GL_n(R)$-$I$-space $\overline{\mathcal{P}_{|.|=1}^R(n,S^-)}$, which we abbreviate by $\oP$. Again we remark that this $GL_n(R)$-$I$-space is in fact the restriction of a $GL_n(R)$-orthogonal space and hence can be examined by the means of Section \ref{sec:orthspaces}. It turns out that not all descriptions from Section \ref{sec:kurank} for $ku$ can be carried over to this setting. We start with the one that works in full generality. Let $\mathcal{P}^R_n$ denote the poset of proper decompositions of $R^n$ into direct sums of free $R$-submodules. Then we have:
\begin{Theorem} \label{theo:alglattice} For every discrete ring $R$ satisfying dimension invariance, there is a zig-zag of $GL_n(R)$-maps between $\oP$ and the constant orthogonal space $|\mathcal{P}^R_n|$, inducing a global equivalence \[ S^1\wedge (E_{gl}GL_n(R)\times_{GL_n(R)} \oP)_+ \simeq S^1\wedge (E_{gl}GL_n(R)\times_{GL_n(R)} |\mathcal{P}^R_n|)_+.\] In particular, there is a global equivalence of symmetric spectra
\[ kR^n/kR^{n-1}\simeq \Sigma^{\infty} (E_{gl}GL_n(R)_+\wedge_{GL_n(R)} |\mathcal{P}^R_n|^{\dia}). \]
\end{Theorem}
\begin{proof} The proof is completely analogous to the one given in Section \ref{sec:lattice}. 
\end{proof}

We now discuss in which cases $\oP$ is a universal space of a collection of subgroups, for which we have to make assumptions on the ring $R$. Let $\mathcal{C}_n^R$ denote the collection of complete subgroups of $GL_n(R)$, i.e., those that are conjugate to one of the form $GL_{n_1}(R)\times \hdots \times GL_{n_k}(R)$ with $n_1+\hdots + n_k=n$ and $k>1$. Then the following still holds for all rings $R$:
\begin{Lemma} The $GL_n(R)$-$I$-space $\oP$ is closed and has all isotropy in complete subgroups. 
\end{Lemma}
\begin{proof}[Proof (cf. Proposition \ref{prop:clacompl})]An element of $GL_n(R)$ fixes each $W_i$ in a partition $R^n=\bigoplus W_i$ if and only if it lies in the product of the $GL(W_i)$, which is a complete subgroup. 
\end{proof}
One might guess that $\oP$ is in fact a universal space for the collection of complete subgroups, but this is not true in general. The issue is the following: Assume given a complete subgroup $H=\prod GL(W_i)$ and a decomposition $R^n=\bigoplus W_j'$ that is (strongly) fixed by $H$. Over the complex numbers this implied that each $W_i$ must be contained in some $W_j'$, or in other words that the decomposition $\bigoplus W_i$ is a refinement of $\bigoplus W_j'$. This allowed an easy description of the $H$-fixed point space and led to the proof that it is contractible. However, for general $R$ this is not the case, as the following example shows:
\begin{Example} Let $R=\mathbb{F}_2$ be the field with two elements and consider the decomposition $\mathbb{F}_2^2=\mathbb{F}_2\oplus \mathbb{F}_2$. Then the associated complete subgroup of $GL_2(\mathbb{F}_2)$ is $GL_1(\mathbb{F}_2)\times GL_1(\mathbb{F}_2)$ and hence trivial. So it fixes all three decompositions of $\mathbb{F}_2^2$ as a sum of two $1$-dimensional subspaces, not only the one it was associated to.
\end{Example}
Arone and Lesh showed that this phenomenon cannot occur under the following assumptions on $R$:
\begin{Lemma}[{\cite[Lemma 8.8]{AL07}}] \label{lem:dom2} Let $R$ be an integral domain with $2\neq 0$, and further be given two proper decompositions $R^n=\bigoplus {W_i}$ and $R^n=\bigoplus W_j'$ into free submodules. Then the subgroup $\prod GL(W_i)$ fixes all of the $W_j'$ if and only if $\bigoplus {W_i}$ is a refinement of $\bigoplus W_j'$.
\end{Lemma}
Under these conditions we see:
\begin{Prop} Let $R$ be an integral domain with $2\neq 0$. Then $\oP$ is a global universal space for the collection of complete subgroups of $GL_n(R)$. 
\end{Prop}
\begin{proof} We have already seen that all the $GL_n(R)$-isotropy lies in complete subgroups. Now let $G$ be a finite group and $\U_G$ a complete $G$-set universe. In Appendix \ref{app:cw} we show that $(\oP)(\U_G)$ is a $(GL_n(R)\times G)$-cell complex. Let $H\subseteq GL_n(R)\times G$ be a subgroup whose intersection with $GL_n(R)\times 1$ (which we denote by $K$) is complete. Making use of Lemma \ref{lem:dom2}, we can associate to $K$ the unique minimal partition $R^n=W_1\oplus\hdots \oplus W_k$ that is fixed by it and it follows that $K=GL(W_1)\times \hdots \times GL(W_k)$. We have to show that the $H$-fixed points $(\oP)(\U_G)^H$ are contractible. By the short exact sequence
 \[ 1\to K\to H\to {\pr}_G(H)\to 1 \]
these $H$-fixed points are the $\pr_G(H)$-fixed points of the action on $(\oP)(\U_G)^K$ which is induced from the associated group homomorphism $\pr_GH\to W_{GL_n(R)}K$. By Lemma \ref{lem:dom2}, a partition is fixed by $K$ if and only if it refines $R^n=\bigoplus W_i$. Refinements of this partition stand in bijection to partitions of the set $\{1,\hdots, k\}$. Via this correspondence we see that the $K$-fixed points $(\oP)(\U_G)^K$ are in fact homeomorphic to the $(\Sigma_{\dim W_1}\times \hdots \times \Sigma_{\dim W_k})$-fixed points of $S(\overline{\R^n}\otimes \mathbb{R}[\U_G])$.
Moreover, the Weyl-groups $W_{GL_n(R)}K$ and $W_{\Sigma_n}(\Sigma_{\dim W_1}\times \hdots \times \Sigma_{\dim W_k})$ are canonically isomorphic and the homeomorphism is equivariant under this isomorphism. Hence the statement follows from the fact that $S(\overline{\R^n}\otimes \mathbb{R}[\U_G])$ is a universal space for the collection $\mathcal{C}_n^\Sigma\langle G\rangle$, which was proved in Proposition \ref{prop:sympow}.
\end{proof}

Altogether we summarize:
\begin{Theorem}[Quotients in the modified rank filtration] Let $R$ be a ring satisfying dimension invariance. Then there is a global equivalence
\[ kR^n/kR^{n-1}\simeq \Sigma^{\infty} (E_{gl}GL_n(R)_+\wedge_{GL_n(R)} |\mathcal{P}^R_n|^{\dia}). \]
If $R$ is in addition an integral domain with $2\neq 0$, then there is also a global equivalence
\[ kR^n/kR^{n-1}\simeq \Sigma^{\infty} (E_{gl}GL_n(R)_+\wedge_{GL_n(R)} (E_{gl}\mathcal{C}^R_n)^\dia). \] 
\end{Theorem}

\subsection{Quotients in the complexity filtration}
Global algebraic $K$-theory $kR$ also allows a canonical map to $Sp^{\infty}$, mapping the modified rank filtration to the symmetric product filtration. On the level of $\Gamma$-spaces it is given by replacing a tuple $(M_a)_{a\in A}$ by its dimensions $(\dim_R(M_a))_{a\in A}$ and all automorphisms by identities. Then the analogous pushout construction as in Section \ref{sec:kucomp} gives the complexity filtration
\[ kR=A_0^R\xr{p_0} A_1^R\xr{p_1} \hdots  \to A_{\infty}^R\simeq H\Z. \]
Again there are cofiber sequences showing that $A_n^R/A_{n-1}^R$ is equivalent to the suspension spectrum of an $I$-space, the homotopy cofiber of
\[ E_{gl}(GL_n(R))_+\wedge_{GL_n(R)}|\mathcal{P}^R_n|^{\dia}\simeq E_{gl}(GL_n(R))_+\wedge_{GL_n(R)}(\oP)^{\dia}\to (B\F^{\Sigma}_n)^{\dia}\cong (\oP/GL_n(R))^{\dia}. \]
This homotopy cofiber is given by the unreduced suspension of the $GL_n(R)$-orbits of the join of $E_{gl}(GL_n(R))$ and $\oP$. For arbitrary $R$ we are not aware of a direct characterization of this global homotopy type, but under the same additional hypotheses as in the previous section we obtain the following via Lemma \ref{lem:join} :
\begin{Theorem} \label{theo:algquotcomp} Let $R$ be an integral domain with $2\neq 0$. Then there is a global equivalence
\[ A^R_n/A^R_{n-1} \simeq \Sigma^{\infty} (B_{gl}(\overline{\mathcal{C}}^R_n)^\dia), \]
where $\overline{\mathcal{C}}^R_n$ denotes the collection of complete subgroups of $GL_n(R)$ plus the trivial subgroup.
\end{Theorem}

\subsection{The modified rank filtration on $0$-th homotopy}
Now we come to the behavior of these filtrations on $\upi_0$, starting with the rank filtration. Every $R[G]$-module $W$ which is free as an $R$-module gives rise to an element $[W]$ of $\pi_0^G (kR)$ represented by the map $S^M\to kR(M);v\mapsto (\varphi(M),v)$, where $\varphi:W\to R[M]$ is an equivariant embedding for some finite $G$-set $M$ (which always exists, cf. the proof of Lemma \ref{lem:alguniversal}). Like for $ku$ this element only depends on the isomorphism type of $W$ and is independent of the other choices. Furthermore, it is already defined in $\pi_0^G (kR^{\rk W})$.
The assignment $W\mapsto [W]$ is additive with respect to direct sum of $R[G]$-modules and induces an isomorphism from the representation ring $\Rep_R(G)$ of $R[G]$-modules that are finitely generated free as an $R$-module to $\pi_0^G(kR)$. Furthermore, it takes restrictions of representations to homotopy-theoretic restrictions and inductions to homotopy-theoretic transfers.

The intermediate terms $\upi_0(kR^n)$ allow a similar formula as the ones for topological $K$-theory. We first treat the case where $R$ is a finite ring. Let $\tau_n^R$ be the tautological $GL_n(R)$-module of $R$-rank $n$. Then analogously to the previous section we define universal relations $a(k,l)$ and $b(i,j)$ by \[ {\res}_{GL_k(R)\times GL_l(R)}^{GL_{k+l}(R)}(\tau_n^R)=p_1^*(\tau_k^R) + (p_2)^*(\tau_l^R), \]
respectively
\[  {\res}_{\Sigma_i\wr GL_j(R)}^{GL_{i\cdot j}(R)}(\tau_{i\cdot j}^R)={\tr}_{GL_j(R)\times (\Sigma_{i-1}\wr GL_j(R))}^{\Sigma_i\wr GL_j(R)}(p^*(\tau_j^R)). \]
And we obtain:
\begin{Theorem}[Modified rank filtration on $\upi_0$] \label{theo:algpirank} Let $R$ be a finite ring. Then the $\F in$-global functor $\upi_0 (kR^n)$ is isomorphic to the free global functor generated by the elements $\tau_1^R,\hdots,\tau_n^R$ modulo the relations $a(k,l)$ for all $k+l\leq n$ and $b(i,j)$ for all $i\cdot j\leq n$.
\end{Theorem}
Here, a \emph{$\F in$-global functor} is a global functor which is only defined on finite groups, also called \emph{inflation functor} in \cite{We93}.
\begin{proof} The proof is analogous to that of Theorem \ref{theo:pirank}, using the equivalence of the filtration quotients to the global homotopy orbits of the decomposition lattice of Theorem \ref{theo:alglattice}.
\end{proof}

The problem with infinite $R$ is that its general linear groups are not finite. Hence the universal elements above do not make sense, as the theory is ``not global enough'' to include infinite discrete groups. One can still give the following concrete description:
\begin{Prop}[Description for arbitrary rings] \label{prop:algpirank2} Let $R$ be a ring satisfying dimension invariance. Then $\upi_0(kR^n)$ is generated as a $\F in$-global functor by the elements $[W]\in \pi_0^G(kR)$, where $(G,W)$ runs through a system of representatives of isomorphism classes of a finite group $G$ together with an $R[G]$-module that is free of rank $\leq n$ over $R$. The relations are generated by:
\begin{itemize}
\item $[W\oplus W']=[W]\oplus [W']$ with $\dim_R(W)+\dim_R(W')\leq n$. 
\item $[\Ind_H^G(W)]=\tr_H^G[W]$ with $\dim_R(W)\cdot [G:H]\leq n$.
\end{itemize}
\end{Prop}
In this case both the generators and the relations are already closed under restrictions, so it suffices to apply transfers to obtain the concrete value at a given finite group. In this sense the global formula is no easier than the one for a specific group.
\subsection{The complexity filtration on $0$-th homotopy} We again start with the case of finite $R$, where formula and proof are analogous to the one for topological $K$-theory.
\begin{Theorem}[Complexity filtration on $\upi_0$] \label{theo:algpicomp} Let $R$ be a finite ring and $n\in \N$. Then the map $(q_n)_*:\Rep_R(-)\cong \upi_0(kR)\to \upi_0(A_n^R)$ is surjective with kernel generated as a $\F in$-global functor by the element
\[ \tau_n^R-n\cdot [1] \in \pi_0^{GL_n(R)} (kR)\cong {\Rep}_R(GL_n(R)). \]
\end{Theorem}
For general $R$ these universal elements are again not part of the theory and hence there is no such compact description. The result then reads as follows:
\begin{Prop}[Description for arbitrary rings] \label{prop:algpicomp2} Let $R$ be a ring satisfying dimension invariance and $n\in \N$. Then the map $\Rep_R(-):\upi_0 (kR)\to \upi_0 (A^R_n)$ is surjective with kernel generated as a $\F in$-global functor by the elements
\[ [W]-n\cdot [1]\in \pi_0^G(kR)\cong {\Rep}_R(G) \]
where $(G,W)$ runs through isomorphism classes of pairs with $W$ an $n$-dimensional $G$-representation over $R$.
\end{Prop}
The proof of the two statements is the same as that of Theorem \ref{theo:picomp}. It uses that the inclusion $B_{gl}(GL_n(R))\hookrightarrow ((E_{gl}GL_n(R)*\oP)/GL_n(R))$ induces a surjection on $\upi_0$, even if $\upi_0 (B_{gl}GL_n(R))$ is not generated by a single element.

\section{Examples}
\label{sec:exa}
In this final section we give examples for the effect on $\pi_0^G$ of the modified rank and complexity filtrations for various finite groups $G$ and topological or discrete rings. The purpose is two-fold: On the one hand we want to explain how one computes concrete values $\pi_0^G(X)$ for the global spectra $X$ that appeared in this paper, using the global formulas we gave in the previous sections. On the other hand we try to demonstrate that the behavior at a specific group is often quite complicated, while the global formula is not.
 
\begin{Example}[The symmetric group $\Sigma_3$, over $\C$] We begin by going through the example $G=\Sigma_3$ in detail. To compute the values of the modified rank and complexity filtration for $\Sigma_3$, we need to know its subgroups, their complex representation rings (together with the conjugation action) and the induction maps between them. The conjugacy classes of subgroups are given by the trivial group $\{e\}$, the cyclic groups $C_2$ (represented by any transposition) and $C_3$ (the normal subgroup of $3$-cycles), and the whole group $\Sigma_3$. Their representation rings are:
\begin{eqnarray*} {\Rep}_{\C}(\{e\})\cong \Z\{[1]\} && {\Rep}_{\C}(C_2)\cong \Z\{[1],[-1]\} \\
									{\Rep}_{\C}(C_3)\cong \Z\{[1],[\eta_3],[\eta_3^2]\} && {\Rep}_{\C}(\Sigma_3)\cong \Z\{[1],[\sgn],[\nu_3]\},
\end{eqnarray*}
where $\nu_3$ is the $2$-dimensional reduced natural representation and $\eta_3$ is a primitive third root of unity. The $C_3$-representations $\eta_3$ and $\eta_3^2$ are conjugate under the Weyl-group action. Furthermore, we have the following formulas for induction:
\begin{eqnarray*}	{\Ind}_{\{e\}}^{C_2}([1])=[1]+[-1] & \hspace{1cm}  {\Ind}_{\{e\}}^{C_3}([1])=[1]+[\eta_3]+[\eta_3^2] \hspace{1cm}  & {\Ind}_{\{e\}}^{\Sigma_3}([1])=[1]+[\sgn]+2\cdot [\nu_3]\\
									{\Ind}_{C_2}^{\Sigma_3}([1])=[1]+[\nu_3] & {\Ind}_{C_2}^{\Sigma_3}([-1])=[\sgn]+[\nu_3] \\
									{\Ind}_{C_3}^{\Sigma_3}([1])=[1]+[\sgn] & {\Ind}_{C_3}^{\Sigma_3}([\eta_3])=[\nu_3]\end{eqnarray*}
To compute the first term $\pi_0^{\Sigma_3}(ku^1)$ we need to consider all transfers of $1$-dimensional representations (modulo the respective Weyl group actions) so we see that it is given by the free abelian group \[ \Z\{[1],[\sgn],{\tr}_{C_3}^{\Sigma_3}([1]),{\tr}_{C_3}^{\Sigma_3}([\eta_3]),{\tr}_{C_2}^{\Sigma_3}([1]),{\tr}_{C_2}^{\Sigma_3}([-1]),{\tr}_{\{e\}}^{\Sigma_3}([1])\}.\]
For the second stage we add on everything that comes from a $2$-dimensional irreducible representation (since, using the relation $a(1,1)$ of Theorem \ref{theo:pirank}, we can replace a non-irreducible representation by the homotopy-theoretic sum of its summands). In this case there is only one $2$-dimensional irreducible representation, the $\Sigma_3$-representation $\nu_3$. Taking into account the relation $b(2,1)$ we furthermore have to identify all representations that are at most $2$-dimensional and an induction over a proper subgroup with the homotopy-theoretic transfer of that respective representation, 
transferred up to the whole group $\Sigma_3$ if necessary. Considering the formulas for induction above, this means that we have to identify the following:
\begin{eqnarray*} {\tr}_{\{e\}}^{\Sigma_3}([1])={\tr}_{C_2}^{\Sigma_3}({\tr}_{\{e\}}^{C_2}([1])) & \text{with} & {\tr}_{C_2}^{\Sigma_3}([1])+{\tr}_{C_2}^{\Sigma_3}([-1)) \\
									{\tr}_{C_3}^{\Sigma_3}([1]) & \text{with} & [1]+[\sgn] \\
									{\tr}_{C_2}^{\Sigma_3}([\eta_3]) & \text{with} & [\nu_3] \end{eqnarray*}
So we see that $\pi_0^{\Sigma_3}(ku^2)$ is a free group with basis $\{[1],[\sgn],[\nu_3],{\tr}_{C_2}^{\Sigma_3}([1]),{\tr}_{C_2}^{\Sigma_3}([-1])\}$. Since there are no irreducible representations of dimension $3$ or higher for any of the subgroups of $\Sigma_3$, we from now on do not add any new generators but only have to take into account new relations. In the third step the universal relation $b(3,1)$ shows that ${\tr}_{C_2}^{\Sigma_3}([1])$ is identified with $[1]+[\nu_3]$ and ${\tr}_{C_2}^{\Sigma_3}([-1])$ with $[\sgn]+[\nu_3]$. Hence, $\pi_0^{\Sigma_3}(ku^3)$ is isomorphic to $\Z\{[1],[\sgn],[\nu_3]\}$, which is the representation ring of $\Sigma_3$, and the rank filtration is constant from then on.

For every finite group $G$ the complexity filtration over $\C$ stabilizes (on $\pi_0^G(-)$) at stage $1$, as noted in Remark \ref{rem:stabcomp}. For $\Sigma_3$ this can be seen concretely as follows: Recall that we start with $\pi_0^{\Sigma_3}(ku)$, the representation ring, a free group on the classes $[1],[\sgn]$ and $[\nu_3]$. As the sign representation is one-dimensional, it is identified with $[1]$ in $\pi_0^{\Sigma_3}(A_1^u)$.
Furthermore, $\nu_3$ is the induction of the $1$-dimensional $C_3$-representation $\eta_3$. So, since $[\eta_3]$ is identified with the trivial representation, application of ${\Ind}_{C_3}^{\Sigma_3}(-)$ shows that $[\nu_3]$ is identified with $[1]+[\sgn]$. We already argued that $[\sgn]$ is identified with $[1]$, so this shows that $[\nu_3]$ becomes $2\cdot [1]$ in $\pi_0^{\Sigma_3}(A_1^u)$, which is hence isomorphic to $\Z$.
\end{Example}
\begin{Example}[The symmetric group $\Sigma_3$, over $\R$] We again discuss the symmetric group on $3$ letters, this time over $\R$. Though the representation rings over $\R$ and $\C$ are isomorphic, like for every symmetric group, the effect of the modified rank and complexity filtrations on $\pi_0^{\Sigma_3}$ differ, as we now see. Again we have to start with the representation rings over all subgroups, and the only difference to the complex case is at the subgroup $C_3$, where $\Rep_{\Q}(C_3)$ only has rank $2$ with basis $[1]$ and the reduced regular representation $[\overline{\rho}_{C_3}]$.
Consequently, we find that $\pi_0^{C_3}(ko^1)$ has one basis element less, it is the free group on $\{[1],[\sgn],{\tr}_{C_3}^{\Sigma_3}([1]),{\tr}_{C_2}^{\Sigma_3}([1]),{\tr}_{C_2}^{\Sigma_3}([-1]),{\tr}_{\{e\}}^{\Sigma_3}([1])\}$. In the next step the irreducible representations $[\nu_3]$ and ${\tr}_{C_3}^{\Sigma_3}[\overline{\rho}_{C_3}]$ are added, and again ${\tr}_{\{e\}}^{\Sigma_3}([1])$ is identified with ${\tr}
_{C_2}^{\Sigma_3}([1])+{\tr}_{C_2}^{\Sigma_
3}([-1))$, as well as ${\tr}_{C_3}^{\Sigma_3}([1])$ with $[1]+[\sgn]$. This gives \[ \pi_0^{\Sigma_3}(ko^2)\cong \Z\{[1],[\sgn],[\nu_3],{\tr}_{C_3}^{\Sigma_3}[\overline{\rho}_{C_3}],{\tr}_{C_2}^{\Sigma_3}([1]),{\tr}_{C_2}^{\Sigma_3}([-1])\}.\]
In the third step the latter two classes are identified with $[1]+[\nu_3]$ and $[-1]+[\nu_3]$ respectively, hence they become algebraically dependent of the first three. Furthermore, applying ${\tr}_{C_3}^{\Sigma_3}(-)$ to the relation ${\tr}_{\{e\}}^{C_3}([1])=[1]+[\overline{\rho}_{C_3}]$ (plus using earlier relations) shows that ${\tr}_{C_3}^{\Sigma_3}([\overline{\rho}_3])$ represents the same class as $2\cdot [\nu_3]$, so $\pi_0^{\Sigma_3}(ko^3)$ is isomorphic to the representation ring $\Rep_{\R}(\Sigma_3)$.

The complexity filtration is also different to the complex one: In the first step $[1]$ and $[\sgn]$ are identified, but this time there are no further relations. This is because applying ${\Ind}_{C_2}^{\Sigma_3}$ to the identification $[1]\sim [-1]$ contributes nothing new, and there is only one $1$-dimensional representation of $C_3$ over $\R$. So $\pi_0^{\Sigma_3}(A_1^o)\cong \Z\{\overline{[1]},\overline{[\nu_3]}\}$. In the second step $[\nu]$ is identified with $2\cdot [1]$ and hence $\pi_0^{\Sigma_3}(A_2^o)\cong \Z$, which is true for all $\pi_0^G(A_2^o)$ with $G$ finite.
\end{Example}

\begin{Example}[Cyclic groups of prime order] Having seen the general algorithm, we now go back to the easiest example and use it to illustrate the behavior over various rings. Let $p$ be a prime and $C_p$ the cyclic group with $p$ elements.

\textbf{Over $\C$}: The irreducible $\C[C_p]$-representations are given by $\eta_p^1,\eta_p^2,\hdots,\eta_p^p$, where $\eta_p$ is a primitive $p$-th root of unity. So we find that $\pi_0^{C_p}(ku^n)\cong \Z\{[\eta_p^1],[\eta_p^2],\hdots,[\eta_p^p],\tr_1^{\Z/p}[1]\}$ for $0<n<p$ and $\Z\{[\eta_p^1],[\eta_p^2],\hdots,[\eta_p^p]\}$ for all $n\geq p$. As mentioned before, $\pi_0^G(A_n^u)$ is isomorphic to $\Z$ for all $n\geq 1$ and any finite group $G$.

\textbf{Over $\R$}: If $p$ is $2$, all complex representations are already defined over the reals, so the filtrations are the same. If $p$ is odd, there are $(p-1)/2$ isomorphism classes of $2$-dimensional indecomposable representations, which can be expressed as the underlying real representations of $\eta_p^1,\hdots,\eta_p^{(p-1)/2}$, plus the trivial $1$-dimensional one. So we find that 
\[\pi_0^{\C_p}(ko^n)\cong \begin{cases} 	\Z\{[1],{\tr}_{\{e\}}^{C_p}[1]\} & \text{for } n=1 \\
																					\Z\{[1],{\res}_\R^\C(\eta^1_p),\hdots,{\res}_\R^\C(\eta^{(p-1)/2}_p),{\tr}_{\{e\}}^{C_p}[1]\} & \text{for } 1<n<p \\
																					\Z\{[1],{\res}_\R^\C(\eta^1_p),\hdots,{\res}_\R^\C(\eta^{(p-1)/2}_p)\} & \text{for } n\geq p. \end{cases} \]																				
Furthermore, $\pi_0^{\C_p}(A_1^o)\cong \pi_0^{\C_p}(ko)\cong {\Rep}_{\R}(C_p)$ since there are no non-trivial one-dimensional representations, and $\pi_0^{\C_p}(A_n^o)\cong \Z$ for all $n>1$.

\textbf{Over $\Q$}: There are only two isomorphism classes of irreducible $C_p$-representations over $\Q$, the trivial $1$-dimensional representation and the reduced regular representation $\overline{\rho}_{C_p}$ of dimension $p-1$. Hence we find that \[ \pi_0^{C_p}(k\Q^n)\cong \begin{cases} \Z\{[1],{\tr}_{\{e\}}^{C_p}[1]\} & \text{for }0<n<p-1 \\
																								\Z\{[1],[\overline{\rho}_{C_p}],{\tr}_{\{e\}}^{C_p}[1]\} & \text{for } n=p-1 \\																							\Z\{[1],[\overline{\rho}_{C_p}]\} & \text{for } n\geq p. \end{cases}	\]
Furthermore, we see that 
\[\pi^{C_p}_0(A_n^{\Q})\cong \begin{cases} \Z\{[1],[\overline{\rho}_{C_p}]\} & \text{for } 0\leq n\leq p-1 \\ \Z & \text{for } n\geq p. \end{cases} \]
In particular, the complexity filtration over $\Q$ does not stabilize globally on $\upi_0$.

\textbf{Over $\mathbb{F}_p$}: Unlike in characteristic $0$ the group ring $\mathbb{F}_p[C_p]\cong \mathbb{F}_p[t]/(t^p-1)\cong \mathbb{F}_p[t]/(t-1)^p\cong \mathbb{F}_p[t]/t^p$ is no longer semisimple. Up to isomorphism, there is exactly one indecomposable representation $V_i$ in every dimension $1\leq i\leq p$ (and none in higher dimensions) and every representation decomposes uniquely as a sum of these. So we see that

\[\pi_0^{C_p}((k\mathbb{F}_p)^n)\cong \begin{cases} \Z\{[V_1],\hdots,[V_n],\tr^{C_p}_{\{e\}}[1]\}  & \text{for }n=1,\hdots,p-1 \\
																										\Z\{[V_1],\hdots,[V_p]\} & \text{for } n\geq p,
\end{cases} \]
where the cases $p-1$ and $p$ are only notationally different, since the map $\pi_0^{C_p}((k\mathbb{F}_p)^{p-1})\to \pi_0^{C_p}((k\mathbb{F}_p)^p)$ sends $\tr^{C_p}_{\{e\}}[1]$ to $[V_p]$. For the complexity filtration we obtain:
\[ \pi_0^{C_p}(A_n^{\mathbb{F}_p}) \cong \begin{cases} \Z\{\overline{[V_{n+1}]},\hdots,\overline{[V_p]}\}  & n=0,\hdots,p-1 \\
																										\Z & n\geq p
\end{cases} \]
\end{Example}

Finally we compute the complexity filtration of the alternating group $A_5$ over $\Q$. To achieve this we first need two preparatory examples:
\begin{Example}[Complexity filtration of the alternating group $A_4$ over $\Q$]
The representation ring is given by $\Rep_{\Q}(A_4)=\Z\{[1],[\eta],[\nu_4]\}$, where $\eta$ is of dimension $2$ and $\nu_4$ is of dimension $3$. There are two conjugacy classes of maximal subgroups, the alternating group $A_3$ and the Klein four-group $K$, with representation rings $\Rep_{\Q}(A_3)=\Z\{[1],[\overline{\rho_{A_3}}]$ respectively $\Rep_{\Q}(K)\cong \Z\{[1],[\varphi_1],[\varphi_2],[\varphi_3]\}$. The $\varphi_i$ are all $1$-dimensional and conjugate under the action of the Weyl group in $A_4$. We have
\begin{align*} {\Ind}_{K}^{A_4}([1])=[1]+[\eta] \hspace{1cm} \text{ and } \hspace{1cm} {\Ind}_{K}^{A_4}([\varphi_i])=[\nu_4]. \end{align*}
So $[\nu_4]$ is identified with $[1]+[\eta]$ in $\pi_0^{A_4}(A_1^{\Q})$ and this is the only relation (since $A_3$ has only one $1$-dimensional representation). Hence $\pi_0^{A_4}(A_1^{\Q})\cong \Z\{[1],[\eta],[\nu_4]\}/([1]+[\eta]-[\nu_4])\cong \Z\{\overline{[1]},\overline{[\eta]}\}$. In $\pi_0^{A_4}(A_2^{\Q})$ the representation $[\eta]$ is identified with $2\cdot [1]$, so the filtration becomes constant $\Z$ from then on.

\end{Example}
\begin{Example}[Complexity filtration of the dihedral group $D_5$ over $\Q$] The representation ring of $D_5$ is given by $\Rep_{\Q}(D_5)\cong \Z\{[1],[-1],[\psi],[(-1)\cdot \psi]\}$, where $[-1]$ is restricted from the projection $D_5\to D_5/C_5\cong C_2$. The $4$-dimensional irreducible representations $[\psi]$ and $[(-1)\cdot \psi]$ are characterized by
\[ {\Ind}_{C_2}^{D_5}([1])=[1]+[\psi] \hspace{1cm} \text{ and } \hspace{1cm} {\Ind}_{C_2}^{D_5}([-1])=[-1]+[(-1)\cdot \psi]. \]
Hence we see that the kernel of $\Rep_{\Q}(D_5)\to \pi_0^{D_5}(A_1^{\Q})$ is generated by $[1]-[-1]$ and $[1]+[\psi]-[-1]-[(-1)\cdot \psi]$, which can be simplified to $[1]-[-1]$ and $[\psi]-[(-1)\cdot \psi]$. So $\pi_0^{D_5}(A_1^{\Q})$ is free of rank $2$ with basis the classes of $[1]$ and $[\psi]$. Since there are no $2$- or $3$-dimensional irreducible representations for any subgroup of $D_5$, this is also the case for $\pi_0^{D_5}(A_2^{\Q})$ and $\pi_0^{D_5}(A_3^{\Q})$. In $\pi_0^{D_5}(A_4^{\Q})$ we have the relation $[\psi]-4\cdot [1]$, so the filtration stabilizes.
\end{Example}
\begin{Example}[Complexity filtration of the alternating group $A_5$ over $\Q$] The representation ring is given by $\Rep_{\Q}(A_5)\cong \Z\{[1],[\nu_5],[\psi],[\Lambda^2 \nu_5]\}$, where $\nu_5$ is the restriction of the reduced natural $\Sigma_5$-representation, $\Lambda^2 \nu_5$ is its $6$-dimensional exterior square and $[\psi]$ is $5$-dimensional. There are $3$ conjugacy classes of maximal subgroups given by $A_4$, $\Sigma_3$ (generated by $(123)$ and $(12)(45)$)  and $D_5$ (generated by $(1234)$ and $(13)$). We note that the rational complexity filtration for $\Sigma_3$ is the same as the one over $\R$, since all real representations of its subgroups are already defined rationally. Using the notation from the previous examples, we have
\begin{eqnarray*}	{\Ind}_{A_4}^{A_5}([1])=[1]+[\nu_5] & \hspace{1cm}{\Ind}_{A_4}^{A_5}([\eta])=2\cdot [\psi]\hspace{1cm} & {\Ind}_{A_4}^{A_5}(\nu_4)=[\nu_5]+[\psi]+[\Lambda^2\nu_5]\\
									{\Ind}_{\Sigma_3}^{A_5}([1])=[1]+[\nu_5]+[\psi] & {\Ind}_{\Sigma_3}^{A_5}([\sgn])=[\nu_5]+[\Lambda^2\nu_5] & {\Ind}_{\Sigma_3}^{A_5}([\nu_3])=[\nu_5]+2\cdot [\psi] + [\Lambda^2\nu_5] \\
									{\Ind}_{D_5}^{A_5}([1])=[1]+[\psi] & {\Ind}_{D_5}^{A_5}([-1])=[\Lambda^2\nu_5] & {\Ind}_{D_5}^{A_5}([\varphi])=2\cdot [\nu_5]+2\cdot [\psi] + [\Lambda^2 \nu_5]\end{eqnarray*}
and ${\Ind}_{D_5}^{A_5}([(-1)\cdot \varphi])=2\cdot [\nu_5]+2\cdot [\psi] + [\Lambda^2 \nu_5]$. From our previous calculations we know that the relations in $\pi_0^G(A_1^{\Q})$ are generated by $[1]+[\eta]-[\nu_4]$ for $G=A_4$, by $[1]-[\sgn]$ for $G=\Sigma_3$ and by $[1]-[-1]$ and $[\psi]-[(-1)\cdot \psi]$ for $G=D_5$. Applying inductions to these relations, we see that they only give the relation $([\psi]+[1]-[\Lambda^2\nu_5])$ in $\pi_0^{A_5}(A_1^{\Q})$. From this we can read off that \[\pi_0^{A_5}(A_1^{\Q})\cong {\Rep}_{\Q}(A_5)/([\psi]+[1]-[\Lambda^2\nu_5]),\]
hence it is free with basis $[1]$, $[\nu_5]$ and $[\psi]$. In step $2$ we have to add the inductions of the relations $2\cdot [1]-[\eta]$ for $A_4$ and $2\cdot [1]-[\nu_3]$ for $\Sigma_3$. This yields the new relation $[1]+[\nu_5]-[\psi]$, so
\[\pi_0^{A_5}(A_2^{\Q})\cong {\Rep}_{\Q}(A_5)/([\psi]+[1]-[\Lambda^2\nu_5],[1]+[\nu_5]-[\psi])\]
is free with basis $[1]$ and $[\nu_5]$. In the third step nothing happens, because $A_5$ has no $3$-dimensional irreducible representation and we have seen that there are no new relations for any of the maximal subgroups. At step $4$ the element $[\nu_5]$ is identified with $4\cdot [1]$ and so $\pi_0^{A_5}(A_n^{\Q})\cong \Z$ for all $n\geq 4$.
\end{Example}

\section{Appendix}
\subsection{Cofibrancy properties of the rank filtration} \label{app:inclusions}
In this appendix we show that the $V$-th level of the inclusions $ku^{n-1}\to ku^n$ is an $O(V)$-cofibration, guaranteeing that the quotient $ku^n/ku^{n-1}$ has the global homotopy type of the homotopy cofiber. For instance, this was used in the proof of Theorem \ref{theo:pirank}. For finite subgroups of $O(V)$ (and hence for the $\mathcal{F}in$-global homotopy type of the quotient) this would follow quite directly from the results of \cite{Ost14}, but we need the general statement. In this and the next appendix we repeatedly make use of a theorem due to Illman (cf. \cite{Ill83}) that states that every smooth manifold equipped with a smooth action by a compact Lie group allows the structure of an equivariant CW complex.

We recall from \cite[Sec. 3]{Lyd99} that the evaluation $X(A)$ of a $\Gamma$-space $X$ on a based space $A$ is naturally filtered by skeleta $sk_m(X(A))$. The $m$-skeleton is obtained from the $m-1$-st by forming a certain pushout (\cite[Thm. 3.10]{Lyd99}). Furthermore, given a map $i:X\to Y$ of $\Gamma$-spaces, one can define relative skeleta $sk_m[i](A)$ by $sk_m(Y(A))\cup_{sk_m(X(A))} X(A)$ and it follows that these are related by a similar pushout square. The colimit over the $sk_m[i](A)$ gives back $Y(A)$ and the map from $X(A)=sk_0[i](A)$ agrees with $i$.
Now let $V$ be a finite dimensional real inner product space. We are interested in the case where $A$ is equal to $S^V$ and $i$ is the inclusion $k^{n-1}(\Sym(\C \otimes V),-)\hookrightarrow k^n(\Sym(\C\otimes V),-)$. There the connecting pushout takes the form

\begin{equation} \label{eq:pushout} \xymatrix{ (\bigvee_{n_1+\hdots+n_m=n}(L_\C(\bigoplus \C^{n_i},\Sym(\C\otimes V))/\prod U(n_i))_+)\times_{\Sigma_m} F((S^V)^{\times m}) \ar[r] \ar[d] & sk_{m-1}[i](S^V) \ar[d]\\
 (\bigvee_{n_1+\hdots+n_m=n}(L_\C(\bigoplus \C^{n_i},\Sym(\C\otimes V))/\prod U(n_i))_+)\times_{\Sigma_m} (S^V)^{\times m} \ar[r] & sk_m[i](S^V),}
\end{equation}
where the wedge is indexed over all $m$-tuples $(n_1,\hdots,n_m)$ which add up to $n$, with all $n_i$ larger than $0$. The notation $F((S^V)^{\times m})$ stands for the subspace of $(S^V)^{\times m}$ of tuples which contain two equal entries or a basepoint. It suffices to show that $sk_{m-1}[i](S^V)\to sk_m[i](S^V)$ is an $O(V)$-cofibration for all $m\in \N$, since the sequential colimit of $O(V)$-cofibrations is again an $O(V)$-cofibration. This follows from:

\begin{Lemma} \label{lem:cofibration1}
The left hand vertical map in Diagram (\ref{eq:pushout}) is an $O(V)$-cofibration.
\end{Lemma}
\begin{proof} We first argue that $\bigvee_{n_1+\hdots+n_m=n}(L_\C(\bigoplus \C^{n_i},\Sym(\C\otimes V))/{\prod U(n_i)})_+$ is $(\Sigma_m\times O(V))$-cofibrant. This would follow directly from Illman's theorem if we put $W_k=\bigoplus_{i=0,\hdots,k}\Sym^i(\C\otimes V)$ instead of the full $\Sym(\C\otimes V)$, since each $L_\C(\bigoplus \C^{n_i},W_k)$ is a smooth manifold with a smooth action by $U(W_k)\times (N_{U(n)}\prod U(n_i))$. The subspace $L_\C(\bigoplus \C^{n_i},W_{k-1})$ is exactly the space of $U(\Sym^k(\C\otimes V))$-fixed points under this action. Since $O(V)$ fixes $\Sym^k(\C\otimes V)$, it normalizes the subgroup $U(\Sym^k(\C\otimes V))$.
This implies that if we forget any $(U(W_k)\times (N_{U(n)}\prod U(n_i)))$-CW structure to an $(O(V)\times (N_{U(n)}\prod U(n_i)))$-cell structure, the space $L_\C(\bigoplus \C^{n_i},W_{k-1})$ is necessarily a subcomplex and hence the inclusion a cofibration. By quotienting out the $\prod U(n_i)$-actions and passing to the colimit we see that the wedge is $(\Sigma_m\times O(V))$-cofibrant, as claimed.

Hence it suffices to show that $F((S^V)^{\times m})\to (S^V)^{\times m}$ is an $(\Sigma_m\times O(V))$-cofibration. By once more applying Illman's theorem we see that $S^V$ is an $O(V)$-CW complex, containing the basepoint $\infty$ as a $0$-cell. This $O(V)$-CW structure induces a $(\Sigma_m\wr O(V))$-CW structure on the $m$-fold cartesian product $(S^V)^{\times m}$ and thus in particular a $(\Sigma_m\times O(V))$-cell structure by choosing $(\Sigma_m\times O(V))$-CW structures on the $(\Sigma_m\wr O(V))$-orbits. We now claim that $F((S^V)^{\times m})$ is a $(\Sigma_m\times O(V))$-subcomplex for this cell structure.
By definition, $F((S^V)^{\times m})$ is the union of two subspaces: The space of tuples containing a basepoint and the space of tuples containing two equal entries. By definition, the former is even a $(\Sigma_m\wr O(V))$-CW subcomplex of $(S^V)^{\times m}$, since the basepoint is a $0$-cell. But the latter is given precisely by those points that have non-trivial $\Sigma_m$-isotropy, hence it is an equivariant subcomplex for any $(\Sigma_m\times O(V))$-cell structure. This finishes the proof.
\end{proof}

\subsection{Equivariant CW structures}
\label{app:cw}
The content of this appendix is to show that the $U(n)$-orthogonal spaces $\oL_n$ that appeared in Section \ref{sec:kurank} give $(U(n)\times G)$-cell complexes when evaluated on any $G$-representation $V$ (at most countably infinite dimensional). This property was needed in Proposition \ref{prop:clacompl} for $\oL_n$ to be a global universal space for the family of complete/non-isotypical subgroups of $U(n)$. The same proof also shows that the spaces $\oP(V)$ arising in the filtrations of algebraic $K$-theory are $(GL_n(R)\times G)$-cell complexes.

The proof is similar to that of the previous section. This time we consider the (absolute) skeleta filtration for the $U(n)$-$\Gamma$-space $\mathcal{L}(n,-)$, where the relating pushouts take the form

\[ \xymatrix{ \bigvee_{n_1+\hdots+n_m=n}(L_\C(\bigoplus \C^{n_i},\C^ n)/{\prod U(n_i)})_+\times_{\Sigma_m} F((S^V)^{\times m}) \ar[r] \ar[d] & sk_{m-1} (\mathcal{L}(n,S^V)) \ar[d]\\
  \bigvee_{n_1+\hdots+n_m=n}(L_\C(\bigoplus \C^{n_i},\C^ n)/{\prod U(n_i)})_+\times_{\Sigma_m} (S^V)^{\times m} \ar[r] & sk_m(\mathcal{L}(n,S^V)).}
\]
The wedge is taken over the same indexing system as in the previous section. We recall that the closed subspace $\oL_n(V)$ of $\mathcal{L}(n,S^V)$ was defined to contain those elements that can be represented by a tuple $(W_i,x_i)_{i\in I}$ with all $x_i$ non-equal to the basepoint and satisfying the equations $\sum \dim(W_i)\cdot x_i=0$ and $\sum \dim(W_i)|x_i|^2=1$. Intersection with $sk_m (\mathcal{L}(n,S^V))$ gives subspaces $sk_m (\oL_n(V))$ whose colimit over $m$ is isomorphic to $\oL_n(V)$. Likewise, for fixed $n_1,\hdots,n_m$ we define closed subspaces $S_{\{n_i\}}((S^V)^{\times m})\subseteq (S^V)^ {\times m}$ as those tuples satisfying $\sum n_i \cdot x_i=0$ and $\sum n_i |x_i|^2=1$.
With these definitions an element of $L(\bigoplus \C^{n_i},\C^ n)/{\prod U(n_i)}\times (S^V)^ {\times m}$ is mapped to $sk_m(\oL_n(V))$ if and only if it lies in $L(\bigoplus \C^{n_i},\C^ n)/{\prod U(n_i)}\times S_{\{n_i\}}((S^V)^{\times m})$. So we obtain a new pushout square
\begin{equation} \label{eq:pushout2} \xymatrix{ (\bigsqcup_{n_1+\hdots+n_m=n}(L(\bigoplus \C^{n_i},\C^ n)/{\prod U(n_i)}))\times_{\Sigma_m} F(S_{\{n_i\}}((S^V)^{\times m})) \ar[r] \ar[d] & sk_{m-1} (\oL_n(V)) \ar[d]\\
  (\bigsqcup_{n_1+\hdots+n_m=n}(L(\bigoplus \C^{n_i},\C^ n)/{\prod U(n_i)}))\times_{\Sigma_m} S_{\{n_i\}}((S^V)^{\times m}) \ar[r] & sk_m(\oL_n(V)).} \end{equation}

Hence it suffices to show:
\begin{Lemma} The left hand vertical map in Diagram (\ref{eq:pushout2}) is a $(U(n)\times G)$-cofibration.
\end{Lemma}
\begin{proof} The proof is very similar to that of Lemma \ref{lem:cofibration1}. Again it suffices to see that \[ \bigsqcup_{\{n_i\neq 0\}_{0\leq i\leq m}\sum n_i=n}(L(\bigoplus \C^{n_i},\C^ n)/{\prod U(n_i)})\] is a $(U(n)\times \Sigma_m)$-CW complex and that the map $F(S_{\{n_i\}}((S^ V)^{\times m})))\to (S^V)^{\times m}$ is a $(\Sigma_m\times G)$-cofibration. The former is easy to see, because each summand is $U(n)$-isomorphic to $U(n)/\prod U(n_i)$ and these summands are permuted by the $\Sigma_m$-action. For the latter we note that by a transformation of variables each $S_{\{n_i\}}((S^V)^{\times m})$ is homeomorphic to the usual unit sphere $S(V\otimes \R^m)$, which -- by Illman's theorem for finite dimensional $V$ and the same trick as in Lemma \ref{lem:cofibration1} for the infinite case -- is a $(\Sigma_m\times G)$-CW complex.
Since $F(S_{\{n_i\}}((S^V)^{\times m}))$ no longer contains any basepoints, it is exactly the subspace of elements with non-trivial $\Sigma_m$-isotropy, and hence always a $(\Sigma_m\times G)$-subcomplex. This finishes the proof.
\end{proof}

\subsection{Verification of cofiber sequence}
\label{app:cofiber}
In this final appendix we give the proof that the map \[ \psi_n: \Sigma^{\infty}_+ (E_{gl}U(n)\times_{U(n)} \oL_n)\to ku^{n-1} \] constructed in Section \ref{sec:pi0ku} makes the following diagram a map of cofiber sequences:
\begin{equation} \label{eq:cof} \xymatrix{ \Sigma^{\infty}_+ (E_{gl}U(n)\times_{U(n)} \oL_n) \ar[r]\ar[d]_{\psi_n} & \Sigma^{\infty}_+ (B_{gl}U(n)) \ar[d]^{\alpha_n}\ar[r] & \Sigma^{\infty} (E_{gl}U(n)_+\wedge_{U(n)} (\oL_n)^\dia)\ar[d]^{\cong}\\
  ku^{n-1}\ar[r]^{i_{n-1}}& ku^n \ar[r] &ku^n/ku^{n-1} } 
\end{equation}
In order to establish this we turn the upper sequence into a strict quotient sequence as well by using the mapping cylinder and by constructing a map \[ \overline{\psi}_n:\Sigma^{\infty}_+ (E_{gl} U(n)\times_{U(n)} ([0,1]\times \oL_n/(0\times \oL_n)))\to ku^n \] with the following three properties:
\begin{enumerate}
 \item The restriction of $\overline{\psi}_n$ to the copy of $\Sigma^{\infty}_+ (E_{gl} U(n)\times_{U(n)} \oL_n)$ at the coordinate $1$ is equal to the composite $i_{n-1} \circ \psi_n$.
 \item The restriction of $\overline{\psi}_n$ to the copy of $\Sigma^{\infty}_+ (B_{gl}U(n))$ at the cone point is equal to $\alpha_n$.
  \item The induced map $\Sigma^{\infty}(E_{gl}U(n)_+\wedge (\oL_n)^\dia)\to ku^n/ku^{n-1}$ (by quotiening out $\Sigma^{\infty}_+ (E_{gl} U(n)\times_{U(n)} \oL_n)$ and $ku^{n-1}$) is homotopic to the isomorphism constructed in Section \ref{sec:kurank}.
\end{enumerate}
The first two properties show that $\overline{\psi}_n$ induces a homotopy between the two composites in the first square of Diagram \ref{eq:cof}. The third property implies that there is a homotopy between the two composites in the square
\[ \xymatrix{ \Sigma^{\infty} (E_{gl}U(n)_+\wedge _{U(n)} (\oL_n)^\dia) \ar[r]\ar[d]_{\cong} & S^1\wedge \Sigma^{\infty}_+ (E_{gl}U(n)\times_{U(n)} \oL_n) \ar[d]^{S^1\wedge \psi_n} \\
  ku^n/ku^{n-1} \ar[r] & S^1\wedge ku^n/ku^{n-1}}
\]
and so we are done. The map $\overline{\psi}_n$ is also used in Section \ref{sec:pi0comp} to determine the effect of the complexity filtration on $\upi_0$.

In order to construct $\overline{\psi}_n$ we quickly recall the objects involved: An element of $E_{gl}U(n)(V)$ is a linear isometry $\C^n\to \Sym(\C\otimes V)$. Points in $\oL_n(V)$ are represented by tuples $(W_i,x_i)_{i\in I}$ where the $x_i$ are elements of $V$ and the $W_i$ are pairwise orthogonal subspaces of $\C^n$ who add up to all of $\C^n$. Furthermore, these tuples have to be reduced and of norm $1$ (cf. Section \ref{sec:kurank}). Finally, elements of $ku^n(V)$ are also represented by tuples $(W_i,x_i)_{i\in I}$, but this time the $W_i$ are orthogonal subspaces of $\Sym(\C\otimes V)$ and the only requirement is that the sum of the dimensions is at most $n$.

We recall also that the definition of $\psi_n$ made use of a function $s:[0,\infty]\to [0,\infty]$ which maps the interval $[0,1/(2n^2)]$ homeomorphically onto $[0,\infty]$ and is constant on the rest. Finally, given a finite tuple of vectors $x=(x_i)_{i\in I}$ of a real inner product space $V$ we defined a map $p_x:V\to \langle \{x_i\}_{i\in I} \rangle\subseteq V$ by $p_x(v)=\sum_I \langle v,x_i \rangle \cdot x_i$.

Now let $H:[0,\infty]\times [0,1]\to [0,\infty]$ be a homotopy relative endpoints from the identity to $s$. Given a real inner product space $V$ with a finite tuple of vectors $x=(x_i)_{i\in I}$ as above, we define a map $H^V_x:S^V\times [0,1]\to S^V$ via
\[ H^V_x(v,t)=(H(|p_xv|,t)-|p_xv|)\cdot p_xv + v. \]
This gives a homotopy from the identity to the map $s^V_x$ used in the definition of $\psi_n$.

Now we can define $\overline{\psi}_n$ by the formula
\[ (\varphi,(W_i,x_i)_{i\in I},t)\wedge v\mapsto \begin{cases}
                 (\varphi(W_i),\frac{t}{1-t} \cdot x_i+v)_{i\in I} & \text{ if } 0\leq t\leq 1/2\\
		 (\varphi(W_i),H^V_x(x_i + v,2t-1))_{i\in I} & \text{ if } 1/2\leq t\leq 1
                                                 \end{cases}
\]
where $x$ is short for the tuple of the $x_i$. Since $H_x(x_i + v,0)$ is equal to $x_i+v$, these two definitions agree on the intersection and glue to a well-defined map. By definition, setting $t$ equal to $1$ gives back $\psi_n$, thus property $(1)$ is satisfied. Furthermore, the elements $(\varphi,(W_i,x_i)_{i\in I},0)$ are mapped to the tuple $(\varphi(W_i),v)_{i\in I}$, which is equal to $(\varphi(\C^n),v)$. Hence it is independent of the $W_i$ and $x_i$ and the induced map $\Sigma^{\infty}_+ (B_{gl}U(n))\to ku^n$ gives back $\alpha_n$, yielding property $(2)$.

It remains to prove property $(3)$, i.e., that the induced map 
\[ \overline{\psi}_n':\Sigma^{\infty}(E_{gl}U(n)_+\wedge \oL_n^\dia)\to ku^n/ku^{n-1} \] obtained by quotiening out $\Sigma^{\infty}_+ (E_{gl} U(n)\times_{U(n)} \oL_n)$ and $ku^{n-1}$ is homotopic to the isomorphism from Section \ref{sec:kurank}. For $t\leq 1/2$ the two maps are in fact equal and hence it suffices to construct a homotopy on the part where $t\geq 1/2$, relative to $t=1/2$. This is achieved by the formula
\[ (\varphi,(W_i,x_i)_{i\in I},t,v,s) \mapsto [(H^V_x((\frac{(1-s)t}{1-(1-s)t}+\frac{s-1}{s+1}+1)\cdot x_i+v,s(2t-1)),\varphi(W_i))_{i\in I}] \]
for $s\in [0,1]$. Continuity is only unclear at points for which $t=1$ and $s=0$, which are mapped to the basepoint. However, by the same estimate as in Section \ref{sec:pi0ku} one sees that the expression $(H^V_x((\frac{(1-s)t}{1-(1-s)t}+\frac{s-1}{s+1}+1)\cdot x_i+v,s(2t-1)),\varphi(W_i))_{i\in I}$ lies in $ku^{n-1}$ already for all $t$ close enough to $1$ and $s$ close enough to $0$. So the homotopy is actually constant around $s=0$ and $t=1$, hence we are done.


\begin{thebibliography}{Seg68b}

\bibitem[AL07]{AL07}
G.~Arone and K.~Lesh.
\newblock Filtered spectra arising from permutative categories.
\newblock {\em Journal f{\"u}r die reine und angewandte {M}athematik ({C}relles
  {J}ournal)}, 2007(604):73--136, 2007.

\bibitem[AL10]{AL10}
G.~Arone and K.~Lesh.
\newblock Augmented {$\Gamma$}-spaces, the stable rank filtration, and a {$bu$}
  analogue of the {W}hitehead conjecture.
\newblock {\em Fund. Math.}, 207(1):29--70, 2010.

\bibitem[CF64]{CF64}
P.~E. Conner and E.~E. Floyd.
\newblock {\em Differentiable periodic maps}.
\newblock Ergebnisse der Mathematik und ihrer Grenzgebiete, N. F., Band 33.
  Academic Press Inc., Publishers, New York; Springer-Verlag,
  Berlin-G\"ottingen-Heidelberg, 1964.

\bibitem[DM14]{DM14}
E.~Dotto and K.~Moi.
\newblock Homotopy {T}heory of {G}-diagrams and equivariant excision.
\newblock {\em arXiv:1403.6101v2}, 2014.

\bibitem[dS03]{dS03}
P.~F. dos Santos.
\newblock A note on the equivariant {D}old-{T}hom theorem.
\newblock {\em J. Pure Appl. Algebra}, 183(1-3):299--312, 2003.

\bibitem[GH08]{GH08}
D.~Gepner and A.~Henriques.
\newblock Homotopy theory of orbispaces.
\newblock {\em arXiv:math/0701916}, 2008.

\bibitem[Gre04]{Gre04}
J.~Greenlees.
\newblock Equivariant connective {K}-theory for compact {L}ie groups.
\newblock {\em Journal of Pure and Applied Algebra}, 187(1):129--152, 2004.

\bibitem[Hau15]{Hau15}
M.~Hausmann.
\newblock Symmetric spectra model global homotopy theory of finite groups.
\newblock {\em arXiv:1509.09270}, 2015.

\bibitem[Ill83]{Ill83}
S.~Illman.
\newblock The equivariant triangulation theorem for actions of compact {L}ie
  groups.
\newblock {\em Mathematische Annalen}, 262(4):487--501, 1983.

\bibitem[JS01]{JS01}
S.~Jackowski and J.~S{\l}omi{\'n}ska.
\newblock {$G$}-functors, {$G$}-posets and homotopy decompositions of
  {$G$}-spaces.
\newblock {\em Fund. Math.}, 169(3):249--287, 2001.

\bibitem[Les00]{Lesh00}
K.~Lesh.
\newblock A filtration of spectra arising from families of subgroups of
  symmetric groups.
\newblock {\em Transactions of the American Mathematical Society},
  352(7):3211--3237, 2000.

\bibitem[Lyd99]{Lyd99}
M.~Lydakis.
\newblock Smash products and {$\Gamma$}-spaces.
\newblock {\em Math. Proc. Cambridge Philos. Soc.}, 126(2):311--328, 1999.

\bibitem[Oli98]{Ol98}
B.~Oliver.
\newblock The representation ring of a compact {L}ie group revisited.
\newblock {\em Comment. Math. Helv.}, 73(3):353--378, 1998.

\bibitem[Ost14]{Ost14}
D.~Ostermayr.
\newblock Equivariant {$\Gamma$}-spaces.
\newblock {\em arXiv: 1404.7626}, 2014.

\bibitem[Rog92]{Rog92}
J.~Rognes.
\newblock A spectrum level rank filtration in algebraic {K}-theory.
\newblock {\em Topology}, 31(4):813--845, 1992.

\bibitem[Sch]{Sch13alg}
S.~Schwede.
\newblock Global algebraic {K}-theory.
\newblock {\em preprint}.

\bibitem[Sch14]{Sch14}
S.~Schwede.
\newblock Equivariant properties of symmetric products.
\newblock {\em arXiv:1403.1355}, 2014.

\bibitem[Sch15]{Sch15}
S.~Schwede.
\newblock Global homotopy theory.
\newblock {\em book project, v0.23, available on the author's homepage}, 2015.

\bibitem[Seg68a]{Seg68b}
G.~Segal.
\newblock Equivariant {$K$}-theory.
\newblock {\em Inst. Hautes \'Etudes Sci. Publ. Math.}, (34):129--151, 1968.

\bibitem[Seg68b]{Seg68}
G.~Segal.
\newblock The representation ring of a compact {L}ie group.
\newblock {\em Publications Mathematiques de l'IHES}, 34(1):113--128, 1968.

\bibitem[Web93]{We93}
P.~Webb.
\newblock Two classifications of simple {M}ackey functors with applications to
  group cohomology and the decomposition of classifying spaces.
\newblock {\em J. Pure Appl. Algebra}, 88(1-3):265--304, 1993.

\end{thebibliography}
\end{document}